\documentclass{compositio}

\usepackage{lmodern}
\usepackage{dsfont}
\usepackage[centertags]{amsmath}
\usepackage[usenames,dvipsnames]{xcolor}
\usepackage[colorlinks=true,linkcolor=Red,urlcolor=NavyBlue,citecolor=blue]{hyperref}
\usepackage{amsfonts}
\usepackage{MnSymbol}
\usepackage{amsthm}
\usepackage{newlfont}
\usepackage{amscd}
\usepackage{amsmath}
\usepackage{enumerate}
\usepackage[all,2cell]{xy}
\usepackage{tikz-cd}
\usetikzlibrary{positioning, calc}
\usepackage{mathrsfs}
\UseAllTwocells
\input xy
\xyoption{2cell}
\xyoption{all}
\usepackage{verbatim}
\usepackage{eucal}
\usepackage{graphicx}

\usepackage{pict2e}

\makeatletter
\newcommand{\adjunction}{\@ifstar\named@adjunction\normal@adjunction}
\newcommand{\normal@adjunction}[4]{%
	#1\colon #2%
	\mathrel{\vcenter{%
			\offinterlineskip\m@th
			\ialign{%
				\hfil$##$\hfil\cr
				\longrightharpoonup\cr
				\noalign{\kern-.3ex}
				\smallbot\cr
				\longleftharpoondown\cr
			}%
	}}%
	#3 \noloc #4%
}
\newcommand{\named@adjunction}[4]{%
	#2%
	\mathrel{\vcenter{%
			\offinterlineskip\m@th
			\ialign{%
				\hfil$##$\hfil\cr
				\scriptstyle#1\cr
				\noalign{\kern.1ex}
				\longrightharpoonup\cr
				\noalign{\kern-.3ex}
				\smallbot\cr
				\longleftharpoondown\cr
				\scriptstyle#4\cr
			}%
	}}%
	#3%
}
\newcommand{\longrightharpoonup}{\relbar\joinrel\rightharpoonup}
\newcommand{\longleftharpoondown}{\leftharpoondown\joinrel\relbar}
\newcommand\noloc{%
	\nobreak
	\mspace{6mu plus 1mu}
	{:}
	\nonscript\mkern-\thinmuskip
	\mathpunct{}
	\mspace{2mu}
}
\newcommand{\smallbot}{%
	\begingroup\setlength\unitlength{.15em}%
	\begin{picture}(1,1)
		\roundcap
		\polyline(0,0)(1,0)
		\polyline(0.5,0)(0.5,1)
	\end{picture}%
	\endgroup
}
\makeatother

\makeatletter
\newcommand*{\doublerightarrow}[2]{\mathrel{
		\settowidth{\@tempdima}{$\scriptstyle#1$}
		\settowidth{\@tempdimb}{$\scriptstyle#2$}
		\ifdim\@tempdimb>\@tempdima \@tempdima=\@tempdimb\fi
		\mathop{\vcenter{
				\offinterlineskip\ialign{\hbox to\dimexpr\@tempdima+1em{##}\cr
					\rightarrowfill\cr\noalign{\kern.5ex}
					\rightarrowfill\cr}}}\limits^{\!#1}_{\!#2}}}
\newcommand*{\triplerightarrow}[1]{\mathrel{
		\settowidth{\@tempdima}{$\scriptstyle#1$}
		\mathop{\vcenter{
				\offinterlineskip\ialign{\hbox to\dimexpr\@tempdima+1em{##}\cr
					\rightarrowfill\cr\noalign{\kern.5ex}
					\rightarrowfill\cr\noalign{\kern.5ex}
					\rightarrowfill\cr}}}\limits^{\!#1}}}
\makeatother

\newcommand{\NN}{\mathbb{N}}

\newcommand{\ZZ}{\mathbb{Z}}
\newcommand{\LL}{\mathbb{L}}
\newcommand{\RR}{\mathbb{R}}

\newcommand{\PP}{\mathbb{P}}

\def\sE{E}
\def\sP{P}

\def\QQ{\mathbb{Q}}
\def\C{\mathcal{C}}
\def\D{\mathcal{D}}
\def\G{\mathcal{G}}

\def\F{\mathcal{F}}
\def\SS{\mathcal{S}}
\def\M{\mathcal{M}}

\def\I{\mathcal{I}}
\def\cS{\mathcal{S}}
\def\T{\mathcal{T}}
\def\P{\mathcal{P}}
\def\Q{\mathcal{Q}}

\def\R{\mathcal{R}}
\renewcommand{\L}{\mathcal{L}}
\def\U{\mathcal{U}}

\def\O{\mathcal{O}}

\def\g{\mathfrak{g}}

\def\tpop{\T_\P \Op(\mathcal{S})}

\def\lrarsimeq{\overset{\simeq}{\lrar}}
\def\lrarsimequ{\underset{\simeq}{\lrar}}

\def\l{\langle}
\def\r{\rangle}

\newcommand{\HSwarrow}{\kern0.05ex\vcenter{\hbox{\Huge\ensuremath{\Swarrow}}}\kern0.05ex}
\newcommand{\hSwarrow}{\kern0.05ex\vcenter{\hbox{\huge\ensuremath{\Swarrow}}}\kern0.05ex}
\newcommand{\LLSwarrow}{\kern0.05ex\vcenter{\hbox{\LARGE\ensuremath{\Swarrow}}}\kern0.05ex}
\newcommand{\LSwarrow}{\kern0.05ex\vcenter{\hbox{\Large\ensuremath{\Swarrow}}}\kern0.05ex}

\newcommand{\HSearrow}{\kern0.05ex\vcenter{\hbox{\Huge\ensuremath{\Searrow}}}\kern0.05ex}
\newcommand{\hSearrow}{\kern0.05ex\vcenter{\hbox{\huge\ensuremath{\Searrow}}}\kern0.05ex}
\newcommand{\LLSearrow}{\kern0.05ex\vcenter{\hbox{\LARGE\ensuremath{\Searrow}}}\kern0.05ex}
\newcommand{\LSearrow}{\kern0.05ex\vcenter{\hbox{\Large\ensuremath{\Searrow}}}\kern0.05ex}

\newcommand{\HDownarrow}{\kern0.05ex\vcenter{\hbox{\Huge\ensuremath{\Downarrow}}}\kern0.05ex}
\newcommand{\hDownarrow}{\kern0.05ex\vcenter{\hbox{\huge\ensuremath{\Downarrow}}}\kern0.05ex}
\newcommand{\LLDownarrow}{\kern0.05ex\vcenter{\hbox{\LARGE\ensuremath{\Downarrow}}}\kern0.05ex}
\newcommand{\LDownarrow}{\kern0.05ex\vcenter{\hbox{\Large\ensuremath{\Downarrow}}}\kern0.05ex}

\newcommand{\HUparrow}{\kern0.05ex\vcenter{\hbox{\Huge\ensuremath{\Uparrow}}}\kern0.05ex}
\newcommand{\hUparrow}{\kern0.05ex\vcenter{\hbox{\huge\ensuremath{\Uparrow}}}\kern0.05ex}
\newcommand{\LLUparrow}{\kern0.05ex\vcenter{\hbox{\LARGE\ensuremath{\Uparrow}}}\kern0.05ex}
\newcommand{\LUparrow}{\kern0.05ex\vcenter{\hbox{\Large\ensuremath{\Uparrow}}}\kern0.05ex}

\newcommand\restr[2]{{
		\left.\kern-\nulldelimiterspace 
		#1 
		\vphantom{\big|} 
		\right|_{#2} 
}}

\newtheorem{thm}{Theorem}[subsection]
\newtheorem{cor}[thm]{Corollary}
\newtheorem{lem}[thm]{Lemma}
\newtheorem{prop}[thm]{Proposition}

\makeatletter
\@addtoreset{thm}{section}
\makeatother

\numberwithin{thm}{subsection}
\numberwithin{equation}{subsection}

\theoremstyle{definition}
\newtheorem{define}[thm]{Definition}

\newtheorem{example}[thm]{Example}

\newtheorem{examples}[thm]{Examples}
\newtheorem{dfn}[thm]{Definition}
\newtheorem{cons}[thm]{Construction}

\newtheorem{notn}[thm]{Notation}
\newtheorem{notns}[thm]{Notations}
\newtheorem{obs}[thm]{Observation}

\newtheorem{obss}[thm]{Observations}
\newtheorem{conv}[thm]{Conventions}

\newtheorem{rem}[thm]{Remark}

\DeclareMathOperator{\hocolim}{hocolim}
\DeclareMathOperator{\colim}{colim}

\DeclareMathOperator{\Ext}{Ext}

\DeclareMathOperator{\Com}{Com}
\DeclareMathOperator{\Ex}{Ex}
\DeclareMathOperator{\Id}{Id}
\DeclareMathOperator{\sD}{D}
\DeclareMathOperator{\id}{id}
\DeclareFontFamily{OT1}{pzc}{}
\DeclareFontShape{OT1}{pzc}{m}{it}{<-> s * [1.10] pzcmi7t}{}
\DeclareMathAlphabet{\mathpzc}{OT1}{pzc}{m}{it}
\DeclareMathOperator{\cof}{cof}

\DeclareMathOperator{\Alg}{Alg}
\DeclareMathOperator{\Ass}{Ass}
\DeclareMathOperator{\Free}{Free}
\DeclareMathOperator{\LMod}{LMod}
\DeclareMathOperator{\Disc}{Disc}

\DeclareMathOperator{\RMod}{RMod}

\DeclareMathOperator{\Pairs}{Pairs}
\DeclareMathOperator{\Rect}{Rect}
\DeclareMathOperator{\diag}{diag}
\DeclareMathOperator{\op}{op}
\DeclareMathOperator{\rE}{E}
\DeclareMathOperator{\Map}{Map}

\DeclareMathOperator{\red}{red}

\DeclareMathOperator{\proj}{proj}

\DeclareMathOperator{\Sets}{Sets}
\DeclareMathOperator{\Top}{Top}

\DeclareMathOperator{\Cat}{Cat}

\DeclareMathOperator{\Mod}{Mod}

\DeclareMathOperator{\Sym}{Sym}
\DeclareMathOperator{\Fun}{Fun}

\DeclareMathOperator{\Un}{Un}
\DeclareMathOperator{\Ob}{Ob}
\DeclareMathOperator{\Sp}{Sp}
\DeclareMathOperator{\Ho}{Ho}

\DeclareMathOperator{\forget}{forgetful}

\DeclareMathOperator{\Op}{Op}
\DeclareMathOperator{\nsOp}{nsOp}
\DeclareMathOperator{\Ab}{Ab}
\DeclareMathOperator{\Env}{Env}

\DeclareMathOperator{\Tw}{Tw}
\DeclareMathOperator{\der}{h}
\DeclareMathOperator{\BMod}{BMod}
\DeclareMathOperator{\aug}{aug}

\DeclareMathOperator{\Conf}{Conf}

\DeclareMathOperator{\h}{h}

\DeclareMathOperator{\sMod}{sMod}

\DeclareMathOperator{\St}{St}
\DeclareMathOperator{\Spectra}{Spectra}

\DeclareMathOperator{\hocofib}{hocofib}
\DeclareMathOperator{\cov}{cov}

\DeclareMathOperator{\Def}{Def}
\DeclareMathOperator{\defi}{def}

\DeclareMathOperator{\IbMod}{IbMod}
\DeclareMathOperator{\Hom}{Hom}
\DeclareMathOperator{\rL}{L}
\DeclareMathOperator{\Fr}{Fr}
\DeclareMathOperator{\rR}{R}
\DeclareMathOperator{\sH}{H}
\DeclareMathOperator{\rP}{P}
\DeclareMathOperator{\sN}{N}
\DeclareMathOperator{\sT}{T}

\DeclareMathOperator{\sCol}{Col}
\DeclareMathOperator{\Seq}{Seq}
\DeclareMathOperator{\Coll}{Coll}
\DeclareMathOperator{\nsColl}{nsColl}
\DeclareMathOperator{\Set}{Set}
\DeclareMathOperator{\sU}{U}
\DeclareMathOperator{\sF}{F}
\DeclareMathOperator{\sG}{G}
\DeclareMathOperator{\E}{E}
\DeclareMathOperator{\HT}{HT}
\DeclareMathOperator{\End}{End}
\DeclareMathOperator{\Fin}{Fin}

\DeclareMathOperator{\Sing}{Sing}

\DeclareMathOperator{\sS}{S}

\def\x{\overset}

\def\Hom{\textrm{Hom}}
\def\End{\textrm{End}}

\newcommand{\tgpd}{\kern0.05ex\vcenter{\hbox{\footnotesize\ensuremath{2}}}\kern0.05ex\mathcal{G}pd} 

\def\rar{\rightarrow}

\def\lrar{\longrightarrow}

\def\ovl{\overline}

\title{Quillen cohomology of enriched operads}

\author{\texttt{Hoang Truong}}

\usepackage{fancyhdr}
\usepackage{lipsum} 
\usepackage{xcolor}

\begin{document}

\begin{abstract} 
A modern insight due to Quillen, which is further developed by Lurie, asserts that many cohomology theories of
interest are particular cases of a single construction, which allows one to define cohomology groups in an abstract
setting using only intrinsic properties of the category (or $\infty$-category) at hand. This universal cohomology theory
is known as Quillen cohomology. In any setting, Quillen cohomology of a given object is classified by its
cotangent complex. The main purpose of this paper is to study Quillen cohomology of operads enriched over a general base category. Our main result provides an explicit formula for computing Quillen cohomology of enriched operads, based on a procedure of taking certain infinitesimal models of their cotangent complexes. Furthermore, we propose a natural construction of the twisted arrow $\infty$-categories of simplicial operads. We then assert that the cotangent complex of a simplicial operad can be represented as a spectrum valued functor on its twisted arrow $\infty$-category.

 When working in stable base categories such as chain complexes and spectra, Francis and Lurie proved the existence of a fiber sequence relating the cotangent complex and Hochschild complex of an $\E_n$-algebra, from which a conjecture of Kontsevich is verified. {We establish an analogous fiber sequence for the operad $\E_n$ itself, in the topological setting.}
\end{abstract}

\maketitle

\tableofcontents

\bigskip

\fontsize{10.5}{13}\selectfont

\medskip

\section{Introduction}\label{s:introduction}

\subsection{An overview of the subject}

A widespread idea in the domain of homotopy theory is to study a given object of interest by associating to it various kinds of cohomology groups. From generalized cohomology theories for spaces and various Ext groups in homological algebra, through group cohomology, sheaf cohomology, Hochschild cohomology and Andr{\'e}-Quillen cohomology, such  invariants vary from fairly useful to utterly indispensable. A modern insight due to Quillen \cite{Quillen}, which is further developed by Lurie \cite{Lurieha}, asserts that all these cohomology theories are particular cases of a single universal construction, which allows one to define cohomology groups in an abstract setting using only intrinsic properties of the category (or $\infty$-category) at hand. This universal cohomology theory is known as \textbf{Quillen cohomology}.

In Quillen's approach, cohomology of an object of interest is classified by its derived abelianization. Suppose we are given a model category $\textbf{M}$. Recall that an \textbf{abelian group object} in $\textbf{M}$ is an object $A$ equipped with two maps $* \lrar A$ and $A\times A \lrar A$ subject to the classical axioms of an abelian group. For each object $X \in \textbf{M}$, the category of abelian group objects in $\textbf{M}_{/X}$, denoted $\Ab(\textbf{M}_{/X})$, possibly inherits a model structure transferred from that of $\textbf{M}$. In this situation, the free-forgetful adjunction $\F :  \adjunction*{}{\textbf{M}_{/X}}{\Ab(\textbf{M}_{/X})}{} : \U $ forms a Quillen adjunction. The \textbf{cotangent complex} of $X$, denoted by $\rL_X$, is then defined to be $\rL_X := \LL \F (X)$ the \textbf{derived abelianization of} $X$. Moreover, given an object $M \in \Ab(\textbf{M}_{/X})$, the $n$'th \textbf{Quillen cohomology group of $X$ with coefficients in} $M$ is computed by the formula 
$$ \sH^{n}_Q(X;M) = \pi_0 \Map^{\h}_{\Ab(\textbf{M}_{/X})} (\rL_X , M[n])$$
where $M[n]$ refers to the $n$-suspension of $M$ in $\Ab(\textbf{M}_{/X})$. The work of Quillen was first devoted to establishing a proper cohomology theory for rings and commutative algebras (cf. \cite{Quillen2}), which is nowadays generalized into the operadic context as described in the following example. 
\begin{example}\label{e:Qopealg1} {Let $\textbf{k}$ be a commutative ring that contains the field $\QQ$ of rational numbers, and let $\P$ be an operad enriched over dg $\textbf{k}$-modules}. For each $\P$-algebra $A$, the free-forgetful adjunction $ \adjunction*{}{(\Alg_\P)_{/A}}{\Ab((\Alg_\P)_{/A})}{}$ is (homotopically) equivalent to the adjunction 
	$$ \Omega^{A}(-) : \adjunction*{}{(\Alg_\P)_{/A}}{\Mod^A_{\P}}{} : A \ltimes (-)$$
	in which $\Mod^A_{\P}$ refers to the category of $A$\textbf{-modules over} $\P$, the left adjoint sends $B \in (\Alg_\P)_{/A}$ to $\Omega^{A}(B)$ the \textbf{module of K{\"a}hler differentials} of $B$ over $A$, while the right adjoint sends $M \in \Mod^A_{\P}$ to $A \ltimes M$ the \textbf{square-zero extension} of $A$ by $M$. Therefore, after sending coefficients into $\Mod^A_{\P}$, the $n$'th Quillen cohomology group of $A$ with coefficients in an object $M \in \Mod^A_{\P}$ is given by
	$$ \sH^{n}_Q(A;M) = \pi_0 \Map^{\h}_{\Mod^A_{\P}} (\Omega^{A}(A^{\cof}) , M[n]) $$
	where $A^{\cof}$ is a cofibrant resolution of $A$ in $\Alg_\P$. (See, e.g., \cite{Loday, Vallette, YonatanCotangent}).
\end{example}  

{Quillen's approach} has certain limitations, despite its success. Indeed, there is not a known criterion assuring the existence of the transferred model structure on abelian group objects in a given model category and moreover, even when realized, this model category structure is not invariant under Quillen equivalences. Improving the work of Quillen, Lurie (\cite{Lurieha}) established the cotangent complex formalism, when working in the $\infty$-categorical framework. In his approach, the notion of derived abelianization is extended to that of \textbf{stabilization}, which itself is inspired by the classical theory of spectra. Let $\C$ be a presentable $\infty$-category and let $X$ be an object of $\C$. Consider the over $\infty$-category  $\C_{/X}$. Conceptually, the stabilization of $\C_{/X}$ is given by the $\infty$-categorical limit of the tower
$$ \cdots  \x{\Omega}{\lrar} (\C_{/X})_* \x{\Omega}{\lrar} (\C_{/X})_* \x{\Omega}{\lrar} (\C_{/X})_* $$
where $\Omega$ refers to the desuspension functor on $(\C_{/X})_*$ the pointed $\infty$-category associated to $\C_{/X}$. As in \cite{Lurieha}, we will refer to the stabilization of $\C_{/X}$ as the \textbf{tangent category to $\C$ at $X$} and denote it by $\T_X\C$. By construction, $\T_X\C$ is automatically a \textbf{stable $\infty$-category}. Moreover, the presentability of $\C$ implies that the canonical functor $\T_X\C \lrar \C_{/X}$ admits a left adjoint called \textbf{suspension spectrum functor} and denoted by $\Sigma^{\infty}_+ : \C_{/X} \lrar \T_X\C$. By this way, the \textbf{cotangent complex of} $X$ is defined  to be $\rL_X := \Sigma^{\infty}_+(X)$. Moreover, the $n$'th Quillen cohomology group of $X$ with coefficients in a given object $M \in \T_X\C$ is now formulated as 
$$ \sH^{n}_Q(X;M)  := \pi_0 \Map_{\T_X\C}(\rL_X , M[n]) .$$
We refer the readers to \cite{YonatanCotangent} for a discussion on the naturality of the evolution from Quillen's approach to Lurie's, and a comparison between them also.

For necessary computations in abstract homotopy theory, model categories seem to be the most favorable environment. Fortunately, the Lurie's setting mentioned above was completely translated into the model categorical language, thanks to the recent works of  Y. Harpaz, J. Nuiten and M. Prasma \cite{YonatanBundle, YonatanCotangent}. Following the setting given in \cite{YonatanBundle}, tangent (model) categories come after a procedure of taking left Bousfield localizations of model categories of interest. Nevertheless, the obstacle is that left Bousfield localizing usually requires the left properness. The recent result of Batanin-White \cite{White} allows one to take left Bousfield localizations of \textbf{semi model categories} (cf., e.g., \cite{Fresse1, Barwick, Spitzweck}) without requiring the left properness. Inspired by this, under our setting, tangent categories exist as semi model categories, {which are basically as convenient as} (full) model categories.

The main purpose of this paper is to formulate Quillen cohomology of operads enriched over a general symmetric monoidal model category, which we will refer to as the \textbf{base category}. Given a base category $\cS$, we denote by $\Op_C(\cS)$ the category of \textbf{$\cS$-enriched $C$-colored operads} with $C$ being some fixed set of colors, yet the one we really care about is the category of  \textbf{$\cS$-enriched operads} (with non-fixed sets of colors), which will be denoted by $\Op(\cS)$. Under some suitable conditions, $\Op(\cS)$ admits the \textbf{canonical model structure}, according to Caviglia's \cite{Caviglia}. In particular, when $\cS$ is the category $\Set_\Delta$ of simplicial sets equipped with the standard model structure, the canonical model structure on $\Op(\Set_\Delta)$ agrees with the \textbf{Cisinski-Moerdijk model structure}, which was known to be a model for the theory of $\infty$\textbf{-operads} (cf. \cite{Cisinski}). 

Given an $\cS$-enriched $C$-colored operad $\P$, one can consider it  either as an object of $\Op_C(\SS)$ or as an object of $\Op(\cS)$. As mentioned, our main focus is on the latter case. Therefore, we will refer to the Quillen cohomology of $\P\in\Op(\mathcal{S})$ as its \textbf{(proper) Quillen cohomology}. In contrast, the corresponding cohomology of $\P\in\Op_C(\mathcal{S})$ is referred to as its \textbf{reduced Quillen cohomology}. Some attention was given in the literature to the reduced Quillen cohomology of operads. For instance, in the context of dg operads, this was studied by Loday-Merkulov-Vallette \cite{Loday, Vallette}, (in which the resultant is described in terms of derivations, similarly as in Example \ref{e:Qopealg1}). On the other hand, the problems of formulating Quillen cohomology of operads and investigating its applications have not yet been considered so far.

It has been widely acknowledged that Quillen cohomology theory keeps a key role in the study of \textbf{deformation theory} and \textbf{obstruction theory}. Let us discuss these in what follows.

Naively speaking, a deformation of an object of interest under $``$small perturbation'' is an object of the same type which is $``$equivalent'' to the original object. In our setting, a small perturbation is precisely an \textbf{artinian dg} $\textbf{k}$-\textbf{algebra} with $\textbf{k}$ being a field of characteristic $0$, i.e. a (connective) augmented commutative dg $\textbf{k}$-algebra $R$ of finite dimension such that the augmentation map $R \rar \textbf{k}$ exhibits the $0$'th homology of $R$ as a local $\textbf{k}$-algebra. For instance, in the context of algebraic objects (e.g., dg modules, dg categories and dg operads), a deformation of an object $X$ under a small perturbation $R \rar \textbf{k}$ is by definition an object $Y$ (parameterized over $R$) coming together with a weak equivalence $Y\otimes_R\textbf{k} \lrarsimeq X$. One can then organize all the deformations of $X$ into a category such that every morphism is an equivalence  of deformations. The $\infty$-groupoid associated to this category will be called the \textbf{space of deformations of $X$ over} $R$, denoted by $\Def(X,R)$. In a somewhat more abstract setting presented in a parallel paper (joint with Y. Harpaz), we propose the notion of spaces of deformations for various types of object. We then show that Quillen cohomology of a given object indeed classifies the homotopy type of its spaces of deformations. Moreover, we show that the functor $R \mapsto \Def(X,R)$   forms a \textbf{formal moduli problem} in the sense of \cite{Luriefmp}. It implies that the deformations of $X$ are $``$governed'' by a single dg Lie algebra (cf. \cite{Luriefmp, Pridham}).

For two topological spaces $X$ and $Y$, understanding maps from $X$ to $Y$ (up to homotopy) is a classical problem in homotopy theory. Suppose that $Y$ is simply connected. As the first step, one filters $Y$ by its \textbf{Postnikov tower}:
$$  \cdots \lrar P_n(Y) \lrar P_{n-1}(Y) \lrar \cdots \lrar  P_1(Y) \lrar  P_0(Y)  .$$
The problem is therefore reduced to understanding maps $X \lrar P_n(Y)$ that extend some given map $f : X \lrar P_{n-1}(Y)$. It is known that the obstruction to a section $P_{n-1}(Y) \lrar P_n(Y)$ is classified by a single cohomology class  $k_{n-1} \in \sH^{n+1}(P_{n-1}(Y) ; \pi_n Y)$. More generally, the obstruction to an extension $X \lrar P_n(Y)$ of $f$ is classified by the image of $k_{n-1}$ under the induced map $$f^{*} :  \sH^{n+1}(P_{n-1}(Y) ; \pi_n Y) \lrar  \sH^{n+1}(X ; \pi_n Y) .$$ The latter is called the \textbf{obstruction class of $f$}. Note that the ordinary cohomology of spaces is nothing but a particular case of Quillen cohomology. The authors of \cite{DKSmith} generalized the above machinery to study the obstruction theory of \textbf{simplicial categories} with fixed set of objects, in which case obstruction classes are indeed contained in Quillen cohomology groups. The obstruction theory of \textbf{dg categories} was considered in \cite{Tabuada1}, in which case Quillen cohomology plays a central role again.  We believe that the present paper may open the way to the study of obstruction theory of simplicial operads.

\subsection{Main statements}

We now present a summary of our main results together with some historical background. We would like to divide this subsection into three parts.

\subsubsection{\textbf{Operadic tangent categories and cotangent complex of enriched operads.}}

 Suppose we are given a sufficiently nice base category $\cS$ (cf. Conventions \ref{convention}). As the starting point, we extend a result of \cite{YonatanCotangent}, which we now recall. Let $\C \in \Cat(\cS)$ be an $\cS$-enriched category. {We} denote by $\Cat_C(\cS)$ the category of $\cS$-enriched categories {that admit $C := \Ob(\C)$ as the fixed set of objects, and denote} by $\BMod(\C)$ the category of $\C$\textbf{-bimodules}. There is a sequence of the obvious Quillen adjunctions $$\adjunction*{}{\BMod(\C)_{\C/}}{\Cat_C(\cS)_{\C/}}{} \adjunction*{}{}{\Cat(\cS)_{\C/}}{} .$$

\begin{thm}\label{t:cattangent} (Y. Harpaz, J. Nuiten and  M. Prasma \cite{YonatanCotangent}) The above sequence induces a sequence of Quillen equivalences connecting the associated tangent categories:
	\begin{equation}\label{eq:QseqYH}
		\adjunction*{\simeq}{\T_{\C}\BMod(\C)}{\T_{\C}\Cat_C(\cS)}{} \adjunction*{\simeq}{}{\T_{\C}\Cat(\cS)}{} 
	\end{equation}
\end{thm}

We now fix $\P$ to be a {$\Sigma$-cofibrant} $C$-colored operad in $\cS$. We {will} let $\BMod(\P)$ and $\IbMod(\P)$ respectively denote the categories of $\P$\textbf{-bimodules} and \textbf{infinitesimal $\P$-bimodules}. Note that the structure of bimodules over an operad is not linear in general. The notion of infinitesimal bimodules comes naturally as the linearization of that usual notion of operadic bimodules, and in particular, is much more convenient for computations. Our creed is to exhibit the Quillen cohomology groups of an operad in terms of infinitesimal bimodules over it. We start with observing that there is a sequence of Quillen adjunctions
$$ \adjunction*{}{\IbMod(\P)_{\mathcal{P}/}}{\BMod(\P)_{\mathcal{P}/}}{} \adjunction*{}{}{\Op_C(\mathcal{S})_{\mathcal{P}/}}{} \adjunction*{}{}{\Op(\mathcal{S})_{\mathcal{P}/}}{} $$
formed by the obvious induction-restriction adjunctions.

\begin{thm}\label{t:mainoptangentsum} (\ref{t:mainoptangent}, \ref{t:mainoptangent'}) The above sequence induces a sequence of Quillen equivalences connecting the associated tangent categories:
	$$\adjunction*{\simeq}{\T_\P\IbMod(\P)}{\T_\P\BMod(\P)}{} \adjunction*{\simeq}{}{\T_\P\Op_C(\mathcal{S})}{} \adjunction*{\simeq}{}{\T_\P\Op(\mathcal{S})}{}.$$
	Moreover, when $\cS$ is in addition stable containing a strict zero object, all the categories above are Quillen equivalent to $\IbMod(\P)$.
	
\end{thm}

To give a convenient formula for the Quillen cohomology of $\P$, we describe the derived image of the cotangent complex $\rL_\P \in \T_\P\Op(\mathcal{S})$ under the right Quillen equivalence $\T_\P\Op(\mathcal{S}) \lrarsimeq \T_\P\IbMod(\P)$. To this end, we first prove the following key proposition, which extends  [\cite{YonatanCotangent}, Proposition 3.2.1]. {We will assume further that $\P$ is cofibrant, and is good in the sense of Definition \ref{S9}.}

\begin{prop}\label{cotantgentcplxOperadandBimodulesum}(\ref{cotantgentcplxOperadandBimodule}) Under the Quillen equivalence $\T_\P \Op(\mathcal{S}) \simeq \T_\P \BMod(\mathcal{P})$, the cotangent complex $\rL_{\mathcal{P}}$ is identified with $\rL^{b}_{\mathcal{P}}[-1]$ where $\rL^{b}_{\mathcal{P}}$ signifies the cotangent complex of $\P$ considered as a bimodule over itself.
\end{prop}

The object $\rL^{b}_{\mathcal{P}}$ can be well calculated, so that we can get a nice description of its derived image in $\T_\P\IbMod(\P)$. We {will} denote the latter by $\widetilde{\rL}_\P$. Furthermore, we denote by $\sS^{n}$ the \textbf{pointed $n$-sphere} in $\cS$, and denote by $\sS_C^{n}$ the $C$-collection which has $\sS_C^{n}(c;c)=\sS^{n}$ for every $c\in C$ and agrees with $\emptyset_\cS$ on the other levels. 

\begin{thm}\label{t:qcohomenriched}(\ref{conclusionCotangentCplx}, \ref{c:quillencohom}) The $n$'th Quillen cohomology group of $\P$ with coefficients in a given object $M\in \T_\P\IbMod(\P)$ is computed by the formula
	$$ \sH^{n}_Q(\P;M) \cong \pi_0\Map^{\h}_{\T_\P\IbMod(\P)}(\widetilde{\rL}_\P,M[n+1])$$
	in which $\widetilde{\rL}_\P \in \T_\P\IbMod(\P)$ is the prespectrum with $(\widetilde{\rL}_\P)_{n,n} = \P\circ \sS_C^{n}$ for $n\geqslant0$. If $\cS$ is in addition stable and contains a strict zero object $0$ then the $n$'th Quillen cohomology group of $\P$ with coefficients in an object $M\in \IbMod(\P)$ is computed by
	$$ \sH^{n}_Q(\P;M) \cong  \pi_0\Map^{\h}_{\IbMod(\P)}(\ovl{\rL}_\P,M[n+1])$$
	where $\ovl{\rL}_\P \in \IbMod(\mathcal{P})$ is given at each $C$-sequence $\overline{c} := (c_1,\cdots,c_m;c)$ by $$\ovl{\rL}_\P(\overline{c}) = \P(\ovl{c})\otimes \hocolim_n\Omega^{n} [\, (\sS^{n})^{\otimes m}  \times^{\h}_{1_\mathcal{S}}  0\,].$$ 
\end{thm}

Besides that, we find a connection between Quillen cohomology and reduced Quillen cohomology of $\P$, expressed as follows.

\begin{thm}\label{c:longexseq1sum} (\ref{c:longexseq1}) Given an object $M\in \T_\P\IbMod(\P)$, there is  a long exact sequence of abelian groups of the form
	$$ \cdots \longrightarrow \sH^{n-1}_Q (\mathcal{P};M) \lrar  \sH^{n}_{Q,r} (\mathcal{P};M) \lrar \sH^{n}_{Q,\red} (\mathcal{P};M) \lrar \sH^{n}_Q (\mathcal{P};M) \lrar \sH^{n+1}_{Q,r} (\mathcal{P};M)  \lrar \cdots $$
	where $\sH^{\bullet}_{Q,r} (\mathcal{P};-)$ refers to Quillen cohomology group of $\P$ when regarded as a right module over itself and $\sH^{\bullet}_{Q,\red} (\mathcal{P};-)$ refers to reduced Quillen cohomology group of $\P$.
\end{thm}

\subsubsection{\textbf{Twisted arrow $\infty$-categories and the cotangent complex of simplicial operads.}}

{Our result concerning the cotangent complex of simplicial operads (or $\infty$-operads) will provide a generalization for the following.}

\begin{thm}\label{t:Qcohominftycat} (Y. Harpaz, J. Nuiten  and M. Prasma \cite{YonatanCotangent}) Let $\C$ be a fibrant simplicial category. There is an equivalence of $\infty$-categories $$\T_\C\Cat(\Set_\Delta)_\infty \simeq \Fun(\Tw(\C) , \Sp) $$ with $\Sp$ being the $\infty$-category of spectra and $\Tw(\C)$ being the \textbf{twisted arrow} $\infty$-\textbf{category} of $\C$. Furthermore, the cotangent complex $\rL_\C\in\T_\C\Cat(\Set_\Delta) $ is then identified {with} the constant functor $\Tw(\C) \lrar \Sp$ on $\mathbb{S}[-1]$, i.e., the desuspension of the sphere spectrum. Consequently, the $n$'th Quillen cohomology group of $\C$ with coefficients in a functor $\F:\Tw(\C) \lrar \Sp$ is given by $$ \sH^{n}_Q(\C ; \F) = \pi_{-n-1} \lim\F.$$
\end{thm}

The construction of twisted arrow $\infty$-categories (of $\infty$-categories) $$\Tw(-) : \Cat_\infty \lrar \Cat_\infty $$ was originally introduced by Lurie [\cite{Lurieha}, $\S$5.2]. We extend that to the construction of \textbf{twisted arrow} $\infty$-\textbf{categories of simplicial operads}. Let $\P$ be a fibrant simplicial operad. Conceptually, the twisted arrow $\infty$-category $\Tw(\P)$ is defined to be the \textbf{covariant unstraightening} of the simplicial functor $\P : \textbf{Ib}^{\mathcal{P}} \lrar \Set_\Delta$, which encodes the data of $\P$ as an infinitesimal bimodule over itself (see $\S$\ref{s:operadicmodules}). In particular, objects of $\Tw(\P)$ are precisely the \textit{operations} of $\P$ (i.e., the vertices of spaces of operations of $\P$), while the morphisms of $\Tw(\P)$ are given by the factorizations of operations. For examples, $\Tw(\Com)$ is equivalent to $\Fin_*^{\op}$ the (opposite) category of finite pointed sets, while $\Tw(\Ass)$ is equivalent to the simplex category $\Delta$.  Moreover, we also give a model for the twisted arrow $\infty$-category of the \textbf{little $n$-discs operad} $\E_n$ in {$\S$\ref{s:twoperad}}.

\begin{thm}\label{c:mainsum} (\ref{c:main}) Let $\P$ be a fibrant and $\Sigma$-cofibrant simplicial operad. Then there is an equivalence of $\infty$-categories
	$$ \T_\P\Op(\Set_\Delta)_\infty \lrarsimeq \Fun(\Tw(\P) , \Sp) .$$
	Moreover, under this equivalence, the cotangent complex $\rL_\P \in \T_\P\Op(\Set_\Delta)_\infty$ is identified {with} the desuspension of the functor $ \F_\P : \Tw(\P)  \lrar \Sp$ given on objects by sending each operation $\mu\in\P$ of arity $m$  to $\F_\P(\mu) = \mathbb{S}^{\oplus m}$ the $m$-fold coproduct of the sphere spectrum. Consequently, the $n$'th Quillen cohomology group of $\P$ with coefficients in a given functor $\F : \Tw(\P) \lrar \Sp$ is given by the formula
	$$ \sH^{n}_Q (\P ; \F) = \pi_{0}\Map_{\Fun(\Tw(\P) , \Sp)} (\F_\P , \F[n+1]).$$
\end{thm}

For example, we consider the case where $\P=\E_\infty$. Denote by $\Mod_{\textbf{k}}$ the category of modules over a given commutative ring $\textbf{k}$. A \textbf{right $\Gamma$-module} is by definition a functor $T : \Fin_*^{\op} \lrar \Mod_{\textbf{k}}$. There is a significant invariant for right $\Gamma$-modules given by the \textbf{stable cohomotopy groups} (cf. \cite{Pirashvili}), which plays a key role in the $\E_\infty$-\textbf{obstruction theory} initiated by Robinson \cite{Robinson}. In Proposition \ref{p:stablecohomotopy}, we show that for each integer $n$ there is an isomorphism 
$$ \pi^n(T) \cong  \sH^{n-1}_Q (\E_\infty ; \widetilde{T})$$
between the $n$'th stable cohomotopy group of $T$ and the $(n-1)$'th Quillen cohomology group of $\E_\infty$ with coefficients in the induced functor $\widetilde{T} : \Tw(\E_\infty) \simeq \sN(\Fin_*^{\op}) \lrar \Sp$. This observation reaffirms the fact that the obstruction for an \textbf{$\E_\infty$-structure} is indeed controlled by the Quillen cohomology groups of $\E_\infty$ (cf. [\cite{Robinson}, Theorem 5]). 

\subsubsection{\textbf{Quillen principle and the cotangent complex of $\E_n$-operads.}}

In the context of algebraic objects, Quillen cohomology has attracted a lot of attention in literature because of its central role in the deformation theory. A traditional method for computing the \textbf{deformation complex} (or \textbf{tangent complex}) of a given object is to relate the cotangent complex to a \textbf{Hochschild-type complex} of that object via a nice enough fiber sequence. When such a fiber sequence exists, we will say that the object is subject to the \textbf{Quillen principle}. Let us recall some fundamental results based on that kind of thinking.  

{Let $\textbf{k}$ be a commutative ring that contains the field $\QQ$.}

\begin{example}  At the beginning, for $A$ a cofibrant $\textbf{k}$-algebra, Quillen \cite{Quillen2} showed that there is a fiber sequence of $A$-bimodules of the form
	\begin{equation}\label{eq:Qprin1}
		\rL_A \lrar A\otimes A \lrar A
	\end{equation}  
	in which $A\otimes A$ represents the free $A$-bimodule generated by $\textbf{k}$ and the second map is induced by the unit map $\textbf{k}\lrar A$. This tells us that the deformation complex of $A$ can be described via the Hochschild cohomology of $A$, which is really accessible using the bar resolution.
\end{example}

\begin{example} Lie algebras are subject to the Quillen principle as well. Indeed, for $\g$ a {cofibrant} Lie algebra over $\textbf{k}$, there is a fiber sequence of (left) $\g$-modules:
	\begin{equation}\label{eq:Qprin2}
		\rL_\g \lrar \sU(\g) \lrar \textbf{k} 
	\end{equation}
	in which $\sU(\g)$ is the universal enveloping algebra (regarded as the free $\g$-module generated by $\textbf{k}$) and $\textbf{k}$ is equipped with the trivial $\g$-action. {Due to the above fiber sequence, we may describe the deformation complex of $\g$ using the usual \textbf{Lie cohomology} of $\g$, which can be viewed as a variant of the Hochschild-cohomology formalism.}
\end{example}

\begin{example} In fact, an analogue was observed for discrete groups even earlier. For a group $G$, one paid attention to the \textbf{augmentation ideal} $I_G$ fitting into the following exact sequence of $G$-modules:  
	\begin{equation}\label{eq:Qprin3}
		I_G \lrar \mathbb{Z}[G] \lrar \mathbb{Z}
	\end{equation}
	where $\mathbb{Z}$ is equipped with the trivial $G$-action. It turns out that $I_G$ is a model for the \textbf{Beck module} of $G$, which is nothing but the (underived) cotangent complex of $G$. A generalized version of \eqref{eq:Qprin3} for \textbf{simplicial groups} can be found in \cite{YonatanCotangent}.
\end{example}

Furthermore, Kontsevich \cite{Ko} predicted the existence of a fiber sequence relating the tangent complex and the Hochschild cohomology complex of a dg $\E_n$-algebra (up to a shift by $n$). A generalization of this conjecture was proved independently by Francis [\cite{Francis}, Theorem 2.26] and Lurie [\cite{Lurieha}, Theorem 7.3.5.1], when working in suitable stable base categories. Let $\C$ be a stable presentable symmetric monoidal $\infty$-category whose monoidal product distributes over colimits, and let $A$ be an $\E_n$-algebra in $\C$. Denote by $\Mod_A^{\E_n}(\C)$ the $\infty$-category of $A$-modules over $\E_n$. The authors first show that there is a composed equivalence of $\infty$-categories:
	\begin{equation}\label{eq:composedeq}
	\T_A\Alg_{\E_n}(\C) \lrarsimeq \T_A\Mod_A^{\E_n}(\C) \x{\Psi_A}{\lrarsimequ} \Mod_A^{\E_n}(\C).
\end{equation}
We use the same notation $\rL_A$ to denote the cotangent complex of $A\in\Alg_{\E_n}(\C)$ and also, its images in $\T_A\Mod_A^{\E_n}(\C)$ and $\Mod_A^{\E_n}(\C)$. The authors verify a higher version of \eqref{eq:Qprin1}, meaning that the Quillen principle holds for $\E_n$-algebras, as indicated below.

\begin{thm}(Francis \cite{Francis}, Lurie \cite{Lurieha}) There is a fiber sequence in $\Mod_A^{\E_n}(\C)$ of the form 
	\begin{equation}\label{eq:cofibEnalg}
		\Free(\textbf{1}) \x{\eta_A}{\lrar} A \lrar \rL_A[n]
	\end{equation}
	where $\Free$ denotes the free functor $\C \lrar \Mod_A^{\E_n}(\C)$ and $\textbf{1}$ refers to the unit of $A$.
\end{thm}

This indeed gives a solution for Kontsevich's conjecture because the middle term of \eqref{eq:cofibEnalg} represents the Hochschild complex of $A$ itself.

\begin{rem}\label{rem:un} It is natural to ask how an unstable-analogue of \eqref{eq:cofibEnalg} (i.e. when the base category is no longer stable) can be formulated. Let us first discuss the sequence \eqref{eq:cofibEnalg} further. Note that under the equivalence $\Mod_A^{\E_n}(\C) \simeq \T_A\Mod_A^{\E_n}(\C)$, the map $\eta_A : \Free(\textbf{1}) \lrar A$ is identified with the map
	$$ \Sigma^\infty_+(\eta_A) \lrar \rL^m_A $$
where $\Sigma^\infty_+(\eta_A)$ is the image of $\eta_A$ under the functor $\Sigma^\infty_+: \Mod_A^{\E_n}(\C)_{/A} \lrar \T_A\Mod_A^{\E_n}(\C)$ and  $\rL^m_A$ denotes the cotangent complex of $A$ regarded as a module over itself. So the sequence \eqref{eq:cofibEnalg} is identified with {a} fiber sequence in $\T_A\Mod_A^{\E_n}(\C)$ {of the form}
\begin{equation}\label{eq:cofibEnalggen}
	\Sigma^\infty_+(\eta_A) \lrar \rL^m_A \lrar \rL_A[n]
\end{equation}
Moreover, note that in the composed functor \eqref{eq:composedeq}, the first functor is an equivalence even in general base categories, whilst $\Psi_A$ is one due to the stability of $\C$. Therefore, we can say that an unstable-analogue of \eqref{eq:cofibEnalg}, if any, should take the form \eqref{eq:cofibEnalggen}. Alternatively, it means  $\rL_A[n]$ is equivalent to $\rL_{A/\Free(\textbf{1})}$ the \textbf{relative cotangent complex} of $\eta_A$ (see Definition \ref{d:relcotangent}).
\end{rem}

{A proof of the existence of \eqref{eq:cofibEnalggen} for topological $\E_n$-algebras can now be found in \cite{Hoang1}, where we consider not only the cotangent complex of such objects but also algebras over a general enriched operad.} In the present paper, it is more convenient for us to concentrate solely on the cotangent complex of topological $\E_n$-operads.

We assert that the cotangent complex of the operad  $\E_n$ is also subject to the Quillen principle. We denote by $\mu_0\in\E_n(0)$ the unique nullary operation of $\E_n$ regarded as an object of $\Tw(\E_n)$. Furthermore, denote by $\eta_!(\mathbb{S})$ the left Kan extension of the map $\{\mathbb{S}\} \lrar \Sp$ along the inclusion $\eta :\{\mu_0\} \lrar \Tw(\E_n)$.
\begin{thm}\label{p:cotantEnsum} (\ref{p:cotantEn}, \ref{c:cotantEn})  There is a canonical fiber sequence in $\Fun(\Tw(\E_n),\Sp)$ of the form
	\begin{equation}\label{eq:cotantEnsum}
		\eta_!(\mathbb{S}) \lrar \ovl{\mathbb{S}} \lrar \F_{\E_n}[n],
	\end{equation}
	where $\ovl{\mathbb{S}}$ denotes the constant functor on the sphere spectrum. Consequently, for a given functor $\F : \Tw(\E_n) \lrar \Sp$, there is a long exact sequence of abelian groups of the form
	$$ \cdots \lrar \sH^{-k-n-1}_Q(\E_n ; \F) \lrar \pi_k\lim\F \lrar \pi_k\F(\mu_0) \lrar \sH^{-k-n}_Q(\E_n ; \F) \lrar \pi_{k-1}\lim\F \lrar \cdots .$$
	In particular, when $\F(\mu_0)=0$ then for every $k\in\ZZ$, we have a canonical isomorphism
	$$ \sH^{-k-n}_Q(\E_n ; \F) \lrarsimeq \pi_{k-1}\lim\F .$$
\end{thm}

\begin{rem}\label{r:QprinEn1} By Theorem \ref{c:mainsum}, the sequence \eqref{eq:cotantEnsum} {can be equivalently rewritten as}
	\begin{equation}\label{eq:cotantEnsum1}
		\eta_!(\mathbb{S}) \lrar \ovl{\mathbb{S}} \lrar \rL_{\E_n}[n+1]
	\end{equation}
{Here we use the same notation for the image of the cotangent complex $\rL_{\E_n} \in \T_\P\Op(\Set_\Delta)_\infty$ in $\Fun(\Tw(\E_n),\Sp)$}. We denote by $\Free^{ib}(\mu_0)$ the free infinitesimal $\E_n$-bimodule generated by a single point (concentrated in level $0$), {and write $\eta_{\E_n} : \Free^{ib}(\mu_0) \lrar \E_n$ for the map} classified by $\mu_0 \in \E_n(0)$. We can show that under the equivalence of $\infty$-categories:
	 $$\Fun(\Tw(\E_n),\Sp)\simeq\T_{\E_n}\IbMod(\E_n)_\infty,$$ 
	 the map $\eta_!(\mathbb{S}) \lrar \ovl{\mathbb{S}}$ is identified with the {canonical} map $\Sigma^\infty_+(\eta_{\E_n}) \lrar \rL^{ib}_{\E_n}$ where $\rL^{ib}_{\E_n}$ denotes the cotangent complex of $\E_n$ regarded as an infinitesimal bimodule over itself. Thus, according to the logic as in Remark \ref{rem:un}, we will refer to \eqref{eq:cotantEnsum} as the \textbf{Quillen principle for topological $\E_n$-operads}.
\end{rem}

 \begin{rem}\label{r:QprinEn2} The Quillen principle for $\E_n$-operads in stable base categories, such as chain complexes over a commutative ring and symmetric spectra, becomes even more outstanding. Namely, in this case we may obtain a fiber sequence of infinitesimal $\E_n$-bimodules {of the form}
 	\begin{equation}\label{eq:cofibEnstab}
 	 \Free^{ib}(\mu_0) \x{\eta_{\E_n}}{\lrar} \E_n \lrar  \rL_{\E_n}[n+1]
 	\end{equation}
{Here, we maintain the same notation for the image of $\rL_{\E_n} \in \T_{\E_n}\Op(\mathcal{S})$} under the equivalence $\T_{\E_n}\Op(\mathcal{S}) \simeq \IbMod(\E_n)$ (see Theorem \ref{t:mainoptangentsum}). This verifies the operadic-analogue of Kontsevich's conjecture because the middle term of \eqref{eq:cofibEnstab} represents the Hochschild complex of $\E_n$ itself. On a different aspect, applying the functor $(-)\circ_{\E_n} A$ (cf. Example \ref{ex:mainTcomex2}) to \eqref{eq:cofibEnstab}, we may recover the fiber sequence \eqref{eq:cofibEnalg} of Francis-Lurie. This reveals that Quillen cohomology of algebras over an operad {can be} controlled by Quillen cohomology of the operad itself. We will address these in a subsequent paper.
 \end{rem}

\subsection{A preview on $\sT$-cohomology}

We are giving a brief preview of what we would like to develop in future work. {Let us first discuss} the construction $\Tw(-)$ a little more. We let $\textbf{S}:=(\Set_\Delta)_\infty$ denote the \textbf{$\infty$-category of spaces}. Let $\P$ be a fibrant simplicial operad ({which, for simplicity, has a single color}). For an infinitesimal $\P$-bimodule $M$, we can associate to it a functor $M^\star : \Tw(\P) \lrar \textbf{S}$ taking an operation $\mu \in \P(k)$ to the space $M(k)$. As discussed in Remark \ref{r:mainTcom}, we have a weak equivalence of spaces:
\begin{equation}\label{eq:mainTcom}
	\lim_{\Tw(\P)}M^\star \lrarsimeq \Map^{\h}_{\IbMod(\P)}(\P,M)
\end{equation}

The story becomes more interesting when examining some special cases of $\P$ and $M$. For the case of the little discs operads, it will be noteworthy that there is an $\infty$-categorical filtration of the nerve of $\Fin_*^{\op}$:
$$ \sN(\Delta) \simeq \Tw(\E_1) \lrar \Tw(\E_2) \lrar \cdots \lrar \Tw(\E_\infty) \simeq \sN(\Fin_*^{\op}) .$$

\begin{example}\label{ex:mainTcomex1} Let $f : \P \lrar \Q$ be a map of single-colored simplicial operads. Considering $\Q$ as an infinitesimal $\P$-bimodule with the structure induced by $f$, we obtain a functor $\Q^\star : \Tw(\P) \lrar \textbf{S}$. In particular, for the case where $\P=\E_1$, then $\Q^\star$ is identified {with} a cosimplicial space, that was first considered by McClure-Smith \cite{ClureSmith}, and developed further in \cite{Hess} with the aim of studying {the space of maps $\E_1 \lrar \Q$}. More generally, {for the case in which $\P=\E_n$ for some $n<\infty$ and $\Q$ is such that $\Q(0)\simeq \Q(1) \simeq *$, then} according to [\cite{DuVic}, Theorem 1.10], we have weak equivalences
	$$ \lim_{\Tw(\E_n)}\Q^\star \, \simeq \, \Map^{\h}_{\IbMod(\E_n)}(\E_n,\Q) \, \simeq \, \Omega^{n+1}\Map^{\h}_{\Op_*}(\E_n,\Q)$$ 
	where $\Map^{\h}_{\Op_*}(-,-)$ refers to the derived mapping space between simplicial single-colored operads, {and the $(n+1)$-fold loop space is taken at the base point $f$.}
\end{example}

{Let us assume further that $\P$ is $\Sigma$-cofibrant.}

\begin{example}\label{ex:mainTcomex2} Let $A\in \Alg_\P(\Set_\Delta)$ be a $\P$-algebra. For any $A$-module $N$, the \textbf{endomorphism $\Sigma_*$-object} $\End_{A,N}$ {is defined by, for each $k\geq0$, letting $$\End_{A,N}(k) := \Map_{\Set_\Delta}(A^{\times \, k},N).$$}
	Moreover, this construction determines the right adjoint of a Quillen adjunction:
	$$ \adjunction{(-)\circ_\P A \, }{ \, \IbMod(\P)}{\Mod_\P^A}{\, \End_{A,-}}  $$
	in which the left adjoint sends each $M\in\IbMod(\P)$ to the relative composite product $M\circ_\P A$. (Notice
	that in a general base category, it is necessary to require that $A$ has a cofibrant underlying object). {Thus
		we have a chain of weak equivalences}
	$$ \lim_{\Tw(\P)}\End_{A,N}^\star \, \simeq \, \Map^{\h}_{\IbMod(\P)}(\P,\End_{A,N}) \, \simeq \, \Map^{\h}_{\Mod_\P^A}(A,N), $$
{provided further that $A\in \Alg_\P(\Set_\Delta)$ is cofibrant and $N\in\Mod_\P^A$ is fibrant.}
\end{example}

Now let $\cS$ be a symmetric monoidal model category equipped with a {weak monoidal Quillen adjunction} $$\adjunction{\mathbb{F} \, }{ \, \Set_\Delta}{\cS}{\, \mathbb{U}}$$
{(cf. \cite{Schwede})}. The cases of interest will be $\Sp^{\Sigma}$ the category of symmetric spectra, the category $\C(\textbf{k})$ of chain complexes over a commutative ring $\textbf{k}$, and $\Set_\Delta$ itself. For an infinitesimal $\mathbb{F}\P$-bimodule $M$, {by adjunction we obtain a weak equivalence of the form}
\begin{equation}\label{eq:mainTcom1}
	\lim_{\Tw(\P)}(\mathbb{U}M)^\star \lrarsimeq \Map^{\h}_{\IbMod(\mathbb{F}\P)}(\mathbb{F}\P,M)
\end{equation}
Here the images through  $\mathbb{F}$ and $\mathbb{U}$ are assumed to be in the right homotopy type already.

\begin{example}\label{ex:mainTcomex3} Let $A$ be a {cofibrant} $\mathbb{F}\P$-algebra and let $N$ be a {fibrant} $A$-module. Picking up the weak equivalences from Example \ref{ex:mainTcomex2}, we obtain a weak equivalence of the form  $$\lim_{\Tw(\P)}(\mathbb{U}\End_{A,N})^\star \, \simeq \, \Map^{\h}_{\Mod_{\mathbb{F}\P}^A}(A,N).$$
	Due to this, when $\P = \Ass$ and $\cS=\C(\textbf{k})$ (resp. $\cS=\Sp^{\Sigma}$), we recover the classical definition of Hochschild cohomology (resp. topological Hochschild cohomology). More generally, when $\P =\E_n$ (for $n<\infty$), we deduce that the space $\lim_{\Tw(\P)}(\mathbb{U}\End_{A,N})^\star$ represents the $\E_n$-Hochschild cohomology of $A$ with coefficients in $N$.
\end{example}

Motivated by the observations above and by the results given in the previous subsection, we propose a cohomology theory for simplicial operads, for which we call \textbf{$\sT$-cohomology}. Here the letter $``\sT$'' stands for $``$twist''. This is because the definition of $\sT$-cohomology is based on the construction of twisted arrow $\infty$-categories of operads. 

\begin{define} Let $\P$ be a fibrant {and $\Sigma$-cofibrant} simplicial operad. The $\sT$-\textbf{cohomology complex of} $\P$ with coefficients in a functor $\F : \Tw(\P) \lrar \cS_\infty$ is defined by $$\HT^\star(\P ; \F) := \lim\F.$$ Dually, the $\sT$-\textbf{homology complex of} $\P$ with coefficients in a functor $\G : \Tw(\P)^{\op} \lrar \cS_\infty$ is defined by $$\HT_\star(\P ; \G):=\colim\G.$$
\end{define}

\begin{rem} When $\cS_\infty=\Sp$ the $\infty$-category of spectra, then $\sT$-cohomology of $\P$ coincides with Quillen cohomology of $\P$ considered as an infinitesimal bimodule over itself ({see} Remark \ref{r:QprinEn1}){; and additionally, for} the case where $\P=\E_n$ {then} $\sT$-cohomology is very close to Quillen cohomology as mentioned in Theorem \ref{p:cotantEnsum} and Remark \ref{r:QprinEn2}.
\end{rem}

\begin{rem} {Unfortunately,} $\sT$-cohomology of the operad $\E_\infty$ is not at all interesting because the category $\Tw(\E_\infty)\simeq\sN(\Fin_*^{\op})$ contains $\l 0 \r$ as its zero object. Despite this, for a given functor $\F:\Tw(\E_\infty) \lrar \cS_\infty$, it would be interesting to consider $\sT$-cohomology of $\E_n$ with coefficients in the restriction of $\F$ to $\Tw(\E_n)$, i.e. the composed functor
	$$ \Tw(\E_n) \lrar \Tw(\E_\infty) \x{\F}{\lrar} \cS_\infty.$$
For the base category of modules over a field of characteristic $0$, this type of (co)homology was first considered by Loday \cite{Lod}, {and developed further} by Pirashvili \cite{Pirash}, in which the obtained (co)homology groups admit a \textbf{Hodge-type decomposition} (see {loc. cit.} for more details). 
\end{rem}

\begin{rem} Notice also that when $\P$ is concentrated in arity $1$ (i.e. $\P$ is identified with an $\infty$-category), then $\sT$-cohomology is just the usual cohomology of $\infty$-categories, which is a natural generalization of the \textbf{Baues-Wirsching cohomology} for small categories (cf. \cite{BauWi}). Furthermore, when $\cS_\infty=\Sp$, then $\sT$-cohomology of $\P$ is (up to a shift) equivalent to Quillen cohomology {of $\P$}, according to Theorem \ref{t:Qcohominftycat} of Harpaz-Nuiten-Prasma. 
\end{rem}

{In summary, the spirit of the proposed cohomology theory can be outlined as follows}: $\sT$-cohomology is established in the fashion of \textbf{functor cohomology}, and can be thought of as a kind of  generalized Hochschild cohomology that is controlled by operads and besides that, in many cases of interest, $\sT$-cohomology gives an approximation to Quillen cohomology, according to the Quillen principle mentioned earlier. Moreover, as we have seen, {it produces a proper cohomology theory not only for	operads but also for algebras over them}.

\medskip

\underline{\textbf{Organization of the paper.}} In $\S$\ref{s:background}, we recall briefly some necessary facts relevant to enriched operads and various types of modules over an operad. This section is also devoted to the most important concepts in the paper including tangent category, cotangent complex and Quillen cohomology groups. In $\S$\ref{s:convention}, we set up several conditions on the base category, which we work with throughout $\S$\ref{s:optangent} and $\S$\ref{chap:Qcohomenrichedop}. Section \ref{s:optangent} is devoted to the proof of Theorem \ref{t:mainoptangentsum}. In $\S$\ref{chap:Qcohomenrichedop}, we first set up an extra condition on the base category and by the way, provide several illustrations for this condition. The ultimate goal of this section is to prove Theorem \ref{t:qcohomenriched}. Section \ref{s:simplicialoperad} is devoted to Quillen cohomology of simplicial operads. We shall discuss the construction of twisted arrow $\infty$-categories of simplicial operads, after having described the unstraightening of simplicial (co)presheaves. We then explain how this construction classifies Quillen cohomology of simplicial operads.   

\smallskip

\underline{\textbf{Acknowledgements.}} This paper is part of the author's PhD thesis, written under the supervision of Yonatan Harpaz at {Universit{\'e} Sorbonne Paris Nord}. The author is greatly indebted to his advisor for suggesting the problems and for carefully reading earlier drafts of the paper. The author also extends gratitude to the anonymous referees for their valuable comments and suggestions, which have greatly contributed to improving the paper. Special thanks to one referee for providing an idea that allowed the author to significantly improve the computation of $\Tw(\E_n)$. The author acknowledges support
from the French project ANR-16-CE40-0003 ChroK.

\section{Background and notations}\label{s:background}

This section is devoted to the basic concepts and notations we work with throughout the paper. We first recall briefly fundamental notions relevant to enriched operads, various types of modules over an operad and their homotopy theories. The last subsection is devoted to the needed concepts relevant to the Quillen cohomology theory.

\subsection{Enriched operads}\label{s:operad}

\smallskip

Let $\SS$ be a \textbf{symmetric monoidal category}. Given a set $C$, regarded as the set of \textbf{colors}, we denote by 
$$\Seq(C) := \{(c_1,\cdots,c_n;c) \, | \, c_i,c \in C , n\geqslant0\}$$ and refer to it as the collection of $C$-\textbf{sequences}. For each $n\geq0$, we let $\Sigma_n$ denote the \textbf{symmetric group on $n$ elements}.

\begin{dfn} The data of a \textbf{symmetric} $C$-\textbf{collection} (also called a $C$-\textbf{symmetric sequence}) in $\SS$ consist of a collection $$M = \{M(c_1,\cdots,c_n;c)\}_{(c_1,\cdots,c_n;c) \in \Seq(C)}$$ of objects in $\cS$ and such that for each $(c_1,\cdots,c_n;c) \in \Seq(C)$ and $\sigma\in \Sigma_n$, there is a map of the form $$\sigma^{*} : M(c_1,\cdots,c_n;c) \lrar M(c_{\sigma(1)},\cdots,c_{\sigma(n)};c) .$$ These maps define a right action by $\Sigma_n$ in the sense that, for every $\sigma,\tau\in \Sigma_n$, we have $\sigma^{*}\tau^*=(\tau\sigma)^*$ and $\iota_n^* = \Id$ where $\iota_n \in \Sigma_n$ signifies the trivial permutation. With the obvious maps, symmetric $C$-collections in $\SS$ form a category, denoted by $\Coll_C(\SS)$.
\end{dfn}

\begin{rem}\label{r:refomcollection} Let $\sP(C)$ denote the groupoid whose set of objects is given by $$\Ob(\sP(C)) =\{\emptyset\} \sqcup (\bigsqcup_{n\geq1} C^{\times n})$$ and whose morphisms consist of the identity $\Id_\emptyset$ and the morphisms of the form $$\sigma^{*} : (c_1,\cdots,c_n) \lrar (c_{\sigma(1)},\cdots,c_{\sigma(n)})$$ with $\sigma\in \Sigma_n$, $n\geqslant1$. One can then show that the category $\Coll_C(\SS)$ is isomorphic to $\Fun(\sP(C) \times C , \cS)$ the category of $\cS$-enriched functors from $\sP(C) \times C$ to $\cS$. 
\end{rem}

The well known \textbf{composite product}
$$ -\circ- : \Coll_C(\SS) \times \Coll_C(\SS) \lrar \Coll_C(\SS)$$ endows  $\Coll_C(\SS)$ with a monoidal structure. The monoidal unit will be denoted by $\I_C$, {where $\I_C(c;c)=1_{\cS}$ for every $c\in C$ and $\I_C$ takes the value $\emptyset_{\cS}$ otherwise.} (See, e.g., \cite{Yonatan, Pavlov, Yau}). 

\begin{dfn} A \textbf{symmetric} $C$-\textbf{colored operad} in $\cS$ is a monoid in the monoidal category $(\Coll_C(\SS),-\circ-,\I_C)$. We denote by $\Op_C(\SS)$ the category of symmetric $C$-colored operads.
\end{dfn}

\begin{rem} Unwinding definition, a symmetric $C$-colored operad in $\cS$ is a symmetric $C$-collection $\P$ equipped with

	$\bullet$ a \textbf{composition} whose data consist of the $\Sigma_*$-equivariant maps of the form 
	$$ \P(c_1,\cdots,c_n;c) \otimes \P(c_{1,1},\cdots,c_{1,k_1};c_1) \otimes \cdots \otimes \P(c_{n,1},\cdots,c_{{n,k_n}};c_n) \lrar \P(c_{1,1},\cdots,c_{1,k_1}, \cdots, c_{n,1},\cdots,c_{{n,k_n}} ; c) , $$

	$\bullet$ and for each color $c \in C$, a \textbf{unit operation} $\id_c : 1_\cS \lrar \P(c;c)$.

	\noindent The composition must satisfy the essential axioms of associativity and unitality.
\end{rem}

Given $\P \in \Op_C(\SS)$, each object $\P(c_1,\cdots,c_n;c)$ will be called a \textbf{space of $n$-ary operations}. Note that the collection of $1$-ary operations of $\P$, denoted by $\P_1$, inherits an obvious $\cS$-enriched category structure. We shall refer to $\P_1$ as the \textbf{underlying category of}  $\P$.

The notion of a \textbf{nonsymmetric $C$-colored operad} (resp. \textbf{nonsymmetric $C$-collection}) is the same as that of a symmetric $C$-colored operad (resp. symmetric  $C$-collection) without the action of symmetric groups. We denote by $\nsColl_C(\cS)$ (resp. $\nsOp_C(\cS)$) the category of nonsymmetric $C$-collections (resp. nonsymmetric $C$-colored operads) in $\cS$. The natural passage from nonsymmetric to symmetric context is performed by the \textbf{symmetrization functor} $\Sym$. Explicitly, the functor $\Sym : \nsColl_C(\cS) \lrar \Coll_C(\cS)$ is given by sending each nonsymmetric $C$-collection $M$ to $\Sym(M)$ with
$$ \Sym(M)(c_1,\cdots,c_n;c) = \bigsqcup_{\sigma \in \Sigma_n} M(c_{\sigma(1)},\cdots,c_{\sigma(n)} ; c) .$$
The structure maps are induced by the multiplication of permutations in an evident way. By construction, the functor $\Sym$ forms a left adjoint to the associated forgetful functor.

\begin{rem} Given two objects $M , N \in \nsColl_C(\cS)$, {there is an obvious natural inclusion in $\nsColl_C(\cS)$ of the form} 
	{$$ M \, \ovl{\circ} \, N \lrar \Sym(M) \circ \Sym(N) $$
where $-\ovl{\circ}-$ refers to the composite product of nonsymmetric $C$-collections. This}	induces a natural map $\Sym(M\,{\ovl{\circ}}\, N) \lrar \Sym(M) \circ \Sym(N)$. It can then be verified that the latter is a natural isomorphism of symmetric $C$-collections with the inverse $$\Sym(M) \circ \Sym(N) \lrarsimeq \Sym(M \, {\ovl{\circ}} \, N)$$ being induced by the concatenation of linear orders 
$$ \Sigma_k \times \Sigma_{n_1} \times \cdots \times \Sigma_{n_k} \lrar \Sigma_{n_1+\cdots+n_k} .$$
Moreover, it is easy to see that the functor $\Sym$ preserves monoidal units and thus, it is strong monoidal.
\end{rem}

By the above remark, the functor $\Sym : \nsColl_C(\cS) \lrar \Coll_C(\cS)$ descends to a functor between monoids
$$ \Sym : \nsOp_C(\cS) \lrar \Op_C(\cS) ,$$
which forms a left adjoint to the associated forgetful functor.

\begin{rem} Since {we are mostly interested in} the symmetric context, from now on, unless
	otherwise specified, we shall omit the word $``$symmetric'' when mentioning an object of $\Op_C(\SS)$ (or $\Coll_C(\SS)$).
\end{rem}

One can integrate all the categories $\Op_C(\SS)$ for $C \in \Sets$ into a single category of $\cS$-enriched operads, just like the way one establishes {the category of categories.}

\begin{dfn} The \textbf{category of} $\SS$-\textbf{enriched operads}, denoted by $\Op(\SS)$, is the (contravariant) \textbf{Grothendieck construction}
	$$ \Op(\SS) := \int_{C\in \Sets} \Op_C(\SS) $$
	in which for each map $\alpha : C \rar D$ of sets, the corresponding functor $\alpha^{*} : \Op_D(\SS) \lrar \Op_C(\SS)$ is given by taking $\Q\in\Op_D(\SS)$ to $\alpha^{*}\Q$ with 
	$$ \alpha^{*}\Q(c_1,\cdots,c_n;c) := \Q(\alpha(c_1),\cdots,\alpha(c_n);\alpha(c)) .$$
\end{dfn}

Unwinding definition, an object of $\Op(\SS)$ is a pair $(C,\P)$ with $C\in \Sets$ and $\P\in \Op_C(\SS)$, while a morphism $(C,\P) \lrar (D,\Q)$ consists of a map $\alpha : C \rar D$ of sets and a map $f :  \P \lrar \alpha^{*}\Q$ of $C$-colored operads.

\begin{rem} For $\P\in \Op_C(\SS)$, there is an adjunction 
\begin{equation}\label{eq:L_PandR_P}
	\mathcal{L}_\P : \adjunction*{}{\Op_C(\mathcal{S})_{\mathcal{P}/}}{\Op(\mathcal{S})_{\mathcal{P}/}}{} : \mathcal{R}_\P
\end{equation}
where the left adjoint $\mathcal{L}_\P$ is the obvious embedding functor. Let $\mathcal{P}\overset{f}{\longrightarrow} \mathcal{Q}$ be an object of $\Op(\mathcal{S})_{\mathcal{P}/}$. Then $\mathcal{R}_\P(\mathcal{Q})$ is given on each level by the restriction of colors:
$$ \mathcal{R}_\P(\mathcal{Q}) (c_1,\cdots,c_n;c) := \mathcal{Q}(f(c_1),\cdots,f(c_n);f(c)) .$$
\end{rem}

\smallskip

\subsection{Various types of operadic modules}\label{s:operadicmodules}

\smallskip

Let $\P$ be a $C$-colored operad in $\cS$.  We let $\LMod(\P)$ (resp. $\RMod(\P)$) denote the {category} of \textbf{left (resp. right)} $\P$-\textbf{modules}. Moreover, we let $\BMod(\P)$ and $\IbMod(\P)$ respectively denote the categories of $\P$-\textbf{bimodules} and \textbf{infinitesimal} $\P$-\textbf{bimodules}. Let us revisit these quickly. 

\smallskip

Operadic left module (resp. right module, bimodule) is the usual notion of left module (resp. right module, bimodule) over an operad when one  regards operads as monoids in the monoidal category of symmetric sequences. More explicitly, these are expressed as follows.

\begin{dfn} \label{d:operadicmodules}

		(i) A \textbf{left $\mathcal{P}$-module} is a $C$-collection $M$ equipped with a \textit{left} $\mathcal{P}$-\textit{action} map $\mathcal{P}\circ M\lrar M$ whose data consist of the $\Sigma_*$-equivariant maps of the form
		$$  \P(c_1,\cdots,c_n;c) \otimes M(d_{1,1},\cdots,d_{1,k_1};c_1) \otimes \cdots \otimes M(d_{n,1},\cdots,d_{n,k_n};c_n) \lrar M(d_{1,1},\cdots,d_{1,k_1},\cdots,d_{n,1},\cdots,d_{n,k_n};c) $$
		satisfying the classical axioms of associativity and unitality for left modules.
		
		(ii) Dually, a \textbf{right $\mathcal{P}$-module} is a $C$-collection $M$ equipped with a \textit{right} $\mathcal{P}$-\textit{action} map $M \circ \P \lrar M$ whose data consist of the $\Sigma_*$-equivariant maps of the form
		$$  M(c_1,\cdots,c_n;c) \otimes \P(d_{1,1},\cdots,d_{1,k_1};c_1) \otimes \cdots \otimes \P(d_{n,1},\cdots,d_{n,k_n};c_n) \lrar M(d_{1,1},\cdots,d_{1,k_1},\cdots,d_{n,1},\cdots,d_{n,k_n};c) $$
		satisfying the classical axioms of associativity and unitality for right modules.
		
		(iii) A $\P$\textbf{-bimodule} is a $C$-collection $M$ equipped with both a left and a right $\mathcal{P}$-module structure. These must satisfy the essential axiom of  compatibility. 
\end{dfn}  

Moreover, a $\P$-algebra is simply a left $\P$-module concentrated in level 0. More precisely, this is expressed as follows.

\begin{dfn} A $\P$-\textbf{algebra} is an object  $A \in \cS^{\times \, C}$ which is equipped, for each $(c_1,\cdots,c_n;c)$, with a $\P$-action map 
	$$ \P(c_1,\cdots,c_n;c)\otimes A(c_1)\otimes\cdots\otimes A(c_n) \lrar A(c) .$$
These maps must satisfy the essential axioms of associativity, unitality and equivariance. We denote by $\Alg_\P(\SS)$ the category of $\P$-algebras.
\end{dfn}

\begin{rem} When $\P$ is concentrated in arity 1 then $\P$ is simply an $\cS$-enriched category. In this situation, the category of $\P$-algebras coincides with $\Fun(\P,\cS)$ the category of $\cS$-valued enriched functors on $\P$.
\end{rem}

\begin{example}\label{ex:initialalgebra} The collection of nullary (= 0-ary) operations of $\P$, denoted $\P_0$, inherits an obvious $\P$-algebra structure and moreover, $\P_0$ then becomes an initial object in $\Alg_\P(\SS)$.
\end{example}

\begin{example}\label{ex:operadofoperads} It is noteworthy that for each set $C$, there exists a $\Seq(C)$-colored operad in $\cS$, denoted by $\textbf{O}_C$, whose algebras are precisely the $C$-colored operads, i.e., there is an isomorphism of categories $\Alg_{\textbf{O}_C}(\cS) \cong \Op_C(\cS)$. We will refer to $\textbf{O}_C$ as the \textbf{operad of $C$-colored operads}. (Cf., e.g., [\cite{Guti}, \S 3]).
\end{example}

The structure of left modules (and hence, bimodules) over an operad is not linear in general. The notion of \textit{infinitesimal left modules (bimodules)}, as introduced by Merkulov-Vallette (\cite{Vallette}), essentially appears as the linearization of the previous one. In terms of enriched colored operads, these are defined as follows.

\begin{dfn}\phantomsection\label{d:infbmod}
	\begin{enumerate}
		\item An \textbf{infinitesimal left $\mathcal{P}$-module} is a $C$-collection $M$ equipped with the action maps of the form 
		$$ \circ^{l}_i : \P(c_1,\cdots,c_n;c) \otimes M(d_1,\cdots,d_m;c_i) \lrar M(c_1,\cdots,c_{i-1},d_1,\cdots,d_m,c_{i+1},\cdots,c_n;c) $$
		which are $\Sigma_*$-equivariant and satisfy the classical axioms of  associativity and  unitality for left modules.
		
		\item Dually, an \textbf{infinitesimal right $\mathcal{P}$-module} is a $C$-collection $M$ equipped with the action maps of the form
		$$ \circ^{r}_i : M(c_1,\cdots,c_n;c) \otimes \P(d_1,\cdots,d_m;c_i) \lrar M(c_1,\cdots,c_{i-1},d_1,\cdots,d_m,c_{i+1},\cdots,c_n;c) $$
		which are $\Sigma_*$-equivariant and satisfy the classical axioms of  associativity and unitality for right modules.
		
		\item An \textbf{infinitesimal $\mathcal{P}$-bimodule} is a $C$-collection $M$ equipped with both an infinitesimal left and an infinitesimal right $\mathcal{P}$-module structure which together satisfy the essential compatibility.
	\end{enumerate}
\end{dfn}

\begin{rem} The structure of an infinitesimal right $\P$-module is equivalent to that of a (non-infinitesimal) right $\P$-module.
\end{rem}

The definition of an infinitesimal $\mathcal{P}$-bimodule can be reformulated using the diagramatical language, which has certain advantages over the above definition. To this end, one starts with the notion of \textbf{infinitesimal composite product}:
$$ - \circ_{(1)} - : \Coll_C(\cS) \times \Coll_C(\cS) \lrar \Coll_C(\cS), $$
which can be thought of as the $``$right linearization'' of the composite product $ - \circ - $. Formally, given two $C$-collections $M$ and $N$, the $C$-collection $M\circ_{(1)} N$ is given by  the maximal sub-collection of $M\circ (\mathcal{I}_C \sqcup N)$ which is linear in $N$. To be precise, looking at the explicit formula of $M\circ (\mathcal{I}_C \sqcup N)$, we have on each level that $(M\circ_{(1)} N) (c_1,\cdots,c_n;c)$ is the sub-object of $M\circ (\mathcal{I}_C \sqcup N) (c_1,\cdots,c_n;c)$ consisting of the \textit{multi-tensor products} which contain one and only one factor in $N$. The readers can find this construction in [\cite{Loday}, Section 6.1]. Observe now that for each $M\in \Coll_C(\cS)$ there is a natural inclusion $$ \mathcal{P}\circ_{(1)}(\mathcal{P}\circ_{(1)}M) \lrar (\mathcal{P}\circ_{(1)}\mathcal{P})\circ_{(1)}M .$$ 
On other hand, the (partial) composition in $\mathcal{P}$ gives a map $\mu_{(1)} : \mathcal{P}\circ_{(1)}\mathcal{P} \lrar \mathcal{P}$. The following is equivalent to Definition \ref{d:infbmod}(i).

\begin{dfn} An \textbf{infinitesimal left} $\mathcal{P}$-\textbf{module} is a $C$-collection $M$ equipped with an action map $\mathcal{P}\circ_{(1)}M \lrar M$ satisfying the classical axioms of associativity and unitality for left modules, which are depicted as the commutativity of the following diagrams

	\begin{center}
		\begin{tikzpicture}[commutative diagrams/every diagram]
			\node (P0) at (90:2.3cm) {$\mathcal{P}\circ_{(1)}(\mathcal{P}\circ_{(1)}M)$};
			\node (P1) at (90+72:2cm) {$(\mathcal{P}\circ_{(1)}\mathcal{P})\circ_{(1)}M$} ;
			\node (P2) at (90+2*72:2cm) {\makebox[5ex][r]{$\mathcal{P}\circ_{(1)}M$}};                                     
			\node (P3) at (90+3*72:2cm) {\makebox[5ex][l]{$M$}};
			\node (P4) at (90+4*72:2cm) {$\mathcal{P}\circ_{(1)}M$};
			\path[commutative diagrams/.cd, every arrow, every label]
			(P0) edge node[swap] {} (P1)
			(P1) edge node[swap] {} (P2)
			(P2) edge node {} (P3)
			(P4) edge node {} (P3)
			(P0) edge node {} (P4);
		\end{tikzpicture}\qquad%
		\begin{tikzcd}
			\mathcal{I}_C\circ_{(1)}M \arrow[r ] \arrow[dr, "\cong	"]
			& \mathcal{P}\circ_{(1)}M  \arrow[d]\\
			& M
		\end{tikzcd}
	\end{center}

\end{dfn}

\noindent Now notice that for each $M\in\Coll_C(\cS)$ there is a natural inclusion 
$$ (\P\circ_{(1)}M)\circ \P \lrar (\P\circ \P)\circ_{(1)} (M \circ \P) . $$
The following is equivalent to Definition \ref{d:infbmod}(iii).
\begin{dfn}\label{d:infbimod} An infinitesimal $\mathcal{P}$-bimodule is a $C$-collection $M$ endowed with an infinitesimal left $\mathcal{P}$-module structure, exhibited by a map $\mathcal{P}\circ_{(1)}M \lrar M$ and with a right $\mathcal{P}$-module structure, exhibited by a map $M \circ \P \lrar M$. These are subject to the axiom of compatibility, depicted as the commutativity of the following diagram

	\begin{center}
		\begin{tikzpicture}[commutative diagrams/every diagram]
			\node (P0) at (90:2.3cm) {$(\P\circ_{(1)}M)\circ \P$};
			\node (P1) at (90+72:2cm) {$(\P\circ \P)\circ_{(1)} (M \circ \P)$} ;
			\node (P2) at (90+2*72:2cm) {\makebox[5ex][r]{$\P\circ_{(1)}M$}};                                     
			\node (P3) at (90+3*72:2cm) {\makebox[5ex][l]{$M$}};
			\node (P4) at (90+4*72:2cm) {$M\circ\P$};
			\path[commutative diagrams/.cd, every arrow, every label]
			(P0) edge node[swap] {} (P1)
			(P1) edge node[swap] {} (P2)
			(P2) edge node {} (P3)
			(P4) edge node {} (P3)
			(P0) edge node {} (P4);
		\end{tikzpicture}
	\end{center}
	
\end{dfn}

\begin{rem}\label{r:freeinfbimod} In the above diagram, the $C$-collection $(\P\circ_{(1)}M)\circ \P$ does not represent the free infinitesimal $\P$-bimodule generated by $M$. (This does not even carry a canonical infinitesimal $\P$-bimodule structure). To find the exact one, we first factor the free functor $\Free^{ib} : \Coll_C(\cS)  \lrar \IbMod(\P)$ as $$ \Coll_C(\cS) \x{\F_1}{\lrar} \RMod(\P) \x{\F_2}{\lrar} \IbMod(\P)$$ where $\F_1$ and $\F_2$ are the left adjoints to the associated forgetful functors. It can then be shown that $\F_1\cong(-)\circ\P$ and $\F_2 \cong \P\circ_{(1)}(-)$. In conclusion, the functor $\Free^{ib}$ is given by the formula  $$\Free^{ib}=\mathcal{P}\circ_{(1)}( - \circ \mathcal{P} ).$$ 
\end{rem}

\begin{rem} Note that the map $\mu_{(1)} : \mathcal{P}\circ_{(1)}\mathcal{P} \lrar \mathcal{P}$ endows $\P$ with the structure of an infinitesimal bimodule over itself. Now let $M$ be a $\mathcal{P}$-bimodule under $\mathcal{P}$. Then, $M$ inherits a canonical infinitesimal $\P$-bimodule structure (under $\P$) induced by inserting the unit operations of $\mathcal{P}$ into $M$. This procedure determines a \textit{restriction functor}
	\begin{equation}\label{eq:adjbimodtoinf}
\BMod(\mathcal{P})_{\mathcal{P}/}\lrar \IbMod(\mathcal{P})_{\mathcal{P}/},
	\end{equation}
which admits a left adjoint named \textit{induction functor}.
\end{rem}

Another important one is the notion of \textit{modules over an operadic algebra}. Let $A$ be a $\P$-algebra.

\begin{dfn} An $A$\textbf{-module over} $\P$ is an object  $M \in \cS^{\times \, C}$ equipped, for each sequence $(c_1,\cdots,c_n;c)$, with a mixed $(\P,A)$-action map of the form
	$$ \P(c_1,\cdots,c_n;c) \otimes \;  \bigotimes _{i \in \{1,\cdots,n\} - \{k\}} \; A(c_i) \otimes M(c_k) \lrar M(c) .$$
 These maps must satisfy the essential axioms of associativity, unitality and equivariance. With the obvious maps, $A$-modules over $\P$ form a category, denoted by $\Mod_\P^{A}$.
\end{dfn}

 To reformulate $\Mod_\P^{A}$ as a category of $\cS$-valued enriched functors, one will need the construction of \textit{enveloping operads}. Denote by $\Pairs_C(\cS)$ the category whose objects are the pairs $(\P , A)$ with $\P\in \Op_C(\mathcal{S})$ and $A \in \Alg_\P(\SS)$, and whose morphisms are the pairs $(\varphi , f) : (\P , A) \lrar (\Q , B)$ with $\varphi : \P \rar \Q$ being a map in $\Op_C(\mathcal{S})$ and $f : A \rar B$ a map of $\P$-algebras. There is a canonical functor $ \delta : \Op_C(\mathcal{S}) \lrar \Pairs_C(\cS)$ sending each $C$-colored operad $\P$ to the pair $(\P , \P_0)$ (see Example \ref{ex:initialalgebra}). According to \cite{BergerMoerdijk}, the functor $\delta$ admits a left adjoint $\Env : \Pairs_C(\cS) \lrar \Op_C(\mathcal{S})$ called the \textbf{enveloping functor}.

\begin{dfn} The \textbf{enveloping operad} associated to a  pair $(\P , A) \in \Pairs_C(\cS)$ is defined to be $\Env(\P , A)$ the image of $(\P , A)$ through the enveloping functor.
\end{dfn}

\begin{rem}\label{r:AmodPasfunctor} According to [Theorem 1.10, \cite{BergerMoerdijk}], there is a canonical isomorphism 
	\begin{equation}\label{eq:AmodP}
		\Mod_\P^{A} \cong \Fun(\Env(\P,A)_1 , \cS)  
	\end{equation}
	between the category of $A$-modules over $\P$ and the category of $\cS$-valued enriched functors on $\Env(\P,A)_1$ the underlying category of $\Env(\P,A)$.
\end{rem}

\begin{rem} Another main interest in this construction is that there is a canonical isomorphism $$\Alg_{\Env(\P,A)}(\cS) \cong \Alg_\P(\cS)_{A/}$$ between the categories of $\Env(\P,A)$-algebras and $\P$-algebras under $A$. Furthermore,  by construction there is a canonical map $ j_A : \P \lrar \Env(\P , A)$ of $C$-colored operads. This endows $\Env(\P , A)_0$ with a canonical $\P$-algebra structure and moreover, $\Env(\P , A)_0$ is isomorphic to $A$ as $\P$-algebras. (See around [\cite{BergerMoerdijk}, Lemma 1.7]). 
\end{rem}

\smallskip

It is convenient that each of the categories $\IbMod(\P)$, $\BMod(\P)$, $\LMod(\P)$ and $\RMod(\P)$ can be encoded by an enriched operad (or category). In terms of single-colored operads, these constructions can be found in \cite{Turchin,Julien}.

\begin{notns}\phantomsection 
	\begin{enumerate}
		\item We let $\Fin$ denote the smallest skeleton of the category of finite sets whose objects consist of $\underline{0}:=\emptyset$ and $\underline{m}:=\{1,\cdots,m\}$ for $m\geqslant1$.
		
		\item We denote by $\Fin_*$ the category whose objects are finite pointed sets $\left \langle m \right \rangle := \{0,1,\cdots,m\}$ (with $0$ as the basepoint) for $m\geqslant0$,  and whose morphisms are basepoint-preserving maps. In other words, $\Fin_*$ is the smallest skeleton of the category of finite pointed sets.
	\end{enumerate}
\end{notns}

Note that there is an obvious embedding functor $\Fin \lrar \Fin_*$ taking each $\underline{m}$ to $\l m \r$.

\begin{cons}\label{catencodinginfbimod}
	We now construct an $\cS$-enriched category, denoted $\textbf{Ib}^{\mathcal{P}}$, which encodes infinitesimal $\P$-bimodules. The set of objects of $\textbf{Ib}^{\mathcal{P}}$ is $\Seq(C)$, while its mapping spaces are defined as follows. For each map $\left \langle m \right \rangle  \overset{f}{\longrightarrow} \left \langle n \right \rangle$ in $\Fin_*$, we denote by
	$$ \Map^{f}_{\textbf{Ib}^{\mathcal{P}}}\left( \, \left(c_1,\cdots,c_n;c \right) , \left(d_1,\cdots,d_m;d \right) \, \right) :=  \mathcal{P}\left (c,\{d_j\}_{j\in f^{-1}(0)};d \right ) \otimes \bigotimes_{i=1,\cdots,n} \mathcal{P} \left (\{d_j\}_{j\in f^{-1}(i)};c_i \right ) $$
	in which, for each $i\in \{0,\cdots,n\}$, the elements of $\{d_j\}_{j\in f^{-1}(i)}$ are put in the natural ascending order of $j$. Then, we define
	$$ \Map_{\textbf{Ib}^{\mathcal{P}}}\left( \, \left(c_1,\cdots,c_n;c \right) , \left(d_1,\cdots,d_m;d \right) \, \right) := \bigsqcup_{\left \langle m \right \rangle  \overset{f}{\rightarrow} \left \langle n \right \rangle} \Map^{f}_{\textbf{Ib}^{\mathcal{P}}}\left( \, \left(c_1,\cdots,c_n;c \right) , \left(d_1,\cdots,d_m;d \right) \, \right) .$$
	Observe that 
	$$\Map^{\Id_{\left \langle n \right \rangle}}_{\textbf{Ib}^{\mathcal{P}}} \left( \, \left(c_1,\cdots,c_n;c \right) , \left(c_1,\cdots,c_n;c \right) \, \right) =  \mathcal{P}(c;c) \otimes \mathcal{P}(c_1;c_1) \otimes \cdots \otimes \mathcal{P}(c_n;c_n).$$
	Due to this, we can define the unit morphisms of $\textbf{Ib}^{\mathcal{P}}$ via the unit operations of $\mathcal{P}$. Moreover, the structure maps of $\textbf{Ib}^{\mathcal{P}}$ are canonically defined via the composition in $\mathcal{P}$, along with the symmetric action on $\P$. (See also \cite{Turchin}, $\S$2).
	
\end{cons}

\begin{prop}\label{rewriteinfinitesimal1}
	There is a canonical isomorphism 
	$$ \IbMod(\mathcal{P})\cong \Fun(\textbf{Ib}^{\mathcal{P}},\mathcal{S})$$
	between the category of infinitesimal $\P$-bimodules and the category of $\cS$-valued enriched functors on $\textbf{Ib}^{\mathcal{P}}$.
\end{prop}
\begin{proof}[Sketch of proof] (1) Let $M: \textbf{Ib}^{\mathcal{P}} \lrar \mathcal{S}$ be an enriched functor. We construct the associated infinitesimal $\mathcal{P}$-bimodule, still denoted  $M$, as follows. Each permutation $\alpha \in \Sigma_n$ determines a map $\left \langle n \right \rangle  \overset{\alpha}{\lrar} \left \langle n \right \rangle$. Observe now that
	$$ \Map^{\alpha}_{\textbf{Ib}^{\mathcal{P}}}\left( \, \left(c_1,\cdots,c_n;c \right) , (c_{\alpha(1)},\cdots,c_{\alpha(n)};c ) \, \right) = \mathcal{P}(c;c) \otimes \mathcal{P}(c_1;c_1) \otimes \cdots \otimes \mathcal{P}(c_n;c_n).$$
	In particular, the functor structure map of $M$
	$$ \Map_{\textbf{Ib}^{\mathcal{P}}} \left( \, \left(c_1,\cdots,c_n;c \right) , (c_{\alpha(1)},\cdots,c_{\alpha(n)};c ) \, \right) \otimes M(c_1,\cdots,c_n;c) \lrar M(c_{\alpha(1)},\cdots,c_{\alpha(n)};c)$$
	has a component of the form
	$$ \mathcal{P}(c;c) \otimes \mathcal{P}(c_1;c_1) \otimes \cdots \otimes \mathcal{P}(c_n;c_n) \otimes M(c_1,\cdots,c_n;c) \lrar M(c_{\alpha(1)},\cdots,c_{\alpha(n)};c). $$
	Now, the evaluation at the unit operations $\id_c, \id_{c_1}, \cdots, \id_{c_n}$ of $\mathcal{P}$ determines an action map of the form $M(c_1,\cdots,c_n;c) \x{\alpha^{*}}{\lrar} M(c_{\alpha(1)},\cdots,c_{\alpha(n)};c)$.

	Next we define the infinitesimal right action of $\mathcal{P}$ on $M$. Observe that the structure map
	$$ \Map_{\textbf{Ib}^{\mathcal{P}}}\left( \, \left(c_1,\cdots,c_n;c \right) , (c_1,\cdots,c_{i-1},d_1,\cdots,d_m,c_{i+1},\cdots,c_n;c) \, \right) \otimes M(c_1,\cdots,c_n;c) $$  
	$$ \lrar M(c_1,\cdots,c_{i-1},d_1,\cdots,d_m,c_{i+1},\cdots,c_n;c)$$
	has a component of the form
	$$ M(c_1,\cdots,c_n;c) \otimes \mathcal{P}(d_1,\cdots,d_m;c_i) \otimes \mathcal{P}(c;c) \otimes \mathcal{P}(c_1;c_1) \otimes \cdots \otimes \mathcal{P} (c_{i-1};c_{i-1}) \otimes \mathcal{P} (c_{i+1};c_{i+1}) \otimes \cdots \otimes \mathcal{P}(c_n;c_n) $$
	$$ \lrar M(c_1,\cdots,c_{i-1},d_1,\cdots,d_m,c_{i+1},\cdots,c_n;c) .$$
	This induces an infinitesimal right $\mathcal{P}$-action on $M$ by evaluating at the unit operations of $\mathcal{P}$.
	
	To define the infinitesimal left $\mathcal{P}$-action, observe that the structure map
	$$ \Map_{\textbf{Ib}^{\mathcal{P}}}\left( \, \left(d_1,\cdots,d_m;c_i \right) , (c_1,\cdots,c_{i-1},d_1,\cdots,d_m,c_{i+1},\cdots,c_n;c) \, \right) \otimes M(d_1,\cdots,d_m;c_i) $$  
	$$\lrar M(c_1,\cdots,c_{i-1},d_1,\cdots,d_m,c_{i+1},\cdots,c_n;c) $$
	has a component of the form
	$$ \mathcal{P}(c_i,c_1,\cdots,c_{i-1},c_{i+1},\cdots,c_n;c) \otimes \mathcal{P}(d_1;d_1) \otimes \cdots \otimes \mathcal{P}(d_m;d_m)\otimes M(d_1,\cdots,d_m;c_i) $$
	$$ \lrar M(c_1,\cdots,c_{i-1},d_1,\cdots,d_m,c_{i+1},\cdots,c_n;c) .$$
	The latter induces an infinitesimal left action by evaluating at the unit operations of $\mathcal{P}$, again.

	\smallskip
	
	(2) Conversely, let $M$ be an infinitesimal $\mathcal{P}$-bimodule. We need to endow $M$ with a canonical enriched functor structure $\textbf{Ib}^{\mathcal{P}} \longrightarrow \mathcal{S}$. To this end, we have to define the maps of the form
	$$ \Map_{\textbf{Ib}^{\mathcal{P}}}\left( \, \left(c_1,\cdots,c_n;c \right) , (d_1,\cdots,d_m;d ) \, \right) \otimes M(c_1,\cdots,c_n;c) \lrar M(d_1,\cdots,d_m;d) .$$
	This map must consist of, for each $\left \langle m \right \rangle  \overset{f}{\longrightarrow} \left \langle n \right \rangle$, a component map of the form
	$$\mathcal{P}\left (c,\{d_j\}_{j\in f^{-1}(0)};d \right ) \otimes \bigotimes_{i=1,\cdots,n} \mathcal{P} \left (\{d_j\}_{j\in f^{-1}(i)};c_i \right ) \otimes M(c_1,\cdots,c_n;c) \lrar M(d_1,\cdots,d_m;d) .$$
	The latter can be naturally defined using the (two sided) infinitesimal $\mathcal{P}$-action on $M$, along with the action of the symmetric groups on $M$.  
\end{proof}

\begin{cons}\label{rewriterightmodules}
	The enriched category which encodes the category of right $\mathcal{P}$-modules will be denoted by $\textbf{R}^{\mathcal{P}}$. Its set of objects is $\Seq(C)$, while its mapping objects are given by
	$$ \Map_{\textbf{R}^{\mathcal{P}}}\left( \, \left(c_1,\cdots,c_n;c \right) , \left(d_1,\cdots,d_m;c \right) \, \right) := \bigsqcup_{ \underline{m} \overset{f}{\rightarrow} \underline{n}} \left [  \bigotimes_{i=1,\cdots,n} \mathcal{P} \left (\{d_j\}_{j\in f^{-1}(i)};c_i \right )  \right ] $$
	where the coproduct ranges over the hom-set $\Hom_{\Fin}(\underline{m},\underline{n})$. Observe that there is a map 
	$$ \Map_{\textbf{R}^{\mathcal{P}}}\left( \, \left(c_1,\cdots,c_n;c \right) , \left(d_1,\cdots,d_m;c \right) \, \right) \lrar \Map_{\textbf{Ib}^{\mathcal{P}}}\left( \, \left(c_1,\cdots,c_n;c \right) , (d_1,\cdots,d_m;c ) \, \right) $$
	induced by the embedding $\Fin \lrar \Fin_*$ and by inserting the unit operation $\id_c$ into the factor $\P(c;c)$ of the right hand side. The categorical structure of $\textbf{R}^{\mathcal{P}}$ is then defined via the operad structure of $\P$, so that $\textbf{R}^{\mathcal{P}}$ forms a subcategory of $\textbf{Ib}^{\mathcal{P}}$.
\end{cons}

{One can then prove the following}, similarly as in the proof of Proposition \ref{rewriteinfinitesimal1}. 
\begin{prop}\label{p:rewriterightmodules} There is a canonical isomorphism 
	$$ \RMod(\mathcal{P})\cong \Fun(\textbf{R}^{\mathcal{P}},\mathcal{S}) $$
	between the category of right $\P$-modules and the category of $\cS$-valued enriched functors on $\textbf{R}^{\mathcal{P}}$.
\end{prop}

We now wish to construct an operad encoding the category of $\P$-bimodules, yet it will be more convenient for us to start with  $\BMod(\mathcal{P})_{\P/}$ the category of  $\P$-bimodules under $\P$. 

\begin{cons}\label{operadencodebimodulesunderP}
	The $\cS$-enriched operad which encodes  $\mathcal{P}$-bimodules under $\mathcal{P}$ will be denoted by $\textbf{B}^{\mathcal{P}/}$. Its set of colors is again $\Seq(C)$. The nullary operations of $\textbf{B}^{\mathcal{P}/}$ agree with $\mathcal{P}$, i.e.,
	$$ \textbf{B}^{\mathcal{P}/}( ; (c_1,\cdots,c_n;c)) := \mathcal{P}(c_1,\cdots,c_n;c) ,$$
	while its 1-ary operations coincide with those of $\textbf{Ib}^{\mathcal{P}}$ (see Construction \ref{catencodinginfbimod}), i.e.,
	$$ \textbf{B}^{\mathcal{P}/}\left( \, \left(c_1,\cdots,c_n;c \right) ; \left(d_1,\cdots,d_m;d \right) \, \right) = \bigsqcup_{\left \langle m \right \rangle  \overset{f}{\rightarrow} \left \langle n \right \rangle} \left [ \mathcal{P}\left (c,\{d_j\}_{j\in f^{-1}(0)};d \right ) \otimes \bigotimes_{i=1,\cdots,n} \mathcal{P} \left (\{d_j\}_{j\in f^{-1}(i)};c_i \right )  \right ]  $$
	where $f$ ranges over the set $\Hom_{\Fin_*}(\left \langle m \right \rangle, \left \langle n \right \rangle)$. Then we may extend the above formula to obtain the spaces of operations of higher arities. A typical space of $n$-ary operations of $\textbf{B}^{\mathcal{P}/}$ is given by 
	$$ \textbf{B}^{\mathcal{P}/}\left( \,    \left(c_1,\cdots,c_{r_1}; c^{(1)} \right) ,    \left(c_{r_1+1},\cdots,c_{r_1+r_2}; c^{(2)} \right) , \cdots ,  \left(c_{r_1+\cdots+r_{n-1}+1},\cdots,c_{r_1+\cdots+r_{n}}; c^{(n)} \right)  ;   \left(d_1,\cdots,d_m;d \right) \, \right)$$  $$=\bigsqcup_{\left \langle m \right \rangle  \overset{f}{\lrar} \left \langle r_1+\cdots+r_{n} \right \rangle} \left [ \mathcal{P}\left (c^{(1)},\cdots,c^{(n)},\{d_j\}_{j\in f^{-1}(0)};d \right ) \otimes \bigotimes_{i=1,\cdots,r_1+\cdots+r_{n}} \mathcal{P} \left (\{d_j\}_{j\in f^{-1}(i)};c_i \right )  \right ] .$$
	The $\Sigma_n$-action is given by permuting the colors $c^{(1)},\cdots,c^{(n)}$ on the factor $$\mathcal{P}\left (c^{(1)},\cdots,c^{(n)},\{d_j\}_{j\in f^{-1}(0)};d \right )$$ and simultaneously, permuting the terms $\{c_1,\cdots,c_{r_1}\}, \cdots, \{c_{r_1+\cdots+r_{n-1}+1},\cdots,c_{r_1+\cdots+r_{n}}\}$ on the factor $$\bigotimes_{i=1,\cdots,r_1+\cdots + r_n} \mathcal{P} \left (\{d_j\}_{j\in f^{-1}(i)};c_i \right ) .$$ 
	The composition {in} $\textbf{B}^{\mathcal{P}/}$ is canonically defined via the composition {in} $\mathcal{P}$, while its unit operations are exactly those of $\textbf{Ib}^{\mathcal{P}}$. 
\end{cons}

\begin{prop}\label{bimoduleunderPsasalgebras} There is a canonical isomorphism 
	$$ \BMod(\mathcal{P})_{\mathcal{P}/} \cong \Alg_{\textbf{B}^{\mathcal{P}/}}(\cS)$$
	between the category of $\P$-bimodules under $\P$ and the category of algebras over $\textbf{B}^{\mathcal{P}/}$.
\end{prop}

\begin{proof}[Sketch of proof] (1) Let $M$ be a $\textbf{B}^{\mathcal{P}/}$-algebra. Note first that since the underlying category of $\textbf{B}^{\mathcal{P}/}$ agrees with $\textbf{Ib}^{\mathcal{P}}$, $M$ already inherits a canonical right $\mathcal{P}$-module structure (cf. Proposition \ref{rewriteinfinitesimal1}).

	Let us see how $M$ comes equipped with a left $\mathcal{P}$-action. For simplicity, we describe only the action maps of the form
	\begin{equation}\label{eq:leftactionmap}
		\mathcal{P}(c,d;e) \otimes M(c_1,\cdots,c_n;c) \otimes M(d_1,\cdots,d_m;d) \lrar M(c_1,\cdots,c_n,d_1,\cdots,d_m;e)
	\end{equation}
	To this end, observe first that the $\textbf{B}^{\mathcal{P}/}$-algebra structure map of $M$ of the form
	$$ \textbf{B}^{\mathcal{P}/}\left( \,    \left(c_1,\cdots,c_n;c \right) ,    \left(d_1,\cdots,d_m;d \right)   ;   \left(c_1,\cdots,c_n,d_1,\cdots,d_m;e \right) \, \right) \otimes $$
	$$ \otimes \; M(c_1,\cdots,c_n;c) \otimes M(d_1,\cdots,d_m;d) \lrar M(c_1,\cdots,c_n,d_1,\cdots,d_m;e) $$
	has a component of the form
	$$ \mathcal{P}(c,d;e) \otimes \mathcal{P}(c_1;c_1) \otimes \cdots \otimes \mathcal{P}(c_n;c_n) \otimes \mathcal{P}(d_1;d_1) \otimes \cdots \otimes \mathcal{P}(d_m;d_m) \otimes$$
	$$ \otimes \;  M(c_1,\cdots,c_n;c) \otimes M(d_1,\cdots,d_m;d) \lrar M(c_1,\cdots,c_n,d_1,\cdots,d_m;e) .$$
	The evaluation at the unit operations $\id_{c_1},\cdots,\id_{c_n},\id_{d_1},\cdots,\id_{d_m}$ of $\mathcal{P}$ to the latter gives us the action map \eqref{eq:leftactionmap} as desired.

	Moreover, the action of the nullary operations of $\textbf{B}^{\mathcal{P}/}$ on $M$ determines a canonical map $\P \rar M$. 
	
	\smallskip
	
	(2) Conversely, let $M$ be a $\mathcal{P}$-bimodule under $\P$. We let the composed map $$ [\I_C \lrar \P \lrar M] =: \varepsilon $$ exhibit the images of the unit operations of $\mathcal{P}$ in $M$. In order to construct a $\textbf{B}^{\mathcal{P}/}$-algebra structure on $M$, one will need to make use of the $\mathcal{P}$-bimodule structure of $M$ and the map $\varepsilon$ in a suitable way.

\end{proof}

\begin{cons} The $\cS$-enriched operad which encodes the category of $\P$-bimodules will be denoted by $\textbf{B}^{\mathcal{P}}$. Its set of colors is again $\Seq(C)$. A typical space of $n$-ary operations is given by
	$$ \textbf{B}^{\mathcal{P}}\left( \,    \left(c_1,\cdots,c_{r_1}; c^{(1)} \right) ,    \left(c_{r_1+1},\cdots,c_{r_1+r_2}; c^{(2)} \right) , \cdots ,  \left(c_{r_1+\cdots+r_{n-1}+1},\cdots,c_{r_1+\cdots+r_{n}}; c^{(n)} \right)  ;   \left(d_1,\cdots,d_m;d \right) \, \right) $$
	$$ := \bigsqcup_{\underline{m} \x{f}{\lrar} \underline{r_1+\cdots+r_{n}}} \left [ \mathcal{P}\left (c^{(1)},\cdots,c^{(n)};d \right ) \otimes \bigotimes_{i=1,\cdots,r_1+\cdots+r_{n}} \mathcal{P} \left (\{d_j\}_{j\in f^{-1}(i)};c_i \right ) \right ] $$
	where the coproduct ranges over the hom-set $\Hom_{\Fin}(\underline{m} , \underline{r_1+\cdots+r_{n}})$. For each map $f : \underline{m} \lrar \underline{r_1+\cdots+r_{n}}$, we will denote by $\textbf{B}_f^{\mathcal{P}}(-)$ the component of $\textbf{B}^{\mathcal{P}}(-)$ corresponding to $f$, taken from the above formula. As in Construction \ref{operadencodebimodulesunderP}, the operad structure of $\textbf{B}^{\mathcal{P}}$ is canonically defined via the structure of $\P$, so that $\textbf{B}^{\mathcal{P}}$ is in fact a suboperad of $\textbf{B}^{\mathcal{P}/}$.  (See also [\cite{Julien}, \S 2.1.1]).
\end{cons}

As in the proof of Proposition \ref{bimoduleunderPsasalgebras}, {one can verify the following}.
\begin{prop}\label{p:bimodulesasalgebras} There is a canonical isomorphism 
	$$ \BMod(\mathcal{P}) \cong \Alg_{\textbf{B}^{\mathcal{P}}}(\cS)$$
	between the category of $\P$-bimodules and the category of algebras over $\textbf{B}^{\mathcal{P}}$.
\end{prop}

Finally, we construct an operad encoding the category of left $\P$-modules. 

\begin{cons} We denote by $\textbf{L}^{\P}$ the $\cS$-enriched operad whose set of colors is $\Seq(C)$ and whose spaces of operations are given as follows. For simplicity of equations, we just describe the spaces of 2-ary operations. These are concentrated in objects of the form:
	$$ \textbf{L}^{\P}\left( \,    \left(c_1,\cdots,c_{r_1}; c \right) ,    \left(c_{r_1+1},\cdots,c_{r_1+r_2}; c' \right)  ;   \left(c_{\sigma(1)},\cdots,c_{\sigma(r_1+r_2)};d \right) \, \right)  := \bigsqcup_{\alpha} \P(c,c';d) $$
	where $\sigma \in \Sigma_{r_1+r_2}$. The coproduct ranges over the subset of $\Sigma_{r_1+r_2}$ consisting of those $\alpha$ satisfying that for every $i\in \{1,\cdots,r_1+r_2\}$, the two colors $c_{\sigma(i)}$ and $c_{\alpha(i)}$ coincide. From the above formula, we denote by $\textbf{L}^{\P}_\alpha(-)$ the component of $\textbf{L}^{\P}(-)$ corresponding to $\alpha$. Observe that there is a canonical map 
	$$ \P(c,c';d) = \textbf{L}^{\P}_\alpha \left( \,    \left(c_1,\cdots,c_{r_1}; c \right) ,    \left(c_{r_1+1},\cdots,c_{r_1+r_2}; c' \right)  ;   \left(c_{\sigma(1)},\cdots,c_{\sigma(r_1+r_2)};d \right) \, \right)    \lrar $$
	$$  \textbf{B}^{\P}_\alpha \left( \,    \left(c_1,\cdots,c_{r_1}; c \right) ,    \left(c_{r_1+1},\cdots,c_{r_1+r_2}; c' \right)  ;   \left(c_{\sigma(1)},\cdots,c_{\sigma(r_1+r_2)};d \right) \, \right) = \P(c_1;c_1) \otimes \cdots \otimes \P(c_{r_1+r_2};c_{r_1+r_2}) \otimes \P(c,c';d)$$
	given by inserting the unit operation $\id_{c_i}$ into the factor $\P(c_i;c_i)$ for $i = 1,\cdots,r_1+r_2$. In this way, the operad structure of $\textbf{L}^{\P}$ is defined via the structure of $\P$, so that $\textbf{L}^{\P}$ forms a suboperad of $\textbf{B}^{\P}$.
\end{cons}

\begin{prop}\label{p:leftmodulesasalgebras} There is a canonical isomorphism 
	$$ \LMod(\P) \cong \Alg_{\textbf{L}^{\P}}(\cS) $$
	between the category of left $\P$-modules and the category of algebras over $\textbf{L}^{\P}$.
\end{prop}
\begin{proof} The proof is similar to the  ones above. 
\end{proof}

\smallskip

\subsection{Operadic transferred model structures}\label{s:optransfermod}

\smallskip

In this subsection, we assume further that $\SS$ is a \textbf{symmetric monoidal model category} (cf. Hovey's \cite{Hovey}). Let $\P$ be a $C$-colored operad in $\cS$ and let $A$ be a $\P$-algebra. We list here all the mentioned operadic categories, except the category of $\cS$-enriched operads $\Op(\cS)$, including 
\begin{equation}\label{eq:thesetA}
	\{\Coll_C(\SS), \Op_C(\SS), \LMod(\P), \RMod(\P), \BMod(\P), \IbMod(\P), \Alg_\P(\cS), \Mod_\P^{A}\} =: \mathbb{A} .
\end{equation}

\begin{dfn} Let $\M$ be any of the categories in $\mathbb{A}$. The \textbf{transferred model structure} on $\M$ is the one whose weak equivalences (resp. fibrations) are precisely the levelwise weak equivalences (resp. fibrations).
\end{dfn}

We wish to set up several suitable conditions on the base category $\cS$ assuring the existence of the transferred model structure on every element of $\mathbb{A}$. As we have seen previously, each category $\M\in \mathbb{A}$ can be represented as the category of algebras over a certain operad (or category) (cf. Remarks \ref{r:refomcollection}, \ref{r:AmodPasfunctor}, Example \ref{ex:operadofoperads}, Propositions \ref{rewriteinfinitesimal1}, \ref{p:rewriterightmodules}, \ref{p:bimodulesasalgebras}, \ref{p:leftmodulesasalgebras}).  Consequently, one just needs to consider the transferred model structure on the category $\Alg_\P(\cS)$. According to the literature, we know several criteria assuring the existence of that. Here are several settings.

\begin{dfn} A \textbf{symmetric monoidal fibrant replacement functor} on $\cS$ is a symmetric monoidal functor $\RR : \cS\rar\cS$ together with a monoidal natural transformation $ \varphi : \Id \rar \RR$ such that for each object $X\in \cS$, the map $\varphi_X : X \lrar \RR(X)$ exhibits $\RR(X)$ as a fibrant replacement of $X$. 
\end{dfn}

\begin{dfn} A \textbf{functorial path data} on $\cS$ is a symmetric monoidal functor $\PP : \cS\rar\cS$ together with monoidal natural transformations $s : \Id \rar \PP$  and $d_0,d_1 : \PP \rar \Id$ such that the composed map $$ X \x{s_X}{\lrar } \PP(X) \x{(d_0,d_1)}{\lrar} X\times X $$ exhibits $\PP(X)$ as a path object for $X$.
\end{dfn}

The following statement is due to [\cite{JY}, Theorem 3.11].
\begin{prop}\label{p:modelonalgebras} Suppose that $\SS$ is strongly cofibrantly generated (i.e., cofibrantly generated with domains of the generating cofibrations and trivial cofibrations being small). If $\SS$ admits both a symmetric monoidal fibrant replacement functor $\RR$ and a functorial path data $\PP$ then the transferred model structure on $\Alg_\P(\SS)$ exists for every operad $\P$. 
\end{prop}

There is another powerful criterion, thanks to the recent work of Pavlov-Scholbach. Here are several settings. Recall from [\cite{Baber}, Definition 1.1] that a map $f : X \rar Y$ in $\cS$ is an $h$-\textbf{cofibration} if and only if for every diagram of coCartesian squares in $\cS$ of the form 
$$ \xymatrix{
	X \ar[r]\ar_f[d] & A \ar^{g}[r]\ar[d] & B  \ar[d] \\
	Y \ar[r] & A' \ar_{g'}[r] & B'  \\
}$$  
the map $g'$ is a weak equivalence whenever $g$ is one. If $f$ is in addition a weak equivalence then it is an \textbf{acyclic} $h$-\textbf{cofibration}.

Let $\overrightarrow{n} = (n_1,\cdots,n_k)$ be a finite sequence of natural numbers. For a family $s=(s_1,\cdots,s_k)$ of maps in $\cS$, denote by $$s^{\square \overrightarrow{n}}:= \square_i \, s_i^{\square \, n_i} ,$$ where the subcript $``\square$'' refers to the \textbf{pushout-product} of maps. The group $\Sigma_{\overrightarrow{n}} := \bigsqcap_{i} \Sigma_{n_i}$ acts on $s^{\square \overrightarrow{n}}$ in an evident way. 

\begin{dfn} The symmetric monoidal model category $\cS$ is said to be \textbf{symmetric} $h$-\textbf{monoidal} if for any finite family $s=(s_1,\cdots,s_k)$ of cofibrations (resp. acyclic cofibrations) and for any object $X \in \cS$ equipped with a right $\Sigma_{\overrightarrow{n}}$-action, the map $X \otimes_{\Sigma_{\overrightarrow{n}}}s^{\square \overrightarrow{n}}$ is an $h$-cofibration (resp. acyclic $h$-cofibration).
\end{dfn}

\begin{prop}\label{p:modelonalgebras1} (Pavlov-Scholbach,  [\cite{Pavlov}, Theorem 5.10]) Suppose that $\SS$ is a combinatorial symmetric monoidal model category such that weak equivalences are closed under transfinite compositions. If $\cS$ is symmetric $h$-monoidal (the acyclic part is sufficient), then the transferred model structure on $\Alg_\P(\SS)$ exists for every operad $\P$.
\end{prop}

We end this subsection by listing some base categories of interest and discussing how they adapt to the criteria mentioned above. 

\begin{examples}\label{ex:basecategories}

		(i) The base category of most interest is the Cartesian monoidal category of \textbf{simplicial sets}, $(\Set_\Delta , \times)$, equipped with the standard (Kan-Quillen) model structure.  The model structure on $\Set_\Delta$ is combinatorial, with weak equivalences closed under filtered colimits. It admits a fibrant replacement functor given by $\Ex^{\infty} := \Sing | - | $ the composition of the realization and singular functors, and admits a functorial path data given by $(-)^{\Delta^{1}}$. Moreover, $\Set_\Delta$ is as well symmetric $h$-monoidal, according to [\cite{Pavlov1}, \S 7.1]. We hence get that $\Set_\Delta$ satisfies the conditions of both {the} propositions \ref{p:modelonalgebras} and \ref{p:modelonalgebras1}.
		
		\smallskip
		
		(ii) The second one is the monoidal category of \textbf{simplicial} $R$-\textbf{modules}, $(\sMod_{R} , \otimes)$, with $R$ being a commutative ring, equipped with the standard model structure transferred from that of $\Set_\Delta$.  As well as simplicial sets, simplicial $R$-modules satisfy the conditions of both {the} propositions \ref{p:modelonalgebras} and \ref{p:modelonalgebras1}. Indeed, note first that $\sMod_{R}$ is combinatorial, has weak equivalences being closed under filtered colimits and moreover, it is a fibrant model category (i.e., all the objects are fibrant). A functorial path data for $\sMod_{R}$ is given by $(-)^{R\{\Delta^{1}\}}$. It is as well symmetric $h$-monoidal, according to [\cite{Pavlov1}, \S 7.3]. 
		
		\smallskip
		
		(iii) Let $\textbf{k}$ be a commutative ring {that contains $\QQ$ the field of rational numbers}. Consider the monoidal category of \textbf{dg} $\textbf{k}$-\textbf{modules}, $(\C(\textbf{k}) , \otimes)$, equipped with the projective model structure. This is also a combinatorial fibrant model category, with weak equivalences closed under filtered colimits. There is a functorial path data for $\C(\textbf{k})$ given by $(-) \otimes \Omega^{*}(\Delta^{1})$  where $\Omega^{*}(\Delta^{1})$ is the \textit{Sullivan's dg algebra of differentials} on the interval $\Delta^{1}\in \Set_\Delta$ (see [\cite{Fresse2}, 5.3]). So we get that $\C(\textbf{k})$ satisfies the conditions of Proposition \ref{p:modelonalgebras}. {Moreover, $\C(\textbf{k})$ is symmetric $h$-monoidal, according to \cite{Pavlov1}}.
		
		\smallskip
		
		(iv) {For $\textbf{k}$ as above}, consider the monoidal category of \textbf{connective dg $\textbf{k}$-modules}, $(\C_{\geqslant0}(\textbf{k}),\otimes)$, equipped with the projective model structure. This model category is combinatorial and its weak equivalences are closed under filtered colimits. Moreover, all its objects are fibrant. However, as far as we know, $\C_{\geqslant0}(\textbf{k})$ does not adapt to the conditions of Proposition \ref{p:modelonalgebras}. The only thing missed is the existence of a functorial path data. {Despite this, $\C_{\geqslant0}(\textbf{k})$ is symmetric $h$-monoidal, as well as the case above.}
		
		\smallskip
		
		(v) More generally, let $R$ be a commutative monoid in $(\C_{\geqslant0}(\textbf{k}) , \otimes)$ with $\textbf{k}$ being {a commutative ring containing $\QQ$}. Consider the monoidal category of $R$\textbf{-modules,} $(\Mod_R , -\otimes_R-)$, equipped with the projective model structure. This is a combinatorial fibrant model category, with weak equivalences closed under filtered colimits. As well as $\C_{\geqslant0}(\textbf{k})$, the category $\Mod_R$ does not admit a functorial path data, yet it satisfies the conditions of Proposition \ref{p:modelonalgebras1} instead. To see the latter, one just needs to verify the symmetric $h$-monoidality. This is in fact transferred from the symmetric $h$-monoidality of $\C_{\geqslant0}(\textbf{k})$ (cf. [\cite{Pavlov1}, Theorem 5.9]).

\end{examples}

In conclusion, we see that all the base categories listed above are nice enough so that the transferred model structure on $\P$-algebras exists for every operad $\P$. Thus for every category $\M\in \mathbb{A}$ \eqref{eq:thesetA}, the transferred model structure on $\M$  exists as well.

\smallskip

\subsection{Dwyer-Kan and canonical model structures on enriched operads}\label{sub:DKcanoical}

\medskip

Let $\cS$  be a \textbf{monoidal model category} and let $\Cat(\cS)$ denote the category of (small)  $\cS$-enriched categories. For each $\C \in \Cat(\cS)$, the \textbf{homotopy category of} $\C$, denoted by $\Ho(\C)$, is the ordinary category whose objects are the same as those of $\C$ and whose hom-sets are given by  $$\Hom_{\Ho(\C)}(x,y) := \Hom_{\Ho(\cS)}(1_\cS,\Map_\C(x,y)) .$$
By convention, a map $f:\C\rar \D$ in $\Cat(\cS)$ is a \textit{levelwise weak equivalence} (resp. \textit{fibration, trivial fibration}, etc) if for every $x,y \in \Ob(\C)$ the  map $\Map_\C(x,y) \lrar \Map_\D(f(x),f(y))$ is a weak equivalence (resp. fibration, trivial fibration, etc) in $\cS$. 

\begin{dfn}\label{d:DKcat} A map $f:\C\rar \D$ in $\Cat(\cS)$ is called a \textbf{Dwyer-Kan equivalence} if it is a levelwise weak equivalence and such that the induced functor $\Ho(f):\Ho(\C)\lrar\Ho(\D)$ between homotopy categories is essentially surjective.
\end{dfn}

\begin{dfn} The \textbf{Dwyer-Kan model structure on} $\Cat(\cS)$ is the one whose weak equivalences are the Dwyer-Kan equivalences and whose trivial fibrations are the levelwise trivial fibrations surjective on objects. (See, e.g., \cite{Muro}, \cite{Luriehtt}).
\end{dfn}

\begin{dfn}(Berger-Moerdijk \cite{Ieke}) The \textbf{canonical model structure on} $\Cat(\cS)$ is the one whose fibrant objects are the levelwise fibrant categories and whose trivial fibrations are the same as those of the Dwyer-Kan model structure. 
\end{dfn}

By extending the two above, G. Caviglia \cite{Caviglia} established both the Dwyer-Kan and canonical model structures on $\Op(\cS)$. Suppose further that $\cS$ is a symmetric monoidal model category. By convention, a map $f : \P \rar \Q$ in $\Op(\cS)$ is called a \textit{levelwise weak equivalence} (resp. \textit{fibration, trivial fibration}, etc) if for every sequence $(c_1,\cdots,c_n;c)$ of colors of $\P$, the induced map $$\P(c_1,\cdots,c_n;c) \lrar \Q(f(c_1),\cdots,f(c_n) ; f(c))$$ is a weak equivalence (resp. fibration, trivial fibration, etc) in $\cS$.

The \textbf{homotopy category of} $\P$ is defined to be $\Ho(\P) := \Ho(\P_1)$ the homotopy category of its underlying category. 

\begin{dfn}\label{d:DKop} A map $f : \P \rar \Q$ in $\Op(\cS)$ is called a \textbf{Dwyer-Kan equivalence} if it is a levelwise weak equivalence and such that the induced functor $\Ho(f):\Ho(\P)\lrar\Ho(\Q)$ between homotopy categories is essentially surjective or alternatively, if $f$ is a levelwise weak equivalence and the underlying map $f_1 : \P_1 \lrar \Q_1$ is a Dwyer-Kan equivalence in $\Cat(\cS)$.
\end{dfn}

\begin{dfn}\label{d:DKope} The \textbf{Dwyer-Kan model structure on} $\Op(\cS)$ is the one whose weak equivalences are the Dwyer-Kan equivalences and whose trivial fibrations are the levelwise trivial fibrations surjective on colors.
\end{dfn}

\begin{dfn}\label{d:canoope} The \textbf{canonical model structure on} $\Op(\cS)$ is the one whose fibrant objects are the levelwise fibrant operads and whose trivial fibrations are the same as those of the Dwyer-Kan model structure.
\end{dfn}

\begin{rem}\label{r:canonicalfib} As originally introduced by G. Caviglia, a map in $\Op(\cS)$ is a fibration (resp. weak equivalence) with respect to the canonical model structure if and only if it is a levelwise fibration (resp. weak equivalence) and such that its underlying map in $\Cat(\cS)$ is a fibration (resp. weak equivalence) with respect to that model structure. (See [\cite{Caviglia}, Definition 4.5] and [\cite{Ieke}, \S 2.2]).
\end{rem}

Following up his work, we give a set of conditions on the base category $\cS$ assuring the existence of the canonical model structure.
\begin{prop}\label{firstfiveconditions}(Caviglia, \cite{Caviglia}) Let $\mathcal{S}$ be a combinatorial symmetric monoidal model category {such that}:

	(S1) the class of weak equivalences is closed under filtered colimits,

	(S2) either (a) $\cS$ admits a symmetric monoidal fibrant replacement functor and a functorial path data, or (b) $\cS$ is symmetric $h$-monoidal,

	(S3) the monoidal unit is cofibrant, and

	(S4) the model structure is right proper.

	\noindent Then $\Op(\mathcal{S})$ admits the canonical model structure, which is as well right proper and combinatorial. Moreover, this model structure coincides then with the Dwyer-Kan model structure.
	\begin{proof} The combinatoriality of $\cS$ implies that it is strongly cofibrantly generated and besides that, implies the existence of a set of \textbf{generating intervals} in the sense of \cite{Ieke}  (cf. Lemma 1.12 of {loc. cit.}). On other hand, by propositions \ref{p:modelonalgebras} and \ref{p:modelonalgebras1}, the condition (S2) ensures the existence of  the transferred model structure on $\Op_C(\mathcal{S})$ for every set $C$. Then by [\cite{Caviglia}, Theorem 4.22 (1)], $\Op(\mathcal{S})$ admits the canonical model structure, which is combinatorial as well. The right properness follows by Proposition 5.3 of  {loc. cit}.

		For the second claim, observe first that the condition (S1), together with the combinatoriality of $\cS$, implies that $\cS$ is \textbf{compactly generated} in the sense of [\cite{Ieke}, Definition 1.2]. Combining this fact with (S2), we get that $\cS$ is \textbf{adequate} in the sense of [\cite{Ieke}, Definition 1.1]. The latter fact, along with the conditions (S3) and (S4), proves that the classes of Dwyer-Kan and canonical weak equivalences in $\Cat(\mathcal{S})$ coincide (cf. \cite{Ieke}, propositions 2.20 and 2.24). Thus, by Remark \ref{r:canonicalfib} these two classes in $\Op(\mathcal{S})$ coincide as well. Combining the latter with the fact that the two model structures have the same trivial fibrations, we deduce that they indeed coincide.
	\end{proof}
\end{prop}

\begin{rem} Under the same assumptions as in Proposition \ref{firstfiveconditions}, the canonical model structure on $\Cat(\mathcal{S})$ automatically exists and coincides with the Dwyer-Kan model structure.
\end{rem}

\begin{example} Some typical base categories satisfying the conditions of Proposition \ref{firstfiveconditions} are {presented in} Examples \ref{ex:basecategories}.
\end{example}

\smallskip

\subsection{Tangent categories and Quillen cohomology}\label{s:tangentcategory}

\smallskip

This subsection, based on the works of \cite{YonatanBundle,YonatanCotangent}, contains the most important concepts appearing throughout the paper. Basically, \textit{tangent category} comes after a procedure of taking the stabilization of  a model category of interest. Note that under our setting, stabilizations exist only as semi model categories. Despite this, the needed results from those papers remain valid. 

\begin{dfn} A model category $\textbf{M}$ is said to be \textbf{weakly pointed} if it contains a \textbf{weak zero object}, i.e., an object which is both homotopy initial and terminal.
\end{dfn}
Let $\textbf{M}$ be a weakly pointed model category and let $X$ be an ($\mathbb{N}\times\mathbb{N} $)-diagram in $\textbf{M}$. The \textit{diagonal squares} of $X$ are of the form
$$ \xymatrix{
	X_{n,n} \ar[r]\ar[d] & X_{n,n+1} \ar[d] \\
	X_{n+1,n} \ar[r] & X_{n+1,n+1} \\
}$$
\begin{dfn}\label{dfnspectrumobj} An  ($\mathbb{N}\times\mathbb{N} $)-diagram in $\textbf{M}$ is called

	 (1) a \textbf{prespectrum} if all its off-diagonal entries are weak zero objects in $\textbf{M}$,

	 (2) an \textbf{$\Omega$-spectrum} if it is a prespectrum and all its diagonal squares are homotopy Cartesian,

	 (3) a \textbf{suspension spectrum} if it is a prespectrum and all its diagonal squares are homotopy coCartesian.
\end{dfn}

The projective model category of ($\mathbb{N}\times\mathbb{N} $)-diagrams in $\textbf{M}$ will be denoted by $\textbf{M}^{\mathbb{N}\times\mathbb{N}}_{\proj}$.

\begin{dfn}(\cite{YonatanBundle}) Let $\textbf{M}$ be a weakly pointed model category. A map $f : X\rar Y$ in $\textbf{M}^{\mathbb{N}\times\mathbb{N}}$ is said to be a \textbf{stable equivalence} if for every $\Omega$-spectrum $Z$ the induced map between derived mapping spaces
	$$ \Map^{\der}_{\textbf{M}^{\mathbb{N}\times\mathbb{N}}_{\proj}}(Y,Z) \lrar  \Map^{\der}_{\textbf{M}^{\mathbb{N}\times\mathbb{N}}_{\proj}}(X,Z) $$
	is a {weak equivalence}. (Note that a stable equivalence between $\Omega$-spectra is always a levelwise {weak}  equivalence).
\end{dfn}

Following [\cite{YonatanBundle}, Lemma 2.1.6], the $\Omega$-spectra in $\textbf{M}$ can be characterized as the \textbf{local objects} against a certain set of maps. Inspired by Definition 2.1.3 of  {loc. cit.}, we give the following definition, which is valid due to \cite{White}.

\begin{dfn}\label{d:stabilization} Let $\textbf{M}$ be a weakly pointed combinatorial model category such that the domains of generating cofibrations are cofibrant. The \textbf{stabilization of}  $\textbf{M}$, denoted by $\Sp(\textbf{M})$, is defined to be the left Bousfield localization of $\textbf{M}^{\mathbb{N}\times\mathbb{N}}_{\proj}$ with $\Omega$-spectra as the local objects. Explicitly, $\Sp(\textbf{M})$ is a cofibrantly generated semi model category whose

	- weak equivalences are the stable equivalences, and whose

	- (generating) cofibrations are the same as those of $\textbf{M}^{\mathbb{N}\times\mathbb{N}}_{\proj}$.

	\noindent In particular, fibrant objects of $\Sp(\textbf{M})$ are precisely the levelwise fibrant $\Omega$-spectra.
\end{dfn}

\begin{rem}\label{r:semitangent} When $\textbf{M}$ is in addition left proper then the stabilization $\Sp(\textbf{M})$ exists as a (full) model category. Nevertheless, we do not require the left properness throughout the paper. 
\end{rem}

\begin{dfn}\label{d:stable} (\cite{YonatanBundle}) A model {(or semi-model)} category $\textbf{M}$ is {said to be} \textbf{stable} if the following equivalent conditions hold:
	
		(1) The \textbf{underlying} $\infty$-\textbf{category} $\textbf{M}_\infty$ (cf., \cite{DK, Hinich}) is stable in the sense of \cite{Lurieha}.
		
		(2) $\textbf{M}$ is weakly pointed and such that a square in $\textbf{M}$ is homotopy coCartesian if and only if it is homotopy Cartesian.
		
		(3) $\textbf{M}$ is weakly pointed and such that the adjunction $\Sigma : \adjunction*{}{\Ho(\textbf{M})}{\Ho(\textbf{M})}{}  : \Omega$ of suspension-desuspension functors is an adjoint equivalence.
	
\end{dfn}

\begin{thm}\phantomsection \label{t:stabilization} (Y. Harpaz, J. Nuiten and M. Prasma \cite{YonatanBundle}) Let $\textbf{M}$ and $\textbf{N}$ be two weakly pointed combinatorial model categories such that the domains of their generating cofibrations are cofibrant. 
\begin{enumerate}

		\item There is a Quillen adjunction $\adjunction{\Sigma^{\infty}}{\textbf{M}}{\Sp(\textbf{M})}{\Omega^{\infty}}$ where $\Omega^{\infty}(X)=X_{0,0}$ and $\Sigma^{\infty}({A})$ is the constant diagram with value {$A$}. 
		
		\item The induced functor $(\Omega^{\infty})_\infty : \Sp(\textbf{M})_\infty \lrar \textbf{M}_\infty$ exhibits $\Sp(\textbf{M})_\infty$ as the stabilization of $\textbf{M}_\infty$ in the  sense of \cite{Lurieha}.
		
		\item The stabilization $\Sp(\textbf{M})$ is stable. Furthermore, if $\textbf{M}$ is already stable then the adjunction $\adjunction{\Sigma^{\infty}}{\textbf{M}}{\Sp(\textbf{M})}{\Omega^{\infty}}$ is a Quillen equivalence.
		
		\item A Quillen adjunction $\adjunction{\sF}{\textbf{M}}{\textbf{N}}{\sG}$ lifts to a Quillen adjunction  $$\adjunction{\Sp(\sF)}{\Sp( \textbf{M})}{\Sp( \textbf{N})}{\Sp(\sG)}$$ which is given on underlying categories by the adjunction $\sF^{\mathbb{N}\times\mathbb{N}} \dashv \sG^{\mathbb{N}\times\mathbb{N}}$. Moreover, if $\sF \dashv \sG$ is a Quillen equivalence then $\Sp(\sF) \dashv \Sp(\sG)$ is one. (We will sometimes write $\sF^{\Sp}$ instead $\Sp(\sF)$).
\end{enumerate}
\end{thm}

Let $\I$ be a diagram category. Then the left Quillen functor $\sF : \textbf{M} \lrar \textbf{N}$ gives rise to a left Quillen functor $$[\I,\Sp(\sF)] : \Fun(\I,\Sp(\textbf{M})) \lrar \Fun(\I,\Sp(\textbf{N}))$$ between projective model categories given by {post-composing} with $\Sp(\sF) :\Sp(\textbf{M}) \lrar \Sp(\textbf{N})$.

\begin{obs}\label{ob:stablecohomotopy} Suppose that $\sF$ preserves weak equivalences. Then $\Sp(\sF)$ also preserves weak equivalences and therefore, so does the functor $[\I,\Sp(\sF)]$.
\end{obs}
\begin{proof} Since $\sF$ preserves weak equivalences, so does the functor $\sF^{\mathbb{N}\times\mathbb{N}} : \textbf{M}^{\mathbb{N}\times\mathbb{N}}_{\proj} \lrar \textbf{N}^{\mathbb{N}\times\mathbb{N}}_{\proj}$. The proof is then straightforward using the definition of stable equivalences.
\end{proof}

\begin{notns}\label{no:aumanted}

		 (i) Let $\C$ be a category containing a terminal object $*$. The \textbf{pointed category} associated to $\C$ is given by $\C_* := \C_{*/}$ the category of objects under $*$.

		 (ii) Suppose that $\C$ contains an initial object $\emptyset$. The \textbf{augmented category} associated to $\C$ is given by $\C^{\aug} := \C_{/\emptyset}$ the category of objects over $\emptyset$.
		
		 (iii) Let $\C$ be a category containing an object $X$. We will denote by $\C_{X//X}:= (\C_{/X})_*$ the pointed category associated to the over category $\C_{/X}$. More explicitly, objects of $\C_{X//X}$ are the diagrams $X\overset{f}{\longrightarrow} A \overset{g}{\longrightarrow} X$ in $\C$ such that $gf=\Id_X$. Alternatively, $\C_{X//X}$ is the same as $(\C_{X/})^{\aug}$ the augmented category associated to the under category $\C_{X/}$.

\end{notns}

Note that a morphism $f:X\rar Y$ in $\C$ gives rise to a canonical adjunction $\adjunction{f_!}{\C_{X//X}}{\C_{Y//Y}}{f^{*}}$ defined by
\begin{gather*} f_!(X \rar A \rar X) = (Y \lrar A \, \underset{X}{\bigsqcup} \, Y \lrar Y)  \; \; \text{and} \\  f^{*}(Y \rar B \rar Y) = (X \lrar B\times_Y X \lrar X).
\end{gather*}

It can be shown that if $\textbf{M}$ is a combinatorial model category such that the domains of generating cofibrations are cofibrant, then so is the transferred model structure on $\textbf{M}_{A//A}$ (see Hirschhorn's \cite{Hirschhorn2}). This makes the following definition valid.

\begin{dfn}\label{d:tangentcat} Let $\textbf{M}$ be a combinatorial model category such that the domains of generating cofibrations are cofibrant and let $A$ be an object of $\textbf{M}$. The \textbf{tangent category} to $\textbf{M}$ at $A$, denoted by $\mathcal{T}_A\textbf{M}$, is defined to be the stabilization of $\textbf{M}_{A//A}$:
\begin{equation}\label{eq:deftan}
	\mathcal{T}_A\textbf{M}:=\Sp(\textbf{M}_{A//A}).
\end{equation}	
\end{dfn}

There is a Quillen adjunction  $ \adjunction{\Sigma^{\infty}_+}{\textbf{M}_{/A}}{\mathcal{T}_A\textbf{M}}{\Omega^{\infty}_+} $ given by the composition
$$ \adjunction*{A\sqcup (-)}{\textbf{M}_{/A}}{\textbf{M}_{A//A}}{\forget} \adjunction*{\Sigma^{\infty}}{}{\mathcal{T}_A\textbf{M}}{\Omega^{\infty}} .$$
Namely, for each $B \in \textbf{M}_{/A}$, then we have $ \Sigma^{\infty}_+(B) = \Sigma^{\infty}(A\longrightarrow A\sqcup B \longrightarrow A)$, while for each $X \in \mathcal{T}_A\textbf{M}$, we have $\Omega^{\infty}_+(X) = (X_{0,0} \lrar A)$.

The following notion keeps a central role {within the framework of Quillen cohomology theory}.

\begin{dfn} Let $\textbf{M}$ be a combinatorial model category such that the domains of generating cofibrations are cofibrant and let $A$ be an object of $\textbf{M}$. The \textbf{cotangent complex} of $A$, denoted by $\rL_A$,  is defined to be the derived suspension spectrum of $A$, i.e., \, $ \rL_A:=\mathbb{L}\Sigma^{\infty}_+(A) \in  \mathcal{T}_A\textbf{M}$.
\end{dfn}

By Theorem \ref{t:stabilization}(iv), a given map $f:A\rightarrow B$ in $\textbf{M}$ gives rise to a Quillen adjunction between tangent categories $$\adjunction{f_!^{\Sp}}{\mathcal{T}_A\textbf{M}=\Sp(\textbf{M}_{A//A})}{\Sp(\textbf{M}_{B//B})= \mathcal{T}_B\textbf{M}}{f^*_{\Sp}} .$$ Moreover, there is a commutative square of left Quillen functors
\begin{equation}\label{eq:Qadj}
	\xymatrix{
		\textbf{M}_{/A}  \ar[r]^{f_!}\ar[d]_{\Sigma^{\infty}_+} & \textbf{M}_{/B} \ar[d]^{\Sigma^{\infty}_+} \\
		\ar[r]_{f_!^{Sp}} \mathcal{T}_A\textbf{M} & \mathcal{T}_B\textbf{M}  \\
	}
\end{equation}

\begin{rem}\label{r:tangentright} {For a weak equivalence $A \lrarsimeq B$ between fibrant-cofibrant objects in $\textbf{M}$, one can show that the induced adjunction $\adjunction*{}{\textbf{M}_{A//A}}{\textbf{M}_{B//B}}{}$ is a Quillen equivalence. Accordingly, the tangent categories to $\textbf{M}$ are homotopy-invariant with respect to the class of such objects. In other words, the tangent category $\mathcal{T}_A\textbf{M}$ has the right type as soon as $A$ is fibrant-cofibrant. Note also that the requirement for fibrancy (resp. cofibrancy) on $A$ can be omitted when $\textbf{M}$ is in addition right (resp. left)  proper. This follows from the fact that, when $\textbf{M}$ is right (resp. left)  proper,  the model category $\textbf{M}_{/X}$ (resp. $\textbf{M}_{X/}$) has already the right type for every object $X$.}
\end{rem}

\begin{dfn}\label{d:relcotangent} Let $\textbf{M}$ be a combinatorial model category such that the domains of generating cofibrations are cofibrant and let $f : A \rar B$ be a map in $\textbf{M}$. We will denote by $$\rL_{B/A}:=\hocofib \, [\,\mathbb{L}\Sigma^{\infty}_+(f) \lrar \rL_B\,]$$ the homotopy cofiber of the map $\mathbb{L}\Sigma^{\infty}_+(f) \lrar \rL_B$ in $\mathcal{T}_B\textbf{M}$, and refer to $\rL_{B/A}$ as the \textbf{relative cotangent complex} of $f$. 
\end{dfn}

Notice that the map $\mathbb{L}\Sigma^{\infty}_+(f) \lrar \rL_B$ can be identified with the map $f_!^{\Sp}(\rL_A)  \lrar \rL_B$, due to the commutativity of the square~\eqref{eq:Qadj}.

\begin{rem}\label{r:relcotant} Suppose that $A$ is cofibrant. Take a factorization $A \lrar B^{\cof} \lrarsimeq B$ in $\textbf{M}$ of the map $f$ into a cofibration followed by a weak equivalence. In particular, $B^{\cof}$ is a cofibrant resolution of $B$ in $\textbf{M}$. Consider the  coCartesian square 
	$$\xymatrix{
		B \sqcup A  \ar[r]\ar[d] & B \sqcup B^{\cof} \ar[d] \\
		\ar[r] B & B \, \underset{A}{\bigsqcup} \, B^{\cof}  \\
	}$$
	When regarded as a coCartesian square in $\textbf{M}_{B//B}$, this square is homotopy coCartesian because the map $A \lrar B^{\cof}$ is a cofibration between cofibrant objects in $\textbf{M}$ and in addition, $B$ is a zero object in $\textbf{M}_{B//B}$. Applying the functor $\Sigma^\infty : \textbf{M}_{B//B} \lrar  \mathcal{T}_B\textbf{M}$ to this square, we obtain a homotopy cofiber sequence in $\mathcal{T}_B\textbf{M}$ {of the form}
	$$ \Sigma^\infty(B \sqcup A) \lrar \Sigma^\infty(B \sqcup B^{\cof}) \lrar \Sigma^\infty(B \, \underset{A}{\bigsqcup} \, B^{\cof})  .$$
	In this sequence, the first term is a model for $\mathbb{L}\Sigma^{\infty}_+(f)$, while the second term is nothing but $\rL_B$. Thus by definition $\Sigma^\infty(B \, \underset{A}{\bigsqcup} \, B^{\cof})$ is a model for the relative cotangent complex $\rL_{B/A}$.
\end{rem}

\begin{dfn}\label{d:Qcohom} (\cite{YonatanCotangent}, Definition 2.2.1) Let $\textbf{M}$ be a combinatorial model category such that the domains of generating cofibrations are cofibrant  and let $X$ be a fibrant object of $\textbf{M}$. Suppose given an object $M\in\mathcal{T}_X\textbf{M}$, regarded as the \textbf{spectrum of coefficients}. For each $n\in\mathbb{Z}$, the \textbf{$n$'th Quillen cohomology group of $X$ with coefficients in $M$} is defined to be
	$$ \sH^{n}_Q (X  ; M) := \pi_0 \Map^{\der}_{\mathcal{T}_X\textbf{M}} (\rL_X , M[n])$$
	where $M[n] : = \Sigma^{n}M$, i.e., the $n$-suspension of $M$ in $\mathcal{T}_X\textbf{M}$.
\end{dfn}

\begin{rem}\label{r:rewriteQcohom} By the Quillen adjunction $ \adjunction{\Sigma^{\infty}_+}{\textbf{M}_{/X}}{\mathcal{T}_X\textbf{M}}{\Omega^{\infty}_+} $, there is a canonical weak equivalence
	$$ \Map^{\der}_{\textbf{M}_{/X}}(X, {\RR}\Omega^{\infty}_+M[n]) \simeq \Map^{\der}_{\mathcal{T}_X\textbf{M}} (\rL_X , M[n]) .$$
 In particular, we have that $$\sH^{n}_Q (X  ; M) = \pi_0 \Map^{\der}_{\mathcal{T}_X\textbf{M}} (\rL_X , M[n])\cong \pi_0\Map^{\der}_{\textbf{M}_{/X}}(X, {\RR}\Omega^{\infty}_+M[n]) .$$
\end{rem}

\begin{rem}\label{r:Qcohominvar} Quillen cohomology is {homotopy-invariant with respect to the class of fibrant objects.} Indeed, a weak equivalence $f:X \lrarsimeq Y$ between fibrant objects induces a right Quillen equivalence $\textbf{M}_{/Y}\lrarsimeq \textbf{M}_{/X}$. Therefore, for any fibrant object $M\in \mathcal{T}_Y\textbf{M}$ we get a canonical weak equivalence
	$$ \Map^{\der}_{\textbf{M}_{/Y}}(Y, \Omega^{\infty}_+M[n])  \lrarsimeq  \Map^{\der}_{\textbf{M}_{/X}}(X, f^{*}\Omega^{\infty}_+M[n])=\Map^{\der}_{\textbf{M}_{/X}}(X, \Omega^{\infty}_+(f_{\Sp}^{*}M)[n]).$$
	This induces an isomorphism $\sH^{n}_Q (Y  ; M)\overset{\cong}{\lrar}\sH^{n}_Q (X  ; f_{\Sp}^{*}M)$, by Remark~\ref{r:rewriteQcohom}. 
\end{rem}

The above remark explains the requirement for fibrancy on objects. Nonetheless, {as in Remark \ref{r:tangentright},}  this requirement can be disregarded when $\textbf{M}$ is in addition right proper.

\smallskip

\section{Conventions}\label{s:convention}

\smallskip

We first recall from \cite{YonatanBundle} the following definition, which itself is inspired by [\cite{Lurieha}, Definition 6.1.1.6].

\begin{dfn}\label{dfndifferentiable} A model category $\textbf{M}$ is said to be \textbf{differentiable} if the derived colimit functor $\mathbb{L} \colim:\textbf{M}^{\mathbb{N}}\longrightarrow \textbf{M}$ preserves finite homotopy limits. A Quillen adjunction $\mathcal{L} :  \adjunction*{}{\textbf{M}}{\textbf{N}}{} : \mathcal{R}$ is said to be \textbf{differentiable} if both $\textbf{M}$ and $\textbf{N}$ are differentiable and the right derived functor $\mathbb{R}\mathcal{R}$ preserves sequential homotopy colimits.
\end{dfn}

\begin{conv}\label{convention} { Let $\cS$ be a combinatorial symmetric monoidal model category such that the domains of generating cofibrations are cofibrant.} 
	
	\noindent  \;	{$\bullet$ We will say that  $\cS$ is \textbf{sufficient} if it satisfies the conditions (S1)-(S4) of Proposition \ref{firstfiveconditions} and in addition:}
	
	\begin{enumerate}[\indent {}]
		
		\item	{(S5) $\mathcal{S}$ is differentiable, and}
		
		\item	{(S6) the unit $1_\mathcal{S}$ is \textbf{homotopy compact} in the sense that the functor $\pi_0 \Map^{\h}_\mathcal{S}(1_\mathcal{S}, - )$ sends filtered homotopy colimits to colimits of sets.}
	\end{enumerate}

	\noindent  \;	{$\bullet$ Moreover, $\cS$ is said to be \textbf{abundant} if it is sufficient, and in addition:} 	
	
	\begin{enumerate}[\indent {}]	
		\item 	{(S7) $\cS$ satisfies the Lurie's \textbf{invertibility hypothesis} [\cite{Luriehtt}, Definition A.3.2.12].}
	\end{enumerate}
	
	\noindent  \;	{$\bullet$ We will say that $\cS$ is \textbf{stably sufficient} (resp. \textbf{stably abundant}) if it is sufficient (resp. abundant), and such that:}
	
	\begin{enumerate}[\indent {}]	
		\item 	{(S8) $\cS$ is a stable model category containing a strict zero object.}
	\end{enumerate}  
\end{conv}

 In particular, {when $\cS$ is sufficient}, we will work on the canonical model structure on $\Op(\mathcal{S})$  (and also, $\Cat(\mathcal{S})$), which coincides with the Dwyer-Kan model structure by Proposition \ref{firstfiveconditions}. On other hand, for any category 
$$  \M \, \in \, \{\Coll_C(\SS), \Op_C(\SS), \LMod(\P), \RMod(\P), \BMod(\P), \IbMod(\P), \Alg_\P(\cS), \Mod_\P^{A}\} ,$$
we will work on the  transferred model structure on $\M$ (see  $\S$\ref{s:optransfermod}).

\begin{rem} Requiring the domains of generating cofibrations to be cofibrant is necessary for the existence of various types of operadic tangent category (cf. $\S$\ref{s:tangentcategory}).
\end{rem}

\begin{rem}\label{categoriesisdifferentiavle} The conditions (S5)-(S6) support the differentiability of  $\Op(\mathcal{S})$ and $\Op_C(\mathcal{S})$ (cf. Lemma \ref{LRisdifferentible}), which is needed for proving Proposition \ref{Quilleneq1}.
\end{rem}

\begin{rem}\label{r:s8} The condition (S7) allows us to inherit [\cite{YonatanCotangent}, Proposition 3.2.1] for the work of $\S$\ref{s:cotangentcplx}. Roughly speaking, the invertibility hypothesis requires that, for any $\C \in \Cat(\mathcal{S})$ containing a morphism $f$, localizing $\C$ at $f$ does not change the homotopy type of $\C$ as long as $f$ is already an isomorphism in $\Ho(\C)$. This condition is in fact {quite common in practice}.  According to \cite{Lawson}, if $\cS$ is a combinatorial monoidal model category satisfying (S1) and such that every object is cofibrant then $\cS$ satisfies the invertibility hypothesis. It also holds for dg modules over a commutative ring by [\cite{Toen}, Corollary 8.7], and for any simplicial monoidal model category, according to [\cite{Iandundas}, Theorem 0.9].
\end{rem}

\begin{example}\label{e:typicalconvention}  {While all the base categories mentioned in Examples \ref{ex:basecategories} are abundant, only $\C(\textbf{k})$ is stably abundant.}
\end{example}

\section{Operadic tangent categories}\label{s:optangent}

\smallskip

{Throughout this section, we will assume that $\cS$ is a  base category that is (at least) \textit{sufficient} in the sense of Conventions \ref{convention}. Moreover, let $\P$ be a fixed $C$-colored operad in $\cS$.}

Note that, as well as every type of monoid, there is a \textit{restriction functor} $\Op_C(\mathcal{S})_{\mathcal{P}/} \lrar \BMod(\mathcal{P})_{\mathcal{P}/}$, which admits a left adjoint called the \textit{induction functor}. Combined with \eqref{eq:L_PandR_P} and \eqref{eq:adjbimodtoinf}, this yields a sequence of adjunctions
\begin{equation}\label{eq:presequenceQuillen}
	\adjunction*{}{\IbMod(\P)_{\mathcal{P}/}}{\BMod(\P)_{\mathcal{P}/}}{} \adjunction*{}{}{\Op_C(\mathcal{S})_{\mathcal{P}/}}{} \adjunction*{}{}{\Op(\mathcal{S})_{\mathcal{P}/}}{} .
\end{equation}
The latter induces a sequence of adjunctions connecting the associated augmented categories 
$$ \adjunction*{}{\IbMod(\P)_{\P//\P}}{\BMod(\P)_{\P//\P}}{} \adjunction*{}{}{\Op_C(\mathcal{S})_{\P//\P}}{} \adjunction*{}{}{\Op(\mathcal{S})_{\P//\P}}{} .$$
Observe that each of the right adjoints in this sequence preserves fibrations and weak equivalences. So  all the adjunctions {above} are Quillen adjunctions. We thus obtain a sequence of Quillen adjunctions connecting the associated tangent categories (cf. Theorem \ref{t:stabilization}(iv)):
\begin{equation}\label{eq:sequenceQuillen}
	\adjunction*{}{\T_\P\IbMod(\P)}{\T_\P\BMod(\P)}{} \adjunction*{}{}{\T_\P\Op_C(\mathcal{S})}{} \adjunction*{}{}{\T_\P\Op(\mathcal{S})}{} \, .
\end{equation}

Recall that the operad $\P \in \Op_C(\cS)$ is said to be $\Sigma$\textbf{-cofibrant} if its underlying $C$-collection is cofibrant as an object of $\Coll_C(\cS)$.  Our goal in this section is to prove that the adjunctions in the above sequence are all Quillen equivalences when $\P$ is $\Sigma$-cofibrant. The work may require making use of the \textbf{Comparison theorem} \cite{Yonatan}, which we now recall.

Let $\textbf{M}$ be a symmetric monoidal model category and  let $\mathcal{O}$ be an $\textbf{M}$-enriched operad. We denote by $\mathcal{O}_{\leqslant 1}$ the operad obtained from $\O$ by removing the operations of arity $>1$. Recall that the collection of nullary operations $\mathcal{O}_0$ inherits the obvious structure of an $\mathcal{O}$-algebra and then, becomes an initial object in the category $\Alg_\mathcal{O}(\textbf{M})$. 

\begin{dfn}The operad $\mathcal{O}$ is said to be \textbf{admissible} if the transferred model structure on $\Alg_\mathcal{O}(\textbf{M})$ exists. Furthermore, $\O$ is called \textbf{stably admissible} (resp. {\textbf{stably semi admissible}}) if it is admissible and the stabilization  $\Sp(\Alg^{\aug}_\mathcal{O}(\textbf{M}))$ exists as a model (resp. {semi-model}) category, where $\Alg^{\aug}_\mathcal{O}(\textbf{M}):=\Alg_\mathcal{O}(\textbf{M})_{\mathcal{O}_0//\mathcal{O}_0}$ the augmented category associated to the category $\Alg_\mathcal{O}(\textbf{M})= \Alg_\mathcal{O}(\textbf{M})_{\mathcal{O}_0/}$.
\end{dfn}

Note that there is a canonical isomorphism $ \Alg_{\,\mathcal{O}_{\leqslant1}}(\textbf{M}) \cong \Alg_{\,\mathcal{O}_{1}}(\textbf{M})_{\mathcal{O}_0/}$. Now the inclusion of operads $\varphi : \mathcal{O}_{\leqslant1}\lrar \mathcal{O}$ induces a Quillen adjunction
$$ \adjunction{\varphi_!^{\aug}}{\Alg^{\aug}_{\,\mathcal{O}_{\leqslant1}}(\textbf{M})=\Alg_{\,\mathcal{O}_{1}}(\textbf{M})_{\mathcal{O}_0//\mathcal{O}_0}}{\Alg_\mathcal{O}(\textbf{M})_{\mathcal{O}_0//\mathcal{O}_0} =\Alg^{\aug}_\mathcal{O}(\textbf{M})}{\varphi^{*}_{\aug}} .$$
\begin{thm}\label{comparison} (\textbf{Comparison theorem}) [Y. Harpaz, J. Nuiten and M. Prasma \cite{Yonatan}] Let $\textbf{M}$ be a differentiable, left proper and combinatorial symmetric monoidal model category and let $\mathcal{O}$ be a $\Sigma$-cofibrant stably admissible operad in $\textbf{M}$. Assume either $\textbf{M}$ is right proper or $\mathcal{O}_0$ is fibrant. Then the induced Quillen adjunction between stabilizations
	$$ \adjunction{\varphi_!^{\Sp}}{\Sp(\Alg^{\aug}_{\,\mathcal{O}_{\leqslant1}}(\textbf{M}))}{\Sp(\Alg^{\aug}_\mathcal{O}(\textbf{M}))}{\varphi^{*}_{\Sp}} $$
	is a Quillen equivalence.
\end{thm}

\begin{rem}\label{remarkoncomparison}
	In fact, many model categories of interest are not left proper (where $\Op_C(\mathcal{S})$ and $\Op(\mathcal{S})$ are typical examples) and as a sequel, their stabilizations {may} not exist as (full) model categories (cf. Remark \ref{r:semitangent}). In {loc. cit.}, the authors were aware of this fact, and made sure to include Corollary 4.1.4 saying that the restriction functor $\varphi^{*}_{\aug} : \Alg^{\aug}_\mathcal{O}(\textbf{M}) \lrar \Alg^{\aug}_{\,\mathcal{O}_{\leqslant1}}(\textbf{M})$, under the same assumptions as in {Theorem}~\ref{comparison} except the left properness of $\textbf{M}$, induces an equivalence of \textbf{relative categories} after taking stabilizations  $$ \varphi^{*}_{\Sp'} : \Sp'(\Alg^{\aug}_\mathcal{O}(\textbf{M})) \overset{\simeq}{\longrightarrow} \Sp'(\Alg^{\aug}_{\mathcal{O}_{\leqslant1}}(\textbf{M})) .$$
	In particular, when the stabilizations exist as semi model categories then $\varphi_!^{\Sp}\dashv\varphi^{*}_{\Sp}$ is indeed a Quillen equivalence. So keep in mind that the statement of the Comparison theorem remains valid when $\mathcal{P}$ is just stably semi admissible.
\end{rem}

\subsection{The first Quillen equivalence}\label{s:firsteq}

The Quillen adjunction $\adjunction{\mathcal{L}_\P}{\Op_C(\mathcal{S})_{\mathcal{P}/}}{\Op(\mathcal{S})_{\mathcal{P}/}}{\mathcal{R}_\P}$ \eqref{eq:L_PandR_P} lifts to a Quillen adjunction between the associated tangent categories
$$ \adjunction{\mathcal{L}^{\Sp}_\P}{\T_\P\Op_C(\cS)}{\T_\P\Op(\cS)}{\mathcal{R}^{\Sp}_\P} .$$
Our goal in this subsection is to prove that $\mathcal{L}^{\Sp}_\P\dashv \mathcal{R}^{\Sp}_\P$  is a Quillen equivalence. {But first,} let us start with the following simple observations.

\begin{obss}\label{ob:simpleobs} We write ${\L} : \Op_C(\cS) \lrar \Op(\cS)$ to denote the obvious embedding functor. 

		(i) The functor ${\L}$ preserves and detects weak equivalences, cofibrations and trivial fibrations.

		(ii) The functor ${\L}$ preserves and detects cofibrant and fibrant objects.

		(iii) The functor ${\L}$ preserves and detects cofibrant resolutions.

\begin{proof} Let $f : \P \rar \Q$ be a map between $C$-colored operads.

	(i)	 If $f$ is a weak equivalence in $\Op(\mathcal{S})$ then it is in particular a levelwise weak equivalence, i.e., a weak equivalence in $\Op_C(\mathcal{S})$. Conversely, suppose that $f$ is a weak equivalence in $\Op_C(\mathcal{S})$. Since $f$ is the identity on colors, the induced map $\Ho(f)$ is automatically essentially surjective (cf. Definition \ref{d:DKop}). Thus by definition $f$ is a Dwyer-Kan equivalence. 
		
		The claim about trivial fibrations immediately follows by definition. 
		
		Assume that $f$ is a cofibration in $\Op(\mathcal{S})$. We prove that $f$ is one in $\Op_C(\mathcal{S})$. For any given trivial fibration $f' : \P' \rar \Q'$ in $\Op_C(\mathcal{S})$, we have to show that $f$ has the left lifting property against $f'$. Since $f'$ is also a trivial fibration in $\Op(\mathcal{S})$, the lifting problem has a solution in $\Op(\mathcal{S})$. But such a lift must be the identity on colors, so it is also a lift in $\Op_C(\mathcal{S})$. We just showed that $f$ is a cofibration in $\Op_C(\mathcal{S})$. Conversely, if $f$ is a cofibration in $\Op_C(\mathcal{S})$ then it is one in $\Op(\mathcal{S})$, just by the fact that the embedding $\mathcal{L}_\P : \Op_C(\mathcal{S})_{\mathcal{P}/} \lrar \Op(\mathcal{S})_{\mathcal{P}/}$ is a left Quillen functor.
		
	(ii)	 The claim about the fibrancy immediately follows by definition. We now prove the claim about cofibrancy. It can be shown that ${\L}$ detects cofibrant objects, similar to how we verify it on cofibrations. For the converse direction, since ${\L}$ already preserves cofibrations, we just need to show that the initial $C$-colored operad $\I_C$ is also cofibrant as an object of $\Op(\cS)$. Notice that a map in $\Op(\cS)$ from $\I_C$ to an operad $\O$ is fully characterized by a map from $C$ to the set of colors of $\O$. The claim hence follows by the fact that any trivial fibration in $\Op(\cS)$ has a surjective underlying map between colors.
		
	(iii)	 This follows by the two above. 
\end{proof}
\end{obss}

To prove the adjunction $\mathcal{L}^{\Sp}_\P \dashv \mathcal{R}^{\Sp}_\P$ is a Quillen equivalence, we will make use of [\cite{YonatanBundle}, Corollary 2.4.9]. To be able to use this tool, we must show that  the Quillen adjunction   $$ \adjunction{\mathcal{L}^{\aug}_\P}{\Op_C(\mathcal{S})_{\mathcal{P}//\mathcal{P}}}{\Op(\mathcal{S})_{\mathcal{P}//\mathcal{P}}}{\mathcal{R}^{\aug}_\P} $$  is differentiable (cf. Definition~\ref{dfndifferentiable}). 

\begin{rem}\label{r:SCatSdiff} By convention, the base category $\cS$ is differentiable, with weak equivalences closed under sequential colimits. It implies that the (underived) functor $\colim: \cS^{\mathbb{N}} \longrightarrow \cS$ already preserves homotopy Cartesian squares and homotopy terminal objects. An analogue holds for the functor $\colim: \Cat(\cS)^{\mathbb{N}} \longrightarrow \Cat(\cS)$, due to Remark \ref{categoriesisdifferentiavle} and [\cite{YonatanCotangent}, Lemma 3.1.10] (saying that weak equivalences in $\Cat(\cS)$ are closed under sequential colimits). 
\end{rem}

\begin{lem}\label{LRisdifferentible} The Quillen adjunction $\adjunction{\mathcal{L}_\P}{\Op_C(\mathcal{S})_{\mathcal{P}/}}{\Op(\mathcal{S})_{\mathcal{P}/}}{\mathcal{R}_\P}$ is differentiable. Consequently, the induced Quillen adjunction $$\mathcal{L}^{\aug}_\P : \adjunction*{}{\Op_C(\mathcal{S})_{\mathcal{P}//\P}}{\Op(\mathcal{S})_{\mathcal{P}//\P}}{} : \mathcal{R}^{\aug}_\P$$ is differentiable as well.  
\end{lem}
\begin{proof} Note first that sequential colimits of $\cS$-enriched categories are created in $\cS$ on each level (cf. [\cite{YonatanCotangent}, Lemma 3.1.10]). We claim that an analogue holds for enriched operads. Consider a sequence of objects in $\Op(\mathcal{S})$
	$$ \mathcal{P}^{(0)}\longrightarrow \mathcal{P}^{(1)} \longrightarrow \mathcal{P}^{(2)}\longrightarrow \cdots .$$
	We construct an operad $\P$ as follows. Let us take $\sCol(\mathcal{P}) := \colim_n\, \sCol(\mathcal{P}^{(n)})$ where $\sCol(-)$ refers to set of colors. For each $\ovl{c}:=(c_1,\cdots,c_{{m}};c) \in \Seq(\sCol(\mathcal{P}))$, we pick $n_0$ large enough such that $\ovl{c}\in \Seq(\sCol(\mathcal{P}^{(n_0)}))$. Then we take $\mathcal{P}(\ovl{c}) := \colim_{n\geqslant n_0}\,\mathcal{P}^{(n)}(\ovl{c})$. The operad structures on the terms $\mathcal{P}^{(n)}$'s together induce an operad structure on $\P$. It can then be verified that $\P \cong \colim_n \P^{(n)}$. In particular, by construction the underlying category $\P_1$ is isomorphic to $\colim_n \P^{(n)}_1$.

	We now claim that weak equivalences in $\Op(\mathcal{S})$ are closed under sequential colimits. Indeed, recall that a map $f$ in $\Op(\mathcal{S})$ is a weak equivalence if and only if it is a levelwise weak equivalence and such that the map $f_1$ between underlying categories is a weak equivalence in $\Cat(\mathcal{S})$ (cf. Remark \ref{r:canonicalfib}). Hence, the claim follows by the above paragraph and by the fact that weak equivalences in $\cS$ and $\Cat(\mathcal{S})$ are closed under sequential colimits.

Next we claim that a square in $\Op(\mathcal{S})$ is homotopy Cartesian if and only if the following two conditions hold:
	\begin{enumerate}
		\item the induced squares of spaces of operations are homotopy Cartesian in $\mathcal{S}$, and
		\item the induced square of underlying categories is homotopy Cartesian in $\Cat(\cS)$.
	\end{enumerate}
	
	\noindent Indeed, {observe} first that the statement is already correct when we forget the word $``$homotopy''. Combined with Remark \ref{r:canonicalfib}, {this observation verifies} the claim. Besides that, it is not hard to show that an object of $\Op(\mathcal{S})$ is homotopy terminal if and only if all its spaces of operations are homotopy terminal in $\mathcal{S}$.

	We now show that $\Op(\mathcal{S})$ is differentiable. By the second paragraph, it suffices to verify that the functor $\colim: \Op(\mathcal{S})^{\mathbb{N}} \longrightarrow \Op(\mathcal{S})$ preserves homotopy Cartesian squares and homotopy terminal objects. This follows by combining the first and third paragraphs, along with Remark \ref{r:SCatSdiff}. In the same manner, one can show that the category $\Op_C(\mathcal{S})$ is also differentiable.

	The differentiability of $\Op(\mathcal{S})_{\mathcal{P}/}$ (resp. $\Op_C(\mathcal{S})_{\mathcal{P}/}$) follows from that of $\Op(\mathcal{S})$ (resp. $\Op_C(\mathcal{S})$). Moreover, it is clear that the functor $\mathcal{R}_\P$ preserves sequential homotopy colimits. Thus the adjunction $\mathcal{L}_\P \dashv \mathcal{R}_\P$ is differentiable, and hence so is the adjunction $\mathcal{L}^{\aug}_\P \dashv \mathcal{R}^{\aug}_\P$. 
\end{proof}

We are now in position to prove the main result of this subsection. 

\begin{prop}\label{Quilleneq1} The adjunction $\adjunction{\mathcal{L}^{\Sp}_\P}{\T_\P \Op_C(\mathcal{S})}{\T_\P \Op(\mathcal{S})}{\mathcal{R}^{\Sp}_\P}$ is a Quillen equivalence.
\end{prop}
\begin{proof} Let $\mathcal{Q}\in\Op(\mathcal{S})_{\mathcal{P}// \mathcal{P}}$ be a fibrant object, exhibited by a diagram $\P\rar\Q\rar\P$ in $\Op(\mathcal{S})$ such that the second map is a fibration. The same argument as in the proof of [\cite{YonatanCotangent}, Lemma 3.1.13] shows that the map between the homotopy pullbacks     $$\mathcal{P}\times^{\h}_{\L_\P\R_\P(\mathcal{Q})}\mathcal{P}\lrar \mathcal{P}\times^{\h}_{\mathcal{Q}}\mathcal{P}$$
	is a weak equivalence in $\Op(\mathcal{S})$. {(We note here that, since $\Op(\mathcal{S})$ is right proper, the requirement for fibrancy on $\P$ can be omitted)}. In particular, the induced map $\Omega\,\mathcal{L}^{\aug}_\P\mathcal{R}^{\aug}_\P(\mathcal{Q}) \lrar \Omega\,\mathcal{Q}$ is a weak equivalence in $\Op(\mathcal{S})_{\mathcal{P}//\mathcal{P}}$. This fact, together with Lemma \ref{LRisdifferentible}, allows us to apply [\cite{YonatanBundle}, Corollary 2.4.9] to  deduce that the derived counit of the Quillen adjunction $\mathcal{L}^{\Sp}_\P\dashv \mathcal{R}^{\Sp}_\P$ is a stable equivalence for every fibrant {object}.

	It remains to show that the derived unit of $\mathcal{L}^{\Sp}_\P\dashv \mathcal{R}^{\Sp}_\P$ is a stable equivalence for any cofibrant object. In fact, we will show that this holds for the larger class of levelwise cofibrant objects. Since $\mathcal{R}_\P^{\aug}\circ\mathcal{L}_\P^{\aug}$ is isomorphic to the identity functor and since $\mathcal{R}_\P^{\aug}$ preserves weak equivalences, the derived unit of $\mathcal{L}_\P^{\aug}\dashv \mathcal{R}_\P^{\aug}$ is a weak equivalence. By the first part of [\cite{YonatanBundle}, Corollary 2.4.9], the derived unit of $\mathcal{L}^{\Sp}_\P\dashv \mathcal{R}^{\Sp}_\P$ is a stable equivalence for any levelwise cofibrant prespectrum. But every levelwise cofibrant object in $\T_\P \Op_C(\mathcal{S})$  is stably equivalent to a levelwise cofibrant prespectrum (see \cite{YonatanBundle}, Remark 2.3.6), it therefore suffices to {argue} that $\mathcal{L}^{\Sp}_\P$ preserves stable equivalences. But this immediately follows from Observation \ref{ob:stablecohomotopy}.   
\end{proof} 

\smallskip

\subsection{The second Quillen equivalence}\label{s:infbimod}

\smallskip

{The next step is to prove the following.}

\begin{prop}\label{operad&infinitesimalbimod} Suppose that $\P$ is {$\Sigma$-cofibrant}. Then the adjunction    
	$$ \adjunction*{}{\T_\P \IbMod(\mathcal{P})}{\T_\P \Op_C(\mathcal{S})}{} $$
	is a Quillen equivalence. 
\end{prop}

{The proof will require a technical lemma. Let us take a cofibrant resolution $f:\Q\lrarsimeq\P$ for $\P\in \Op_C(\mathcal{S})$ such that $f$ is a trivial fibration.}
	
	\begin{lem} {Suppose that $\P$ is $\Sigma$-cofibrant. Then the induced adjunctions 
		\begin{gather*} \widetilde{f_!} : \adjunction*{}{\IbMod(\mathcal{Q})_{\mathcal{Q}//\mathcal{Q}}}{\IbMod(\mathcal{P})_{\mathcal{P}//\mathcal{P}}}{} : \widetilde{f_*}, \; \; \text{and} \\ \hat{f_!} : \adjunction*{}{\Op_C(\mathcal{S})_{\mathcal{Q}//\mathcal{Q}}}{\Op_C(\mathcal{S})_{\mathcal{P}//\mathcal{P}}}{} : \hat{f_*}	
		\end{gather*}	
			 are a Quillen equivalence.}
		\begin{proof}  {Let us start with the second one. For an object $\O\in\Op_C(\mathcal{S})_{\mathcal{P}//\mathcal{P}}$, recall by definition that $\hat{f_*}(\O)$ is given by the pullback $\Q\times_\P\O \in \Op_C(\mathcal{S})_{\mathcal{Q}//\mathcal{Q}}$. Note that the map $\hat{f_*}(\O) \lrar \O$ is a trivial fibration, since it is a base change of $f$. In particular, $\hat{f_*}$ creates weak equivalences. Thus, we just need to show that for any cofibrant object $\R\in\Op_C(\mathcal{S})_{\mathcal{Q}//\mathcal{Q}}$ exhibited by a sequence $\Q\lrar\R\lrar\Q$ such that the first map is a cofibration, the unit map $\R \lrar \hat{f_*}\hat{f_!}(\R)$ is a weak equivalence. This is equivalent to saying that the composed map $$ \R \lrar \hat{f_*}\hat{f_!}(\R) \lrar \hat{f_!}(\R) $$ is a weak equivalence. By construction, the latter is given by the induced map $\R\longrightarrow \R\,\underset{\Q}{\bigsqcup}\,\P$, which is indeed a weak equivalence due to the \textbf{relative left properness} (cf., e.g. \cite{Caviglia, Spitzweck, Fresse1}), after having assumed $\P$ to be $\Sigma$-cofibrant.}
			
			We now show that $\widetilde{f_!} \dashv \widetilde{f_*}$ is a Quillen equivalence. First, by the fact that the map $f:\Q\lrarsimeq\P$ is a weak equivalence between levelwise cofibrant operads, we obtain that the induced adjunction
			\begin{equation}\label{eq:secqui}
				\adjunction*{}{ \IbMod(\mathcal{Q})_{\Q/}\cong\Alg_{\, \textbf{B}^{\mathcal{Q}/} _{\leqslant1}}(\cS)}{\Alg_{\, \textbf{B}^{\mathcal{P}/} _{\leqslant1}}(\cS) \cong \IbMod(\mathcal{P})_{\P/}}{}
			\end{equation}	
			is a Quillen equivalence (see Construction \ref{operadencodebimodulesunderP}). {Now, 
			for an object $M\in \IbMod(\mathcal{P})_{\mathcal{P}//\mathcal{P}}$ we have $\widetilde{f_*}(M) = \Q\times_\P M \in \IbMod(\mathcal{Q})_{\mathcal{Q}//\mathcal{Q}}$ where $\P$ and $M$ are considered as infinitesimal $\Q$-bimodules, with the structure induced by $f$. Therefore, similarly to the case above, the map $\widetilde{f_*}(M) \lrar M$ is a trivial fibration in $\IbMod(\mathcal{Q})$, and hence, $\widetilde{f_*}$ also creates weak equivalences.} Thus, along with the previous observation, the same argument as in the case above can be applied to complete the proof.
		\end{proof}
	\end{lem}

	\begin{proof}[\underline{Proof of Proposition \ref{operad&infinitesimalbimod}}] {In the first step, let us assume further that $\mathcal{P}$ is cofibrant.} We regard $\P$ as an algebra over $\textbf{O}_{C}$ the operad of $C$-colored operads. Then we get a canonical isomorphism $ \Alg_{\Env(\textbf{O}_{C},\P)}(\cS) \cong \Op_C(\mathcal{S})_{\mathcal{P}/}$ between the category of algebras over the enveloping operad $\Env(\textbf{O}_{C},\P)$ and the category of $C$-colored operads under $\P$ (cf. \S \ref{s:operadicmodules}). On other hand, similarly as in the proof of [\cite{Guti}, Proposition 3.5], it can be shown that the structure of an infinitesimal $\P$-bimodule is equivalent to that of a $\P$-module over $\textbf{O}_{C}$. So we get an isomorphism of categories $\Alg_{\Env(\textbf{O}_{C},\P)_1}(\cS) \cong \IbMod(\mathcal{P})$ (cf. Remark \ref{r:AmodPasfunctor}). 
		
The symmetric groups act freely on $\textbf{O}_{C}$. In particular,  $\textbf{O}_{C}$ is $\Sigma$-cofibrant. Moreover, since $\mathcal{P}$ is cofibrant, it implies that $\Env(\textbf{O}_{C},\P)$ is $\Sigma$-cofibrant as well (cf. [\cite{Pavlov}, Lemma 6.1]).  This allows us to apply the Comparison theorem \ref{comparison} for $\Env(\textbf{O}_{C},\P)$. The first paragraph shows that the functor
		$$\Alg^{\aug}_{\Env(\textbf{O}_{C},\P)_{\leqslant1}}(\cS)\lrar \Alg_{\Env(\textbf{O}_{C},\P)}^{\aug}(\cS)$$
		coincides with the functor $\IbMod(\mathcal{P})_{\mathcal{P}//\mathcal{P}} \lrar \Op_C(\mathcal{S})_{\mathcal{P}//\mathcal{P}}$. From this, we obtain that the adjunction $\adjunction*{}{\T_\P \IbMod(\mathcal{P})}{\T_\P \Op_C(\mathcal{S})}{}$ is a Quillen equivalence. 
		
		Finally, in order to show that the statement already holds true when $\mathcal{P}$ is $\Sigma$-cofibrant, we just need to make use of the above lemma.
\end{proof}

\smallskip

\subsection{The third Quillen equivalence}\label{s:bimod}

\smallskip

Recall from \S \ref{s:operadicmodules} that the category $\BMod(\mathcal{P})_{\mathcal{P}/}$ can be represented as the category of algebras over the operad $\textbf{B}^{\mathcal{P}/}$, whose underlying category $\textbf{B}^{\mathcal{P}/}_1$ agrees with $\textbf{Ib}^{\P}$ (cf. Constructions \ref{catencodinginfbimod} and \ref{operadencodebimodulesunderP}). The following is the last piece for the proof of Theorem \ref{t:mainoptangent}.

\begin{prop}\label{bimod&infinitesimalbimod} The adjunction $\adjunction*{}{\IbMod(\mathcal{P})_{\mathcal{P}/}}{\BMod(\mathcal{P})_{\mathcal{P}/}}{}$ induces a Quillen equivalence of the associated tangent categories $$\overset{\simeq}{   \adjunction*{}{\T_\P \IbMod(\mathcal{P})}{\T_\P \BMod(\mathcal{P})}{}}$$ whenever $\P$ is $\Sigma$-cofibrant.
\end{prop}
\begin{proof} When $\P$ is $\Sigma$-cofibrant, by construction $\textbf{B}^{\mathcal{P}/}$ is also $\Sigma$-cofibrant. We are now applying the Comparison theorem \ref{comparison}, along with noting Remark \ref{remarkoncomparison}, to the operad $\textbf{B}^{\mathcal{P}/}$. The key point is that the adjunction $ \adjunction*{} {\Alg^{\aug}_{\, \textbf{B}^{\mathcal{P}/} _{\leqslant1}}(\cS)} {\Alg^{\aug}_{\, \textbf{B}^{\mathcal{P}/} }(\cS)} {} $ which arises from the inclusion $(\textbf{B}^{\mathcal{P}/})_{\leqslant1}\lrar \textbf{B}^{\mathcal{P}/}$ is the same as the adjunction of induction-restriction functors
	$$  \adjunction*{} {\IbMod(\mathcal{P})_{\mathcal{P}//\mathcal{P}}} {\BMod(\mathcal{P})_{\mathcal{P}//\mathcal{P}}} {} .$$
	Thus the adjunction $\adjunction*{}{\T_\P \IbMod(\mathcal{P})}{\T_\P \BMod(\mathcal{P})}{}$ is indeed a Quillen equivalence. 
\end{proof}

\smallskip

\subsection{Main statements}

\smallskip

Combining the propositions \ref{Quilleneq1}, \ref{operad&infinitesimalbimod} and \ref{bimod&infinitesimalbimod}, we obtain the main result of this section stated as follows.

\begin{thm}\label{t:mainoptangent} The adjunctions in the sequence 
	\begin{equation}\label{eq:sequenceQuillen1}
		\adjunction*{}{\T_\P\IbMod(\P)}{\T_\P\BMod(\P)}{} \adjunction*{}{}{\T_\P\Op_C(\mathcal{S})}{} \adjunction*{}{}{\T_\P\Op(\mathcal{S})}{}
	\end{equation}
	are all Quillen equivalences provided that $\P$ is $\Sigma$-cofibrant.
\end{thm}

{We are also interested in the case of a stable base category.} The same argument as in the proof of [\cite{Yonatan}, Lemma 2.2.3] proves the following.

\begin{lem}\label{l:kerstable} {Suppose further that $\cS$ is stably sufficient.} Let $M\in \IbMod(\P) $ be a levelwise cofibrant infinitesimal $\P$-bimodule. Then the adjunction
	$$ (-) \sqcup M : \adjunction*{}{\IbMod(\P)}{\IbMod(\P)_{M//M}}{} : \ker $$
	is a Quillen equivalence, where the functor $\ker$ is defined by sending $[M \rar P \rar M]$ to $P\times_M 0$, while {the} left adjoint takes $N\in\IbMod(\P)$ to $[M \x{i_0}{\lrar} M\sqcup N \x{\Id_M + 0}{\xrightarrow{\hspace*{1cm}}} M]$.
\end{lem}

\begin{thm}\label{t:mainoptangent'} {Suppose that $\cS$ is stably sufficient and $\P$ is $\Sigma$-cofibrant}. The sequence \eqref{eq:sequenceQuillen1} is then prolonged to a sequence of Quillen equivalences of the form 
	\begin{equation}\label{eq:sequenceQuillen'}
		\adjunction*{(-) \sqcup \P}{\IbMod(\P)}{\IbMod(\P)_{\P//\P}}{\ker} \adjunction*{\Sigma^{\infty}}{}{\T_\P\IbMod(\P)}{\Omega^{\infty}} \adjunction*{\simeq}{}{\T_\P\BMod(\P)}{} \adjunction*{\simeq}{}{\T_\P\Op_C(\mathcal{S})}{} \adjunction*{\simeq}{}{\T_\P\Op(\mathcal{S})}{}
	\end{equation}
\end{thm}
\begin{proof} The category $\IbMod(\P)$ is stable since $\cS$ is stable (cf. [\cite{YonatanCotangent}, Remark 2.2.2] and Proposition \ref{rewriteinfinitesimal1}) and thus, the category $\IbMod(\P)_{\P//\P}$ is stable as well. So by Theorem \ref{t:stabilization} the adjunction $$\Sigma^{\infty} : \adjunction*{}{\IbMod(\P)_{\P//\P}}{\T_\P\IbMod(\P)}{} : \Omega^{\infty}$$ is a Quillen equivalence. The statement hence follows by Lemma \ref{l:kerstable} and Theorem \ref{t:mainoptangent}.
\end{proof}

\section{Quillen cohomology of enriched operads}\label{chap:Qcohomenrichedop}

This section contains the central results of the paper. As the main goal, we give an explicit description of the cotangent complex of enriched operads. Moreover, we prove the existence of a long exact sequence relating Quillen cohomology and reduced Quillen cohomology of a given operad.

\smallskip

\subsection{An extra condition}\label{sub:extracondition}

\smallskip

{Let $\cS$ be a symmetric monoidal model category such that the transferred model structure on operadic algebras is available and let $\P$ be a $C$-colored operad in $\cS$.}

\begin{notn}\label{not:pointedbimod} We denote by $\BMod(\mathcal{P})^{*}:=\BMod(\mathcal{P})_{\mathcal{P}\circ \mathcal{P}/}$ the category of  $\P$-bimodules under $\mathcal{P}\circ \mathcal{P}$, which {represents} the free $\P$-bimodule generated by $\mathcal{I}_C$, and refer to $\BMod(\mathcal{P})^{*}$ as the category of \textbf{pointed $\mathcal{P}$-bimodules}. 
\end{notn}

Observe that the composition $\mu:\mathcal{P}\circ \mathcal{P} \longrightarrow \mathcal{P}$ exhibits $\mathcal{P}$ itself as a pointed $\mathcal{P}$-bimodule. Let $\mathcal{P}\sqcup\mathcal{P} \x{f + g}{\lrar} \Q$ be a map in $\Op_C(\mathcal{S})$. Then $\Q$ inherits a $\P$-bimodule structure with the left (resp. right) $\P$-action induced by $f$ (resp. $g$).  This determines a restriction functor $$\Op_C(\mathcal{S})_{\mathcal{P}\sqcup\mathcal{P}/} \lrar \BMod(\mathcal{P})^{*} ,$$ which admits a left adjoint denoted by $\E : \BMod(\mathcal{P})^{*}\lrar \Op_C(\mathcal{S})_{\mathcal{P}\sqcup\mathcal{P}/}$. Observe that $\E$ sends $\P\in \BMod(\mathcal{P})^{*}$ to itself $\P\in \Op_C(\mathcal{S})_{\mathcal{P}\sqcup\mathcal{P}/}$ equipped with the fold map $\Id_\P + \Id_\P : \mathcal{P}\sqcup\mathcal{P} \lrar \P$. 

{We recall a hypothesis introduced by Dwyer-Hess in \cite{Hess}.}
\begin{dfn}\label{S9} {(\textbf{Goodness hypothesis})} We will say that the operad $\P\in \Op_C(\mathcal{S})$ is \textbf{good} if the left derived  functor  $$\LL\E:\BMod(\mathcal{P})^{*}\lrar \Op_C(\mathcal{S})_{\mathcal{P}\sqcup\mathcal{P}/}$$ sends $\P$ to itself $\mathcal{P}$. Moreover, the base category $\cS$ is said to be \textbf{good} if every cofibrant operad in $\cS$ is good.
\end{dfn}

{For example, in {loc. cit.}, the authors verified that every \textbf{nonsymmetric simplicial operad} is good. Expanding upon their result, we shall now prove that the category $\Set_\Delta$ is good (when working in the context of symmetric operads). Moreover, we show that the category $\sMod_R$ with $R$ being any commutative ring is also good.}

\begin{prop}\label{p:splcsetS9} The category $\Set_\Delta$ of simplicial sets is good. Namely, for every set of colors $C$ and for every cofibrant simplicial $C$-colored operad $\P$, the  left derived   functor $$\LL\E:\BMod(\mathcal{P})^{*}\lrar \Op_C(\Set_\Delta)_{\mathcal{P}\sqcup\mathcal{P}/}$$ of $\E$ sends $\P$ to itself $\mathcal{P}$.
\end{prop}

The proof first requires constructing a nice cofibrant resolution for $\P$ as an object of $\BMod(\mathcal{P})^{*}$. For this, we will follow Rezk's [\cite{Rezk}, $\S$ 3.7.2]. ({Nonetheless}, note that the operadic model structures considered in {loc. cit.} are different from ours, so it should be used carefully). Let $\M$ be a simplicial model category. The \textbf{diagonal} (or \textbf{realization}) functor $\diag : \M^{\Delta^{\op}} \lrar \M $ is by definition the left adjoint to the functor $\M \lrar \M^{\Delta^{\op}}$ taking each $X\in\M$ to the simplicial object $[n] \mapsto X^{\Delta^{n}}$. For each $Y_\bullet\in\M^{\Delta^{\op}}$, one defines the \textbf{latching object} $\rL_nY_\bullet$ as the coequalizer in $\M$ of the form
$$ \underset{0\leqslant i < j \leqslant n}{\bigsqcup} Y_{n-1}  \; \doublerightarrow{}{} \;  \underset{0\leqslant k \leqslant n}{\bigsqcup} Y_{n} \lrar \rL_nY_\bullet $$
in which one of the two maps sends the $(i,j)$ summand to the $j$'th summand by $s_i$ while the other sends the $(i,j)$ summand to the $i$'th summand by $s_{j-1}$. By convention, one takes  $\rL_{-1}Y_\bullet := \emptyset$.  By construction, there is a unique map $\rL_nY_\bullet \lrar Y_{n+1}$ factoring the map $ s_k : Y_n \lrar Y_{n+1}$ for every $k=0,\cdots,n$. One then  constructs a filtration $\diag(Y_\bullet) = \colim_n \diag_nY_\bullet$ built up inductively by taking $\diag_0Y_\bullet := Y_0$ and then, for each $n\geqslant1$, taking the pushout:
\begin{equation}\label{eq:diag}
	\xymatrix{
		d_nY_\bullet \ar[r]\ar[d] & \Delta^{n}\otimes Y_n \ar[d] \\
		\diag_{n-1}Y_\bullet \ar[r] & \diag_nY_\bullet \\
	}
\end{equation}
in which $$d_nY_\bullet := \Delta^{n}\otimes\rL_{n-1}Y_\bullet \bigsqcup_{\partial\Delta^{n}\otimes\rL_{n-1}Y_\bullet}\partial\Delta^{n}\otimes Y_n.$$ As a consequence, if for every $n\geqslant0$ the \textbf{latching map} $\rL_{n-1}Y_\bullet \lrar Y_n$ is a cofibration then $\diag(Y_\bullet)$ is cofibrant. More generally, we have the following lemma.

\begin{lem}\label{l:diagcofib} Let $X_\bullet \lrar Y_\bullet$ be a map of simplicial objects in $\M$. Suppose that for every $n\geqslant 0$ the (relative) latching map 
	\begin{equation}\label{eq:latchingmap}
		X_{n}\bigsqcup_{\rL_{n-1}X_\bullet}\rL_{n-1}Y_\bullet \lrar Y_{n}
	\end{equation}
	is a cofibration. Then the induced map $\diag(X_\bullet) \lrar \diag(Y_\bullet)$ is a cofibration as well.  
\end{lem}
\begin{proof} By the filtrations of $\diag(X_\bullet)$ and $\diag(Y_\bullet)$ mentioned above, the map $\diag(X_\bullet) \lrar \diag(Y_\bullet)$ is a cofibration as soon as  the map $\diag_nX_\bullet \lrar \diag_nY_\bullet$ is one for every $n\geqslant0$. Note first that when $n=0$ the  map \eqref{eq:latchingmap} coincides with the map $\diag_0X_\bullet \lrar \diag_0Y_\bullet$. Let us assume by induction that the map $\diag_{n-1}X_\bullet \lrar \diag_{n-1}Y_\bullet$ is a cofibration. Then, factor $\diag_nX_\bullet \lrar \diag_nY_\bullet$ as 
	$$ \diag_nX_\bullet \lrar \diag_nX_\bullet\bigsqcup_{\diag_{n-1}X_\bullet}\diag_{n-1}Y_\bullet \x{\varphi}{\lrar} \diag_nY_\bullet .$$
	By the inductive assumption, the first map in this composition is a cofibration. Hence, it remains to show that $\varphi$ is a cofibration. Let us denote by $L_{n-1}(X_\bullet,Y_\bullet) := X_{n}\bigsqcup_{\rL_{n-1}X_\bullet}\rL_{n-1}Y_\bullet$. We can then form a canonical map $$ \Delta^{n}\otimes L_{n-1}(X_\bullet,Y_\bullet) \bigsqcup_{\partial\Delta^{n}\otimes L_{n-1}(X_\bullet,Y_\bullet)} \partial\Delta^{n}\otimes Y_n \lrar \Delta^{n}\otimes Y_n, $$
	which is a cofibration by the pushout-product axiom. {Moreover, it can be shown that this map is isomorphic to the canonical map}
	\begin{equation}\label{eq:latchingmap1}
		d_nY_\bullet \bigsqcup_{d_nX_\bullet} \Delta^{n}\otimes X_n \lrar \Delta^{n}\otimes Y_n .
	\end{equation}
	Now, consider the following commutative cube 
	\begin{center}
		\begin{tikzcd}[row sep=0.5em, column sep = 0.5em]
			d_nX_\bullet \arrow[rr] \arrow[dr,dashed, swap] \arrow[dd,swap] &&
			\Delta^{n}\otimes X_n \arrow[dd] \arrow[dr] \\
			& d_nY_\bullet \arrow[rr] \arrow[dd] &&
			\Delta^{n}\otimes Y_n \arrow[dd] \\
			\diag_{n-1}X_\bullet \arrow[rr,] \arrow[dr] && \diag_nX_\bullet \arrow[dr] \\
			& \diag_{n-1}Y_\bullet \arrow[rr] && \diag_{n}Y_\bullet
		\end{tikzcd}
	\end{center}
	whose the front and back squares are coCartesian. Applying the pasting law of pushouts iteratively, we find that $\varphi$ is a cobase change of the map \eqref{eq:latchingmap1}, which is a cofibration. We therefore get the conclusion.
\end{proof}

{For any simplicial operad $\P$, the category $\BMod(\P)$ admits a simplicial model structure transferred from that on the category of symmetric sequences} (see [\cite{Rezk}, Propositions 3.1.5, 3.2.8]). One constructs the Hochschild resolution of $\P$ as follows.

\begin{cons}  Let $\sH_\bullet\P : \Delta^{\op} \lrar \BMod(\mathcal{P})$ be the simplicial object of $\P$-bimodules with $\sH_n\P := \P^{\circ(n+2)}$, the face map $d_i : \sH_n\P\lrar \sH_{n-1}\P$ given by using the composition $\mu : \P\circ\P \lrar \P$ to combine the factors $i+1$ and $i+2$ in $\sH_n\P$ and with the degeneracy map $s_i$ given by inserting the unit operations of $\P$ between the factors $i+1$ and $i+2$. The realization $\diag(\sH_\bullet\P) \in \BMod(\mathcal{P})$ has $n$-simplices being those of  $\sH_n\P$. The  map $\mu$ induces a canonical map of simplicial objects $\sH_\bullet\P \lrar \P$. According to [\cite{Rezk}, Corollary 3.7.6], the induced map $\psi : \diag(\sH_\bullet\P) \lrar \diag(\P)=\P$ is a weak equivalence of $\P$-bimodules, {and exhibits $\diag(\sH_\bullet\P)$ as the \textbf{Hochschild resolution of} $\P\in \BMod(\P)$}.

	 On other hand, since $\P\circ\P = \sH_0\P$, there is a unique map of simplicial objects $\P\circ\P \lrar \sH_\bullet\P$, which is the identity on degree 0. Now, the diagonal functor determines a map $ \rho : \P\circ\P \lrar \diag(\sH_\bullet\P)$ of $\P$-bimodules. Moreover, the composition $ \P\circ\P \x{\rho}{\lrar} \diag(\sH_\bullet\P) \x{\psi}{\lrar} \P$ agrees with $\mu : \P\circ\P \lrar \P$.
\end{cons}

\begin{lem}\label{c:acofibration} Suppose that $\P$ is a $\Sigma$-cofibrant simplicial operad. The map $\psi$ exhibits $\diag(\sH_\bullet\P)$ as a cofibrant resolution for $\P\in \BMod(\P)$. Moreover, the map $\rho : \P\circ\P \lrar \diag(\sH_\bullet\P)$ is a cofibration of $\P$-bimodules. In particular, $\diag(\sH_\bullet\P)$ is also a {cofibrant} resolution for $\P$ when regarded as a pointed $\P$-bimodule.
\end{lem}
\begin{proof} For the first statement, it suffices to show that the latching map $\rL_{n-1}\sH_\bullet\P \lrar \sH_n\P$ is a cofibration for every $n\geqslant 0$. To this end, one can follow up the argument as in the proof of [\cite{Rezk}, Corollary 3.7.6]. Besides that, one will need {to make use of} the fact that, for two maps $f$ and $g$ of symmetric sequences in a sufficiently nice symmetric monoidal model category, the pushout-product of $f$ with $g$ is a cofibration as soon as both of them are one and in addition, the domain of $g$ is cofibrant (cf. [\cite{Fresse1},  Lemma 11.5.1]). 
	
To prove that the map $\rho$ is a cofibration, we make use of Lemma \ref{l:diagcofib}. Note that since $\P\circ\P$ is considered as a constant simplicial object, the latching map \eqref{eq:latchingmap} is simply given by the map $\rL_{n-1}\sH_\bullet\P \lrar \sH_n\P$ when $n\geqslant 1$ and the identity map $\Id_{\P\circ\P}$ when $n=0$. 
\end{proof}

\begin{rem}\label{r:hochschildresalg} Let $\P$ be a {$\Sigma$-cofibrant} simplicial operad and let $A$ be any $\P$-algebra. {The} \textbf{Hochschild resolution of} $A$ is {given by} the realization of the simplicial $\P$-algebra $\sH^{\P}_\bullet A$  with $\sH^{\P}_n A = \P^{\circ(n+1)}\circ A$. The augmentation map $\diag(\sH^{\P}_\bullet A) \lrar A$ is a weak equivalence by [\cite{Rezk}, Corollary 3.7.4], and indeed exhibits $\diag(\sH^{\P}_\bullet A)$ as a cofibrant resolution for $A$. To see this, one just needs to show that the latching map $\rL_{n-1}\sH^{\P}_\bullet A \lrar \sH_n^{\P}A$ is a cofibration for every $n\geqslant 0$. This is an analogue of the first statement in the above lemma. 
\end{rem}

\begin{proof}[\underline{Proof of Proposition~\ref{p:splcsetS9}}] The idea will be to first argue that it suffices to treat only the case where $\P$ is the symmetrization of a discrete free nonsymmetric operad.

\noindent $\textbf{Step 1.}$ Let us first follow the arguments as in [\cite{Hess}, $\S$5]. Applying the functor $\E$ to $\sH_\bullet\P$ degreewise, one obtains a simplicial object $\E \sH_\bullet\P$ of operads under $\mathcal{P}\sqcup\mathcal{P}$. The realization $\diag(\E \sH_\bullet\P)$ is then an operad under $\mathcal{P}\sqcup\mathcal{P}$. Since $\E(\P) \cong \P$ in $\Op_C(\Set_\Delta)_{\P\sqcup\P/}$, there is a canonical map $\varphi_\P : \diag(\E \sH_\bullet\P) \lrar \P$ of operads under $\mathcal{P}\sqcup\mathcal{P}$.
	One observes that there is a canonical isomorphism $\diag(\E \sH_\bullet\P)\cong \E(\diag(\sH_\bullet\P))$ of operads under $\mathcal{P}\sqcup\mathcal{P}$ and over $\P$ (cf. [\cite{Hess}, Proposition 5.3]). Since $\E(\diag(\sH_\bullet\P))$ is a model for $\LL \E(\P)$ by Lemma \ref{c:acofibration}, it remains to show that the map $ \varphi_\P : \diag(\E \sH_\bullet\P) \lrar \P$ is a weak equivalence of operads. Now, observe that $\E \sH_n\P \cong \P \sqcup \F_*(\P^{\circ n}) \sqcup \P$ where $$\F_* : \Coll_C(\Set_\Delta)_{\I_C/} \lrar \Op_C(\Set_\Delta)$$ refers to the free-operad functor on \textit{pointed $C$-collections}. Consequently, if $\Q\lrarsimeq\P$ is a weak equivalence between cofibrant operads then the induced map  $\E \sH_n\Q \lrar \E \sH_n\P$ is also a weak equivalence. Therefore, by the {\textit{realization lemma} (see e.g., [\cite{Goerss}, $\S$4.1]),} $\varphi_\Q$ is a weak equivalence if and only if $\varphi_\P$ is one. Applying the {realization lemma} in the other direction, we get that $\varphi_\P$ is a weak equivalence as soon as the map $\varphi_{\P_{n}}$ is one for every $n\geqslant0$ (where $\P_{n}$ denotes the  operad of $n$-simplices of $\P$).

	Now, consider $\P$ as an $\textbf{O}_C$-algebra with $\textbf{O}_C$ being the operad of simplicial $C$-colored operads. Remark \ref{r:hochschildresalg} suggests that we can make use of $\P^{c} := \diag(\sH^{\textbf{O}_C}_\bullet \P)$ as (another) cofibrant model for $\P$. By the above paragraph, it suffices to verify that the map $\varphi_{\P^{c}}$ is a weak equivalence and to this end, we just need to show that the map $\varphi_{\P^{c}_{n}}$ is a weak equivalence for every $n$. By construction, we have  $$\P^{c}_{n} = ((\textbf{O}_C)^{\circ(n+1)} \circ \P)_{n} = (\textbf{O}_C)^{\circ(n+1)} \circ \P_{n} .$$ (The second identification is because of the fact that $\textbf{O}_C$ is a discrete operad). In particular, $\P^{c}_{n}$ is a discrete free $\textbf{O}_C$-algebra. This is in fact equivalent to saying that $\P^{c}_{n}$ is the symmetrization of a discrete free nonsymmetric operad. 
	
	\smallskip

\noindent $\textbf{Step 2.}$ By the first step, we can assume without  loss of generality that $\P$ is the symmetrization of a discrete free nonsymmetric operad. Nevertheless, for the remainder we just need to assume that $\P$ is the symmetrization of a nonsymmetric operad $\Q$, i.e., $\P =\Sym(\Q)$ (see \S \ref{s:operad}). Observe that the symmetrization functor lifts to the functors
\begin{gather*} \Sym : \BMod(\Q)^{*} \lrar \BMod(\Sym(\Q))^{*}, \; \; \text{and} \\
\Sym : \nsOp_C(\Set_\Delta)_{\mathcal{Q}\sqcup\mathcal{Q}/} \lrar \Op_C(\Set_\Delta)_{\Sym(\Q)\sqcup\Sym(\Q)/},
\end{gather*}
which are left adjoints to the associated forgetful functors. Moreover, we have a commutative square of left Quillen functors 
	$$ \xymatrix{
		\BMod(\Q)^{*} \ar^{\E \;\;\;\;\;\;}[r]\ar_{\Sym}[d] & \nsOp_C(\Set_\Delta)_{\mathcal{Q}\sqcup\mathcal{Q}/} \ar^{\Sym}[d] \\
		\BMod(\Sym(\Q))^{*} \ar_{\E \;\;\;\;\;\;\;\;\;}[r] & \Op_C(\Set_\Delta)_{\Sym(\Q)\sqcup\Sym(\Q)/} \\
	}$$
	(for this, it suffices to verify the commutativity of the associated square of right adjoints). According to [\cite{Hess}, Proposition 5.4], we have that $\LL\E\Q \cong \Q$. Thus, by the commutativity of the above square and by the fact that the symmetrization functor preserves weak equivalences, we obtain the expected identification $\LL\E\Sym(\Q) \cong \Sym(\Q)$.
\end{proof}

Let $R$ be any commutative ring. Hochschild resolutions work in the context of simplicial $R$-modules under a slightly different setting. Let $\P\in \Op_C(\sMod_R)$ be given such that the unit map $\I_C \lrar \P$ is a cofibration of symmetric sequences. Then the composition $$\P\circ\P\lrar\diag(\sH_\bullet\P)\lrarsimeq \P$$ exhibits $\diag(\sH_\bullet\P)$ as a cofibrant resolution for $\P\in \BMod(\mathcal{P})^{*}$. {Moreover, for $A$ a levelwise cofibrant $\P$-algebra,} the augmentation map $\diag(\sH^{\P}_\bullet A) \lrarsimeq A$ exhibits $\diag(\sH^{\P}_\bullet A)$ as a cofibrant resolution for $A$.

\begin{prop} \label{p:splcmoduleS9} The category $\sMod_R$ is good. Namely, for every set of colors $C$ and for every cofibrant object $\P \in \Op_C(\sMod_R)$, the left derived functor $$\LL\E:\BMod(\mathcal{P})^{*}\lrar \Op_C(\sMod_R)_{\mathcal{P}\sqcup\mathcal{P}/}$$ of $\E$ sends $\P$ to itself $\mathcal{P}$.
\end{prop}
\begin{proof} As discussed above, we can make use of $\diag(\sH_\bullet\P)$ as a cofibrant resolution for $\P\in \BMod(\mathcal{P})^{*}$ and $\diag(\sH^{\textbf{O}_C}_\bullet \P)$ as (another) cofibrant model for $\P\in \Op_C(\sMod_R)$. ({Here we} note that since the operad $\textbf{O}_C$ is discrete and $\Sigma$-cofibrant, its unit map is {indeed} a cofibration). {Then} using the same argument as in the first step of the proof of Proposition \ref{p:splcsetS9}, we may assume that $\P$ is {the symmetrization of a discrete free nonsymmetric operad, or alternatively, $\P$ is} the free operad generated by a discrete free symmetric sequence. In particular, there exists a {(discrete)} cofibrant simplicial operad $\Q\in \Op_C(\Set_\Delta)$ such that $R\{\Q\} = \P$ {where $R\{-\}$ signifies the canonical left Quillen functor $\Op_C(\Set_\Delta) \lrar \Op_C(\sMod_R)$. Accordingly, we may complete  the proof just by applying Proposition \ref{p:splcsetS9}.}
\end{proof}

{In what follows, we provide additional examples for the goodness hypothesis.}

\begin{example} Combined with the \textbf{operadic Dold-Kan correspondence} obtained in \cite{Hoang}, the above proposition implies that the category $\C_{\geqslant0}({\textbf{k}})$ of connective dg ${\textbf{k}}$-modules,  {with $\textbf{k}$ being a commutative ring containing $\QQ$ (see Example \ref{ex:basecategories})}, is also good. 
\end{example}

\begin{examples} {Suppose that $\cS$ is good and suppose we are given a symmetric monoidal model category $\widetilde{\cS}$, and together with a \textbf{weak monoidal Quillen adjunction} $\F: \adjunction*{}{\cS}{\widetilde{\cS}}{} : \sU$ (cf. \cite{Schwede}). Then for every cofibrant operad $\P$ in $\cS$, it is straightforward to show that $\F(\P)$ is a good operad in $\widetilde{\cS}$ as well. Due to this, we may include more examples for the goodness hypothesis. For instance, every \textbf{spectral operad} (i.e. an operad enriched over symmetric spectra) that comes from a cofibrant simplicial (or topological) operad is good. Moreover, due to the above example, every operad in $\C({\textbf{k}})$ that comes from a cofibrant operad in $\C_{\geqslant0}({\textbf{k}})$ is good as well.}
\end{examples}

\smallskip

\subsection{Cotangent complex of enriched operads}\label{s:cotangentcplx}

\smallskip

{Suppose we are given a base category $\cS$ that is (at least) sufficient in the sense of $\S$\ref{s:convention}, and let $\P$ be a fixed $C$-colored operad in $\cS$.}

\begin{notns} To avoid confusion, in the remainder of this section, we will use the symbols $$\bigsqcup \; , \; \; \underset{\bullet}{\bigsqcup} \; , \; \; \underset{\ast}{\bigsqcup} \; \; \; \text{and} \; \; \; \underset{\circ}{\bigsqcup}$$ to denote the coproduct operations in  $\Op(\mathcal{S})$, $\BMod(\mathcal{P})$, $\BMod(\mathcal{P})^{*}$ and $\Op_C(\cS)$, respectively. On other hand, the pushouts will always be understood from the context.
\end{notns}

\begin{notns}\label{no:cotangentcomplex} We let $\rL_\P\in \T_\P \Op(\mathcal{S})$ and $\rL_\P^{b}\in \T_\P \BMod(\mathcal{P})$ respectively denote the \textbf{cotangent complexes of} $\P$ when regarded as an object of $\Op(\mathcal{S})$ and $\BMod(\mathcal{P})$. Besides that, we denote by $\rL^{\red}_\P\in \T_\P \Op_C(\mathcal{S})$ the cotangent complex of $\P$ when regarded as an object of $\Op_C(\mathcal{S})$.
\end{notns}

\begin{conv}\label{conv:reducedQcohom} From now on, by \textbf{Quillen cohomology of} $\P$ we will mean the Quillen cohomology of $\P$ when regarded as an object of $\Op(\mathcal{S})$, which is therefore classified by $\rL_\P$. On other hand, by \textbf{reduced Quillen cohomology of} $\P$ we will mean the Quillen cohomology of $\P$ when regarded as an object of $\Op_C(\mathcal{S})$, which is classified by $\rL^{\red}_\P$.
\end{conv}

By Theorem \ref{t:mainoptangent}, when $\mathcal{P}$ is $\Sigma$-cofibrant, we have a sequence of Quillen equivalences connecting the tangent categories:
$$ \adjunction*{\simeq}{\T_\P\IbMod(\P)}{\T_\P\BMod(\P)}{} \adjunction*{\simeq}{}{\T_\P\Op_C(\mathcal{S})}{} \adjunction*{\simeq}{}{\T_\P\Op(\mathcal{S})}{} .$$

\begin{notns}\label{no:severaladj} We will write $\F^{ib}_{\P} \dashv \U^{ib}_{\P}$ for the adjunction $\adjunction*{}{\T_\P\IbMod(\P)}{\T_\P\Op(\mathcal{S})}{}$ and write $\F^{b}_{\P} \dashv \U^{b}_{\P}$  for the adjunction  $\adjunction*{}{\T_\P\BMod(\P)}{\T_\P\Op(\mathcal{S})}{}$. 
\end{notns}

In order to get the desired formula of Quillen cohomology of $\P$, we will describe the derived image of $\rL_\P\in\T_\P \Op(\mathcal{S})$ under the right Quillen equivalence $ \U^{ib}_{\P} : \T_\P \Op(\mathcal{S}) \overset{\simeq}{\longrightarrow} \T_\P \IbMod(\mathcal{P})$. As the first step, we will show that the derived image of $\rL_\P$ in $\T_\P\BMod(\P)$ is up to a shift weakly equivalent to $\rL_\P^{b}$. Our result therefore extends [\cite{YonatanCotangent}, Proposition 3.2.1], {albeit through a different approach}. For our approach, the operad $\P$ is technically required to be good in the sense of Definition \ref{S9}. After having done that first step, it remains to compute the derived image of $\rL_\P^{b}$ in $\T_\P \IbMod(\mathcal{P})$.

As discussed above, we first wish to prove the following.

\begin{prop}\label{cotantgentcplxOperadandBimodule} {Suppose that $\cS$ is abundant (see Conventions \ref{convention}) and $\P$ is a good cofibrant operad (see Definition \ref{S9}). Then under the left Quillen equivalence $$\F_\P^{b} : \T_\P \BMod(\mathcal{P}) \lrarsimeq \T_\P \Op(\mathcal{S}),$$  $\rL^{b}_{\mathcal{P}} \in \T_\P \BMod(\mathcal{P})$ is identified with $\rL_{\mathcal{P}}[1] \in \T_\P \Op(\mathcal{S})$.}
\end{prop}

The proof will require several technical lemmas.

Given two $\cS$-enriched categories $\C$ and $\D$, the tensor product $\C\otimes\D$ is by definition the $\cS$-enriched category whose set of objects is $\Ob(\C\otimes\D) := \Ob(\C) \times \Ob(\D)$ and such that for every $c,c'\in \Ob(\C)$ and $d,d'\in \Ob(\D)$ we have
$$ \Map_{\C\otimes\D}((c,d),(c',d')) := \Map_\C(c,c') \otimes \Map_\D(d,d') .$$
Recall that the category of $\C$-bimodules is isomorphic to $\Fun(\C^{\op}\otimes\C,\mathcal{S})$ the category of $\cS$-valued enriched  functors on $\C^{\op}\otimes\C$. Under this identification, the functor $$\Map_\C : \C^{\op}\otimes\C \lrar \mathcal{S}, (x,y) \mapsto \Map_\C(x,y)$$ is exactly $\C$ viewed as a bimodule over itself.

\begin{lem}\label{cotantcplxofenrichedcat} {Suppose that $\cS$ is abundant and let $\mathcal{C}\in \Op(\mathcal{S})$ be a levelwise cofibrant operad that is} concentrated in arity 1. Then there is a weak equivalence $\theta_\C : \F^{b}_\C(\rL^{b}_{\mathcal{C}}) \lrarsimeq  \rL_{\mathcal{C}}[1]$ in $\T_\C \Op(\mathcal{S})$. 
\end{lem}
\begin{proof} We also regard $\C$ as an $\cS$-enriched category. The proof is then straightforward by observing that the category $\Cat(\mathcal{S})$ is already a $``$neighborhood'' of $\mathcal{C}$ in $\Op(\mathcal{S})$.  This idea is expressed as follows.  There is  a commutative square of left Quillen functors
	$$ \xymatrix{
		\T_{\Map_\C}\Fun(\C^{\op}\otimes\C,\mathcal{S}) \ar[r]\ar_{\F^{b}_{\Map_\C}}[d] & \T_\C \BMod(\mathcal{C}) \ar^{\F^{b}_{\C}}[d] \\
		\T_\C \Cat(\mathcal{S}) \ar[r] & \T_\C \Op(\mathcal{S}) \\
	}$$
	The horizontal functors are the obvious embedding functors, which preserve cotangent complexes, while the functor $\F^{b}_{\Map_\C}$ is the left Quillen equivalence appearing in [\cite{YonatanCotangent}, Theorem 3.1.14]. According to  Proposition 3.2.1 of {loc. cit.}, there is a weak equivalence $\ovl{\theta}_\C : \F^{b}_{\Map_\C}(\rL_{\Map_{\mathcal{C}}}) \lrarsimeq \rL_{\mathcal{C}}[1]$ in $\T_\C \Cat(\mathcal{S})$. Finally, the expected weak equivalence $\theta_\C$ is given by the image of $\ovl{\theta}_\C$ under the embedding functor $\T_\C \Cat(\mathcal{S}) \lrar \T_\C \Op(\mathcal{S})$.
\end{proof}

In what follows, we consider the case where $\C =\I_C$ the initial $C$-colored operad and describe the weak equivalence $\theta_{\I_C} : \F^{b}_{\I_C}(\rL^{b}_{\I_C}) \lrarsimeq \rL_{\I_C}[1]$ of the above lemma. Let us pick up several notations of [\cite{YonatanCotangent}, \S 3.2]. We denote by $*$ the category which has a single object whose endomorphism object is $1_\cS$.  Moreover, let $[1]_\cS$ denote the category with objects $0, 1$ and mapping spaces $\Map_{[1]_\cS}(0,1) = 1_\cS, \Map_{[1]_\cS}(1,0) = \emptyset$ and $\Map_{[1]_\cS}(0,0) = \Map_{[1]_\cS}(1,1) = 1_\cS$. Localizing $[1]_\cS$ at the unique non-trivial morphism $0 \rar 1$ gives us the category $[1]_\cS^{\sim}$, which is the same as $[1]_\cS$ except that $\Map_{[1]^{\sim}_\cS}(1,0) = 1_\cS$. By construction, the canonical map $[1]^{\sim}_\cS \lrar *$ is a weak equivalence. Take a factorization $$[1]_\cS \lrar \mathcal{E} \lrarsimeq [1]_\cS^{\sim}$$ of the canonical map $[1]_\cS \lrar [1]_\cS^{\sim}$ into a cofibration followed by a trivial fibration. We now obtain a sequence of maps $$*\bigsqcup* \lrar [1]_\cS \lrar \mathcal{E} \lrarsimeq [1]_\cS^{\sim} \lrarsimeq *$$
such that the first two maps are cofibrations, while the others are weak equivalences. Tensoring with $\I_C$ (viewed as an $\cS$-enriched category) produces a sequence of maps in $\Cat(\mathcal{S})$ 
$$ \I_C \bigsqcup \I_C \lrar \I_C \otimes [1]_\cS \lrar \I_C \otimes \mathcal{E} \lrarsimeq  \I_C \otimes [1]_\cS^{\sim} \lrarsimeq \I_C.$$
The last two maps are weak equivalences, while the others are again cofibrations because $\I_C$ is discrete.

{By the above discussion}, the pushout $\I_C \, \underset{\I_C \sqcup \I_C}{\bigsqcup} \, \I_C\otimes \mathcal{E}$ is a cofibrant model for $\I_C \, \underset{\I_C \sqcup \I_C}{\x{\h}{\bigsqcup}} \, \I_C \in \Op(\cS)_{\I_C//\I_C}$, and hence $\rL_{\I_C}[1]$ is given by
$$ \rL_{\I_C}[1] \x{\defi}{=} \Sigma^{\infty}(\I_C \, \underset{\I_C \sqcup \I_C}{\x{\h}{\bigsqcup}} \, \I_C) \simeq \Sigma^{\infty}(\I_C \, \underset{\I_C \sqcup \I_C}{\bigsqcup} \, \I_C\otimes \mathcal{E}) .$$
On the other hand, note that $\BMod(\I_C)$ is isomorphic to $\Coll_C(\cS)$ the category of $C$-collections, so we find that $\F^{b}_{\I_C}(\rL^{b}_{\I_C}) = \Sigma^{\infty}(\Fr(\I_C))$ where $\Fr$ is the free functor $\Coll_C(\cS) \lrar \Op_C(\cS)$ and $\Fr(\I_C)$ is regarded as an object in $\Op(\mathcal{S})_{\I_C//\I_C}$. Moreover, we can show that $\Fr(\I_C)$ is the same as the pushout $\I_C \underset{\I_C\sqcup\I_C}{\bigsqcup}\I_C \otimes [1]_\cS$. Finally, we find the expected map
\begin{equation}\label{eq:thetaIC}
	\theta_{\I_C} : \F^{b}_{\I_C}(\rL^{b}_{\I_C})  = \Sigma^{\infty}\left(\I_C \underset{\I_C\sqcup\I_C}{\bigsqcup}\I_C \otimes [1]_\cS \right) \lrarsimeq \Sigma^{\infty}\left(\I_C \underset{\I_C\sqcup\I_C}{\bigsqcup}\I_C \otimes \mathcal{E} \right) = \rL_{\I_C}[1]
\end{equation}
canonically induced by the map $[1]_\cS \lrar \mathcal{E}$. 

\smallskip

Consider $\mu : \P\circ\P \lrar \P$ as a map in $\BMod(\mathcal{P})$ and the unit map $\eta : \I_C \lrar \P$ as a map in $\Op(\mathcal{S})$. Recall that $\Sigma^{\infty}_+(\mu)$ is the image of $\mu$ under the left Quillen functor $\Sigma^{\infty}_+ : \BMod(\P)_{/\P} \lrar \T_\P \BMod(\mathcal{P})$. Note that $\Sigma^{\infty}_+(\mu)$ has already the right type since $\P\circ\P \in \BMod(\mathcal{P})$ is cofibrant. Also, $\Sigma^{\infty}_+(\eta)$ is the image of $\eta$ under the left Quillen functor $\Sigma^{\infty}_+ : \Op(\mathcal{S})_{/\P} \lrar \T_\P \Op(\mathcal{S})$ and has the right type.

\begin{lem}\label{l:muandeta} There is a weak equivalence $\F_\P^{b}(\Sigma^{\infty}_+(\mu)) \lrarsimeq \Sigma^{\infty}_+(\eta)[1]$ in $\tpop$.
\end{lem}
\begin{proof} The map $\eta : \I_C \lrar \P$ gives rise to a commutative square of left Quillen functors 
	\begin{equation}\label{eq:lbplp}
		\xymatrix{
			\T_{\mathcal{I}_C}\BMod(\mathcal{I}_C)   \ar^{\F_{\I_C}^{b}}[r]\ar_{\eta^{b}_!}[d] & \T_{\mathcal{I}_C}\Op(\mathcal{S}) \ar^{\eta^{op}_!}[d] \\
			\T_\P \BMod(\mathcal{P})\ar_{\F_\P^{b}}[r] & \T_\P \Op(\mathcal{S}) \\
		}
	\end{equation}
	Let us start with the cotangent complex $\rL^{b}_{\mathcal{I}_C}\in \T_{\mathcal{I}_C}\BMod(\mathcal{I}_C)$. Note that the functor $\BMod(\mathcal{I}_C) \lrar \BMod(\mathcal{P})$ agrees with the free $\P$-bimodule functor $\Coll_C(\cS)\lrar \BMod(\mathcal{P})$, which sends $\I_C$ to $\P\circ\P$. Thus we find that the functor $\eta^{b}_!$ sends $\rL^{b}_{\mathcal{I}_C}$ to $\Sigma^{\infty}(\P \, \underset{\bullet}{\bigsqcup} \, (\P\circ\P))\in \T_\P \BMod(\mathcal{P})$, which is a model for $\Sigma^{\infty}_+(\mu)$. The commutativity of \eqref{eq:lbplp} shows that $$\F_\P^{b}(\Sigma^{\infty}_+(\mu)) \simeq \eta^{op}_! \F_{\I_C}^{b}(\rL^{b}_{\mathcal{I}_C}).$$ On other hand, {by the discussion after Definition \ref{d:relcotangent}} we have $\Sigma^{\infty}_+(\eta)[1] = \eta^{op}_!(\rL_{\mathcal{I}_C}[1])$. Using $\Sigma^{\infty}(\I_C \, \underset{\I_C \sqcup \I_C}{\bigsqcup} \, \I_C\otimes \mathcal{E})$ as a cofibrant model for $\rL_{\mathcal{I}_C}[1]$ as discussed above, we find the desired weak equivalence given by 
	$$ \F_\P^{b}(\Sigma^{\infty}_+(\mu)) = \eta^{op}_! \F_{\I_C}^{b}(\rL^{b}_{\mathcal{I}_C}) \x{\eta^{op}_!(\theta_{\I_C})}{\underset{\simeq}{\xrightarrow{\hspace*{1.3cm}}}} \eta^{op}_!(\rL_{\mathcal{I}_C}[1]) = \Sigma^{\infty}_+(\eta)[1] $$
	where $\theta_{\I_C}$ is the weak equivalence \eqref{eq:thetaIC}.
\end{proof}

\begin{rem}\label{r:necess} It is necessary to give an explicit description of the map $\eta^{op}_!(\theta_{\I_C})$. We have that $$\F_\P^{b}(\Sigma^{\infty}_+(\mu)) \simeq \eta^{op}_! \F_{\I_C}^{b}(\rL^{b}_{\mathcal{I}_C}) = \Sigma^{\infty} \left ( \P \, \underset{\I_C}{\bigsqcup} \, \Fr(\I_C) \right ) = \Sigma^{\infty} \left ( \P \, \underset{\I_C \sqcup \I_C}{\bigsqcup} \, \I_C \otimes [1]_\cS \right ),$$ 
	and $$ \Sigma^{\infty}_+(\eta)[1] = \eta^{op}_!(\rL_{\mathcal{I}_C}[1]) = \Sigma^{\infty} \left ( \P \, \underset{\I_C}{\bigsqcup} \; (\I_C \, \underset{\I_C \sqcup \I_C}{\bigsqcup} \, \I_C\otimes \mathcal{E}) \right ) = \Sigma^{\infty} \left ( \P \, \underset{\I_C \sqcup \I_C}{\bigsqcup} \, \I_C \otimes \mathcal{E} \right ) .$$
	Under these identifications, the map $\eta^{op}_!(\theta_{\I_C})$ is given by applying $\Sigma^{\infty}$ to the map 
	$$ \P \, \underset{\I_C \sqcup \I_C}{\bigsqcup} \, \I_C \otimes [1]_\cS \lrar \P \, \underset{\I_C \sqcup \I_C}{\bigsqcup} \, \I_C \otimes \mathcal{E}$$
	canonically induced by the map $[1]_\cS \lrar  \mathcal{E}$.
\end{rem}

Consider the Quillen adjunction $\mathcal{L}^{\Sp}_\P: \adjunction*{}{\T_\P \Op_C(\mathcal{S})}{\T_\P \Op(\mathcal{S})}{} : \mathcal{R}^{\Sp}_\P$.

\begin{lem}\label{CotangentCplxOpC&OP} The left Quillen functor $\mathcal{L}^{\Sp}_\P$ sends $\rL^{\red}_{\P} \in \T_\P \Op_C(\mathcal{S})$  to $\rL_{\P/\I_C}$ the relative cotangent complex of the unit map $\eta : \I_C \lrar \P$ (cf. $\S$\ref{s:tangentcategory}).
\end{lem}
\begin{proof} Let $\P^c \lrarsimeq \P$ be a cofibrant resolution of $\P$ in $\Op_C(\mathcal{S})$. By Remark \ref{r:relcotant} we have that $\rL_{\P/\I_C} = \Sigma^{\infty}(\P \, \underset{\I_C}{\bigsqcup} \, \P^c)$. On other hand, by definition the reduced cotangent complex $\rL^{\red}_{\P} \in \T_\P \Op_C(\mathcal{S})$ is given by $\Sigma^{\infty}(\P \, \underset{\circ}{\bigsqcup} \, \P^c)$ the constant spectrum on $\P \, \underset{\circ}{\bigsqcup} \, \P^c \in \Op_C(\mathcal{S})$ considered as a $C$-colored operad over and under $\P$. {Moreover,} since $\I_C$ is {an} initial object in $\Op_C(\mathcal{S})$, {it implies that} the coproduct $\P \, \underset{\circ}{\bigsqcup} \, \P^c$, when regarded as an object in $\Op(\cS)$, {coincides with} the pushout $\P \, \underset{\I_C}{\bigsqcup} \,\P^c \in \Op(\cS)$. We hence get the conclusion.
\end{proof}

Suppose that $\P$ is cofibrant. Then $\rL_{\P/\I_C}$ is simply given by $\Sigma^{\infty}(\P \, \underset{\I_C}{\bigsqcup} \, \P)$. Nevertheless, we will need two more models for this. Let $\R$ and $\R'$ be in $\Op(\cS)$ with $$\R = \P \, \underset{\I_C}{\bigsqcup} \, (\I_C \otimes [1]_\cS) \; \; \text{and} \; \;  \R' = \P \, \underset{\I_C}{\bigsqcup} \, (\I_C \otimes \mathcal{E}) .$$ 
We then form a diagram of coCartesian squares in $\Op(\cS)$ as follows 
\begin{equation}\label{eq:pushoutsss}
	\xymatrix{
		\I_C \ar[r]\ar_{i_0}[d] & \P \ar[d] \\
		\I_C \, \bigsqcup \, \I_C   \ar[r]\ar[d] & \P \, \bigsqcup \,  \I_C  \ar[r]\ar[d] & \P \, \bigsqcup \, \P \ar[d] \\
		\I_C \otimes [1]_\cS   \ar[r]\ar[d] & \R  \ar[r]\ar[d] & \R \, \underset{\I_C}{\bigsqcup} \, \P \ar[d] \\
		\I_C \otimes \mathcal{E}  \ar[r]\ar_{\simeq}[d] & \R'  \ar[r]\ar_{\simeq}[d] & \R' \, \underset{\I_C}{\bigsqcup} \, \P \ar_{\simeq}[d] \\
		\I_C \ar[r] & \P \ar[r] & \P \, \underset{\I_C}{\bigsqcup} \, \P \\
	}
\end{equation}
The middle column is nothing but the image of the first column through the left adjoint functor $\Op(\cS)_{\I_C//\I_C} \lrar \Op(\cS)_{\P//\P}$ induced by $\eta : \I_C \lrar \P$. We consider $\R$ and $\R'$ as objects in $\Op(\cS)_{\P//\P}$ via that way. The three squares on the right hand side are considered as coCartesian squares in $\Op(\cS)_{\P//\P}$. Moreover, note that all the arrows in this diagram are cofibrations, except the three bottom vertical maps, which are all weak equivalences (the last two ones are a homotopy cobase change of the weak equivalence $\I_C \otimes \mathcal{E} \lrarsimeq \I_C$).

\begin{lem}\label{l:twomodels} The map $\R \lrar \R'$ induces a weak equivalence of spectrum objects in $\T_\P \Op(\mathcal{S})$: $$\theta_{\P/\I_C} : \Sigma^{\infty}(\R \, \underset{\I_C}{\bigsqcup} \, \P) \lrarsimeq \Sigma^{\infty}(\R' \, \underset{\I_C}{\bigsqcup} \, \P).$$  Moreover, the two are both weakly equivalent to the relative cotangent complex $\rL_{\P/\I_C}$.
\end{lem}
\begin{proof} By the weak equivalence $\R' \, \underset{\I_C}{\bigsqcup} \, \P \lrarsimeq \P \, \underset{\I_C}{\bigsqcup} \, \P$ of objects in $\Op(\cS)_{\P//\P}$, the constant spectrum $\Sigma^{\infty}(\R' \, \underset{\I_C}{\bigsqcup} \, \P)$ is indeed a model for $\rL_{\P/\I_C}$. It remains to show that $\theta_{\P/\I_C}$ is a weak equivalence. For this, we will follow the argument as in the proof of [\cite{YonatanCotangent}, Proposition 3.2.1]. It hence suffices to show that the map $\R \, \underset{\I_C}{\bigsqcup} \, \P \lrar \R' \, \underset{\I_C}{\bigsqcup} \, \P$ is (-1)-cotruncated in $\Op(\cS)_{\P//\P}$. Since the latter is a homotopy cobase change in $\Op(\cS)_{\P//\P}$ of the map $\R \lrar \R'$, we just need to show that $\R \lrar \R'$ is (-1)-cotruncated in $\Op(\cS)_{\P//\P}$. Moreover, since the map $\R \lrar \R'$ agrees with the image of the map $\I_C \otimes [1]_\cS \lrar \I_C \otimes \mathcal{E}$ through the left Quillen functor $\Op(\cS)_{\I_C//\I_C} \lrar \Op(\cS)_{\P//\P}$, the proof will be completed after showing that the latter map is (-1)-cotruncated in $\Op(\cS)_{\I_C//\I_C}$. For this, it suffices to show that the fold map $$\I_C \otimes \mathcal{E} \, \underset{\I_C \otimes [1]_\cS}{\bigsqcup} \, \I_C \otimes \mathcal{E}  \lrar \I_C \otimes \mathcal{E}$$ is a weak equivalence in $\Op(\cS)_{\I_C//\I_C}$ or in $\Cat(\cS)$, alternatively. This follows by the fact that the fold map $\mathcal{E} \, \underset{[1]_\cS}{\bigsqcup} \, \mathcal{E}  \lrar  \mathcal{E}$ is a weak equivalence in $\Cat(\cS)$, due to the convention that $\cS$ satisfies the invertibility hypothesis (cf. $\S$\ref{s:convention}).
\end{proof}

We will denote by $\rL^{b}_{\mathcal{P}/\mathcal{P}\circ \mathcal{P}}\in \T_\P \BMod(\mathcal{P})$ the relative cotangent complex of the map $\mu : \P\circ\P\lrar \P$ regarded as a map in $\BMod(\mathcal{P})$. Take a factorization $\mathcal{P}\circ \mathcal{P} \lrar \P^{b} \lrarsimeq \mathcal{P}$ in $\BMod(\mathcal{P})$ of the map $\mu$ into a cofibration followed by a weak equivalence. By Remark \ref{r:relcotant} we have that $\Sigma^{\infty}(\P\underset{\P\circ \P}{\bigsqcup}\P^{b})$ is a (cofibrant) model for $\rL^{b}_{\mathcal{P}/\mathcal{P}\circ \mathcal{P}}$. {The following lemma illustrates how we can make use of the goodness hypothesis.}

\begin{lem}\label{l:relcotanbimod}  Suppose that $\P$ is cofibrant and is good in the sense of Definition \ref{S9}. Then the left Quillen functor $\F_\P^{b} : \T_\P \BMod(\mathcal{P}) \lrar \tpop$ sends $\rL^{b}_{\mathcal{P}/\mathcal{P}\circ \mathcal{P}} \in \T_\P \BMod(\mathcal{P})$ to $ \rL_{\P/\I_C}[1]$. 
\end{lem}
\begin{proof} We first describe the image of $\rL^{b}_{\mathcal{P}/\mathcal{P}\circ \mathcal{P}}$ through the functor $\T_\P \BMod(\mathcal{P}) \lrar \T_\P \Op_C(\mathcal{S})$. It will suffice to describe the image of $\P\underset{\P\circ \P}{\bigsqcup}\P^{b}$ through the functor $\BMod(\P)_{\P//\P} \lrar \Op_C(\mathcal{S})_{\P//\P}$, which is the same as the composed functor 
	$$ \BMod(\P)_{\P//\P} \cong \BMod(\P)^{*}_{\P//\P} \x{\widetilde{\E}}{\lrar}  (\Op_C(\mathcal{S})_{\P \, \underset{\circ}{\bigsqcup} \, \P /})_{\P//\P} \cong  \Op_C(\mathcal{S})_{\P //\P} $$
	where $\widetilde{\E}$ is the functor induced by $\E : \BMod(\P)^{*} \lrar \Op_C(\mathcal{S})_{\P \, \underset{\circ}{\bigsqcup} \, \P /}$ (see $\S$\ref{sub:extracondition}). Note that the pushout $\P\underset{\P\circ \P}{\bigsqcup}\P^{b}$, when regarded as an object in $\BMod(\P)^{*}$, is exactly $\P \, \underset{\ast}{\bigsqcup} \, \P^{b}$ the  coproduct of $\P$ with $\P^{b}$ as pointed $\P$-bimodules. Denote by $\Q\in\Op_C(\mathcal{S})_{\P \, \underset{\circ}{\bigsqcup} \, \P /}$ the image of $\P^{b}$ under $\E$. So the underlying $C$-colored operad of $\E(\P \, \underset{\ast}{\bigsqcup} \, \P^{b})$ is given by the pushout $\Q \, \underset{\P \, \underset{\circ}{\bigsqcup} \, \P}{\bigsqcup} \, \P$. Since $\P$ is good, it implies that $\Q$ is a cofibrant resolution of $\P$ when considered as an object in $\Op_C(\mathcal{S})_{\P \, \underset{\circ}{\bigsqcup} \, \P /}$ and hence, the latter pushout is a model for $\Sigma(\P \, \underset{\circ}{\bigsqcup} \, \P)$ the suspension of $\P \, \underset{\circ}{\bigsqcup} \, \P$ considered as an object in $\Op_C(\mathcal{S})_{\P // \P}$. Thus we find that the derived image of $\rL^{b}_{\mathcal{P}/\mathcal{P}\circ \mathcal{P}}$ in $\T_\P \Op_C(\mathcal{S})$ is given by $\Sigma^{\infty}(\Sigma(\P \, \underset{\circ}{\bigsqcup} \, \P))$, which is nothing but $\rL^{\red}_\mathcal{P}[1]$ the suspension of $\rL^{\red}_\P$. Finally, we deduce by using Lemma \ref{CotangentCplxOpC&OP}, which proves that the derived image of $\rL^{\red}_\mathcal{P}[1]$ through the functor $\T_\P \Op_C(\mathcal{S}) \lrar \tpop$ is exactly $\rL_{\P/\I_C}[1]$. 
\end{proof}

\begin{proof}[\underline{Proof of Proposition~\ref{cotantgentcplxOperadandBimodule}}] By the definition of relative cotangent complex, we have two cofiber sequences in $\T_\P \Op(\mathcal{S})$:
	$$ \F_\P^{b}(\Sigma^{\infty}_+(\mu)) \lrar \F_\P^{b}(\rL^{b}_{\mathcal{P}}) \lrar \F_\P^{b}(\rL^{b}_{\P / \P \circ \P}) \; \; \; \text{and} \; \; \; \Sigma^{\infty}_+(\eta) \lrar \rL_{\mathcal{P}} \lrar \rL_{\mathcal{P}/\mathcal{I}_C} .$$
	Let us consider the corresponding shifted cofiber sequences:
	\begin{gather*}
		 \F_\P^{b}(\rL^{b}_{\P / \P \circ \P})[-1] \x{\alpha}{\lrar} \F_\P^{b}(\Sigma^{\infty}_+(\mu)) \lrar  \F_\P^{b}(\rL^{b}_{\mathcal{P}}), \; \;  \text{and} \\ \rL_{\mathcal{P}/\mathcal{I}_C} \x{\beta}{\lrar} \Sigma^{\infty}_+(\eta)[1] \lrar \rL_{\mathcal{P}}[1].
	\end{gather*}
The map $\alpha$ is described as follows. From Remark \ref{r:necess} we have that $$\F_\P^{b}(\Sigma^{\infty}_+(\mu)) \simeq \Sigma^{\infty} ( \P \, \underset{\I_C \sqcup \I_C}{\bigsqcup} \, \I_C \otimes [1]_\cS) .$$ On other hand, by Lemma \ref{l:relcotanbimod} we have an identification $\F_\P^{b}(\rL^{b}_{\P / \P \circ \P})[-1] \simeq \rL_{\mathcal{P}/\mathcal{I}_C}$. From Lemma \ref{l:twomodels}, we have two models for $\rL_{\mathcal{P}/\mathcal{I}_C}$ given by $\Sigma^{\infty}(\R \, \underset{\I_C}{\bigsqcup} \, \P)$ and $\Sigma^{\infty}(\R' \, \underset{\I_C}{\bigsqcup} \, \P)$.  We will exhibit the first one as a model for $\F_\P^{b}(\rL^{b}_{\P / \P \circ \P})[-1]$:
	$$ \F_\P^{b}(\rL^{b}_{\P / \P \circ \P})[-1] \simeq \Sigma^{\infty}(\R \, \underset{\I_C}{\bigsqcup} \, \P) = \Sigma^{\infty} \left ( (\P \, \bigsqcup \, \P) \underset{\I_C \sqcup \I_C}{\bigsqcup} \, \I_C \otimes [1]_\cS \right) .$$
	Then by construction, $\alpha$ is given by the map between constant spectra $$ \alpha : \Sigma^{\infty} \left ( (\P \, \bigsqcup \, \P) \underset{\I_C \sqcup \I_C}{\bigsqcup} \, \I_C \otimes [1]_\cS \right) \lrar \Sigma^{\infty} \left ( \P \, \underset{\I_C \sqcup \I_C}{\bigsqcup} \, \I_C \otimes [1]_\cS \right ) $$
	canonically induced by the fold map $\P \, \bigsqcup \, \P \lrar \P$.

	We now describe the map $\beta$. Again by Remark \ref{r:necess}, we have $\Sigma^{\infty}_+(\eta)[1] \simeq \Sigma^{\infty} \left ( \P \, \underset{\I_C \sqcup \I_C}{\bigsqcup} \, \I_C \otimes \mathcal{E} \right )$. Furthermore, when using $\Sigma^{\infty}(\R' \, \underset{\I_C}{\bigsqcup} \, \P) = \Sigma^{\infty} \left ( (\P \, \bigsqcup \, \P) \underset{\I_C \sqcup \I_C}{\bigsqcup} \, \I_C \otimes \mathcal{E} \right)$ as a model for $\rL_{\mathcal{P}/\mathcal{I}_C}$, then by construction $\beta$ is given by the map between constant spectra
	$$ \beta : \Sigma^{\infty} \left ( (\P \, \bigsqcup \, \P) \underset{\I_C \sqcup \I_C}{\bigsqcup} \, \I_C \otimes \mathcal{E} \right) \lrar \Sigma^{\infty} \left ( \P \, \underset{\I_C \sqcup \I_C}{\bigsqcup} \, \I_C \otimes \mathcal{E} \right) $$
	canonically induced by the fold map $\P \, \bigsqcup \, \P \lrar \P$ again.

	We now obtain a commutative square in $\T_\P \Op(\mathcal{S})$ 
	$$ \xymatrix{
		\Sigma^{\infty} \left ( (\P \, \bigsqcup \, \P) \underset{\I_C \sqcup \I_C}{\bigsqcup} \, \I_C \otimes [1]_\cS \right)  \ar^{ \; \; \; \; \; \alpha}[r]\ar_{\theta_{\P/\I_C}}^{\simeq}[d] & \Sigma^{\infty} \left ( \P \, \underset{\I_C \sqcup \I_C}{\bigsqcup} \, \I_C \otimes [1]_\cS \right ) \ar_{\simeq}^{\eta^{op}_!(\theta_{\I_C})}[d] \\
		\Sigma^{\infty} \left ( (\P \, \bigsqcup \, \P) \underset{\I_C \sqcup \I_C}{\bigsqcup} \, \I_C \otimes \mathcal{E} \right) \ar_{ \; \; \; \; \beta}[r] & \Sigma^{\infty} \left ( \P \, \underset{\I_C \sqcup \I_C}{\bigsqcup} \, \I_C \otimes \mathcal{E} \right) \\
	}$$
	in which the left and right vertical maps are the weak equivalences of lemmas \ref{l:twomodels} and \ref{l:muandeta}, respectively. It hence induces a weak equivalence between homotopy cofibers of the maps $\alpha$ and $\beta$. We thus obtain a natural weak equivalence $\theta_\P : {\F_\P^{b}}(\rL^{b}_{\mathcal{P}}) \lrarsimeq \rL_{\mathcal{P}}[1]$ in $\T_\P \Op(\mathcal{S})$ as expected.
\end{proof}

In the next step, we will describe the cotangent complex $\rL^{b}_{\mathcal{P}} \in \T_\P \BMod(\mathcal{P})$. We will need the following lemma.

\begin{lem}\label{bimodules&leftmodules} The forgetful functor $\sU:\BMod(\mathcal{P}) \lrar \LMod(\mathcal{P})$ is a left Quillen functor provided that $\mathcal{P}$ is $\Sigma$-cofibrant.
\end{lem}
\begin{proof} The functor $\sU$ is given by forgetting the right $\P$-action, which is linear. It implies that $\sU$ preserves colimits. By the adjoint functor theorem and by the combinatoriality of $\cS$, $\sU$ is indeed a left adjoint.

	Since $\sU$ preserves weak equivalences, the proof will be completed after showing that it preserves cofibrations. To this end, we first prove that every cofibration in $\RMod(\mathcal{P})$ has underlying map in $\Coll_C(\cS)$ being a cofibration as well. Observe that the model structure on $\RMod(\mathcal{P})$ admits a set of generating cofibrations given by $ \{i \circ \P : M \circ \P \lrar N \circ \P \}_i $ where  $i:M\rar N$ ranges over the set of those of $\Coll_C(\cS)$. Since the forgetful functor $\RMod(\mathcal{P}) \lrar \Coll_C(\cS)$ is a left adjoint, it suffices to show that each map $i \circ \P : M \circ \P \lrar N \circ \P$ is a cofibration in $\Coll_C(\cS)$. Let $\emptyset_C$ denote the initial $C$-collection, which agrees with $\emptyset_\cS$ on every level. Then, factor the map $i \circ \P$ as $$M \circ \P \lrar N \circ \emptyset_C \underset{M \circ \emptyset_C}{\bigsqcup} M \circ \P \lrar N \circ \P .$$ Since $i$ is a cofibration,  the map $i \circ \emptyset_C : M \circ \emptyset_C \lrar N \circ \emptyset_C$ is a cofibration as well. So the first map of the above composition is a cofibration. The second map is also a cofibration by [\cite{Fresse1},  Lemma 11.5.1], along with the $\Sigma$-cofibrancy of $\P$. We just showed that $i \circ \P$ is indeed a cofibration.

	It can be shown that the model structure on $\BMod(\mathcal{P})$ admits a set of generating cofibrations given by $\{\P\circ j : \P\circ K \lrar \P\circ L \}_j $ where $j:K\rar L$ ranges over the set of those of $\RMod(\mathcal{P})$. It therefore suffices to show that each map $\P\circ j:\P\circ K\lrar \P\circ L$ is a cofibration in $\LMod(\mathcal{P})$. This is now clear since  $\mathcal{P}\circ(-)$ represents the free left $\P$-module functor $\Coll_C(\cS) \lrar \LMod(\mathcal{P})$, which is a left Quillen functor, and since $j$ is a cofibration in $\Coll_C(\cS)$ as indicated above.
\end{proof}

Suppose that $\P$ is $\Sigma$-cofibrant. Properly, a model for $\rL^{b}_{\mathcal{P}}$ is given by $\Sigma^{\infty}(\mathcal{P} \, \underset{\bullet}{\bigsqcup} \, \mathcal{P}^{b})$ with $\mathcal{P}^{b}$ being a cofibrant resolution of $\P$ in $\BMod(\mathcal{P})$. But the map $\mathcal{P} \, \underset{\bullet}{\bigsqcup} \, \mathcal{P}^{b} \lrar \mathcal{P} \, \underset{\bullet}{\bigsqcup} \, \mathcal{P}$ is a weak equivalence of $\P$-bimodules due to Lemma \ref{bimodules&leftmodules}, so we may exhibit $\Sigma^{\infty}(\mathcal{P} \, \underset{\bullet}{\bigsqcup} \, \mathcal{P})$ as a model for $\rL^{b}_{\mathcal{P}}$. According to [\cite{YonatanBundle}, Corollary 2.3.3], $\rL^{b}_{\mathcal{P}}=\Sigma^{\infty}(\mathcal{P} \, \underset{\bullet}{\bigsqcup} \, \P)$ admits a \textbf{suspension spectrum replacement} simply given by fixing $\mathcal{P} \, \underset{\bullet}{\bigsqcup} \, \P$ as its value at the bidegree $(0,0)$, and hence the value at the bidegree $(n,n)$ is given by $\Sigma^{n}(\mathcal{P} \, \underset{\bullet}{\bigsqcup} \, \P)$ the $n$-suspension of $\mathcal{P} \, \underset{\bullet}{\bigsqcup} \,  \P$ in $\BMod(\P)_{\P//\P}$. We would like to exhibit this suspension spectrum as a model for $\rL^{b}_{\mathcal{P}}$ and to describe the $\P$-bimodule  $\Sigma^{n}(\mathcal{P} \, \underset{\bullet}{\bigsqcup} \, \P)$ for every $n\geqslant0$.

\begin{notns} (i) For each $n \geqslant 0$, we denote by $\sS^{n}:=\Sigma^{n}(1_\mathcal{S}\sqcup 1_\mathcal{S})\in\cS$ with the suspension $\Sigma(-)$ taken in $\mathcal{S}_{1_\mathcal{S}//1_\mathcal{S}}$, and refer to $\sS^{n}$ as the \textbf{pointed} $n$\textbf{-sphere} in $\cS$.
	
(ii)	 Furthermore, we write $\sS_C^{n}$  for the $C$-collection which has $\sS_C^{n}(c;c)=\sS^{n}$ for every $c\in C$ and agrees with $\emptyset_\cS$ on the other levels. Note that for every $n \geqslant 0$, $\sS_C^{n}$ is weakly equivalent to $\Sigma^{n}(\sS_C^{0})$ the $n$-suspension of $\sS_C^{0}\in\Coll_C(\cS)_{\I_C//\I_C}$.
\end{notns}

We can now describe $\Sigma^{n}(\mathcal{P} \, \underset{\bullet}{\bigsqcup} \, \P)$. By Lemma~\ref{bimodules&leftmodules}, the underlying left $\mathcal{P}$-module of $\Sigma^{n}(\mathcal{P} \, \underset{\bullet}{\bigsqcup} \,  \mathcal{P})$ is weakly equivalent to
	$\Sigma^{n}(\mathcal{P}  \, \underset{\diamond	}{\bigsqcup} \,  \mathcal{P}) \in \LMod(\mathcal{P})_{\mathcal{P}//\mathcal{P}}$ where $``\underset{\diamond	}{\bigsqcup}$'' refers to the coproduct operation in $\LMod(\mathcal{P})$. The good thing is that $\P$ is free (generated by $\I_C$) as a left module over itself. So we get that $\mathcal{P} \, \underset{\diamond	}{\bigsqcup} \,  \mathcal{P} \in \LMod(\mathcal{P})$ is isomorphic to $\mathcal{P} \circ \sS_C^{0}$ the free left $\P$-module generated by $\sS_C^{0}$. We have further that
	$$  \Sigma^{n}(\mathcal{P} \, \underset{\diamond	}{\bigsqcup} \,  \mathcal{P}) \simeq \Sigma^{n}(\P\circ\sS_C^{0}) \simeq  \mathcal{P}\circ \sS_C^{n} .$$ 
In particular, for each $C$-sequence $\overline{c}:=(c_1,\cdots,c_m;c)$, we have
$$ \Sigma^{n}(\mathcal{P} \, \underset{\diamond	}{\bigsqcup} \,  \mathcal{P})\,(\overline{c})\simeq\mathcal{P}\circ \sS_C^{n}\,(\overline{c})=\mathcal{P}(\overline{c})\otimes (\sS^{n})^{\otimes m}.$$

\begin{notns}\label{notn:LL_P} (i) We denote by $\widetilde{\rL}_\P\in \T_\P\IbMod(\P)$ the derived image of $\rL^{b}_\P\in \T_\P\BMod(\P)$ under the right Quillen equivalence   $\T_\P\BMod(\P) \lrarsimeq \T_\P\IbMod(\P)$ (cf. Theorem \ref{t:mainoptangent}).
	
	(ii) Furthermore, {recall from Theorem \ref{t:mainoptangent'} that when $\cS$ is stably sufficient, } we have a sequence of right Quillen equivalences
	$$ \T_\P\IbMod(\P) \overset{\Omega^{\infty}}{\underset{\simeq}{\longrightarrow}} \IbMod(\mathcal{P})_{\mathcal{P}//\mathcal{P}} \overset{\ker}{\underset{\simeq}{\longrightarrow}} \IbMod(\mathcal{P}).$$
In this situation, we will denote by $\ovl{\rL}_\P := \mathbb{R}(\ker\circ \, \Omega^{\infty})(\widetilde{\rL}_\P)$.
\end{notns}

In what follows, we give the explicit descriptions of $\widetilde{\rL}_\P$ and $\ovl{\rL}_\P$.

	\smallskip
	
\; $\bullet$ It is not hard to see that the Quillen adjunction $\adjunction*{}{\IbMod(\mathcal{P})_{\P//\P}}{\BMod(\mathcal{P})_{\P//\P}}{}$ is differentiable (see Definition~\ref{dfndifferentiable}). Thus by [\cite{YonatanBundle}, Corollary 2.4.8], $\widetilde{\rL}_\P$ is simply given by the underlying prespectrum of infinitesimal $\P$-bimodules of $\rL^{b}_{\mathcal{P}}$, and so we get $(\widetilde{\rL}_\P)_{n,n} \simeq \mathcal{P}\circ \sS_C^{n}$ for every $n \geqslant 0$.

	\smallskip

\; $\bullet$ Suppose {further} that $\cS$ is {stably sufficient}. By [\cite{YonatanBundle}, Remark 2.4.7], we have {a weak equivalence} $\mathbb{R}\Omega^{\infty}(\widetilde{\rL}_\P)\simeq \hocolim_n\Omega^{n}(\widetilde{\rL}_\P)_{n,n}$. Moreover, we find that 
	$$\mathbb{R}(\ker\circ \, \Omega^{\infty})(\widetilde{\rL}_\P)\simeq ( \hocolim_n\Omega^{n}(\mathcal{P}\circ \sS_C^{n}) ) \overset{\h}{\times_\P}  0 \simeq \hocolim_n\Omega^{n}  [\,(\mathcal{P}\circ \sS_C^{n})  \overset{\h}{\times_\P}  0 \,] $$
	in which the last desuspension $\Omega(-)$ is taken in $\IbMod(\mathcal{P})$. More explicitly, for each $\overline{c}=(c_1,\cdots,c_m;c)$ we find that
	\begin{align*} \mathbb{R}(\ker\circ \, \Omega^{\infty})(\widetilde{\rL}_\P)\,(\overline{c}) &\simeq \hocolim_n\Omega^{n} [\,(\mathcal{P}(\overline{c})\otimes (\sS^{n})^{\otimes m})  \overset{\h}{\times_{\P(\overline{c})}}  0\,] \\
		 &\simeq \P(\ovl{c})\otimes \hocolim_n\Omega^{n} [\, (\sS^{n})^{\otimes m}  \overset{\h}{\times_{1_\mathcal{S}}}  0\,]		
	\end{align*}
	 in which the last desuspension $\Omega(-)$ is now taken in $\cS$. Here the second weak equivalence is because of the fact that homotopy pullbacks in $\cS$ are also homotopy pushouts and that the functor $\mathcal{P}(\overline{c})\otimes (-)$ preserves colimits.

\smallskip

Combining  the above computations with Proposition~\ref{cotantgentcplxOperadandBimodule}, we obtain the following statement.
\begin{thm}\label{conclusionCotangentCplx} {Suppose that $\cS$ is abundant and that $\P$ is a good cofibrant operad}. Then the  right Quillen equivalence $$\U^{ib}_{\P}:\T_\P \Op(\mathcal{S}) \lrarsimeq \T_\P\IbMod(\P)$$ identifies the cotangent complex $\rL_\P$ to $\widetilde{\rL}_\P[-1]\in \T_\P\IbMod(\P)$ in which $\widetilde{\rL}_\P$ is the prespectrum with $(\widetilde{\rL}_\P)_{n,n} = \P\circ \sS_C^{n}$ for $n\geq0$. When $\mathcal{S}$ is in addition {stably abundant}, then under the right Quillen equivalence $$ \T_\P \Op(\mathcal{S}) \lrarsimeq \IbMod(\mathcal{P}) ,$$ the cotangent complex $\rL_\P$ is identified {with} $\ovl{\rL}_\P[-1]$ with $\ovl{\rL}_\P\in \IbMod(\mathcal{P})$ being given by
	$$ \ovl{\rL}_\P(\overline{c}) = \P(\ovl{c})\otimes \hocolim_n\Omega^{n} [\, (\sS^{n})^{\otimes m}  \times_{1_\mathcal{S}}^{\h}  0\,] $$
	for each $C$-sequence $\overline{c} :=(c_1,\cdots,c_m;c)$.
\end{thm}

By the definition of Quillen cohomology groups~\ref{d:Qcohom}, we {obtain} the following conclusion.

\begin{cor}\label{c:quillencohom} Suppose we are given an object $M\in \T_\P\IbMod(\P)$. Under the same assumptions as in Theorem \ref{conclusionCotangentCplx}, the $n$'th Quillen cohomology group of $\P$ with coefficients in $M$ is computed by the formula
	$$ \sH^{n}_Q(\P;M) \cong \pi_0\Map^{\h}_{\T_\P\IbMod(\P)}(\widetilde{\rL}_\P[-1],M[n]) \cong \pi_0\Map^{\h}_{\T_\P\IbMod(\P)}(\widetilde{\rL}_\P,M[n+1]).$$
	{Suppose further that $\cS$ is stably abundant}. For a given object $M\in \IbMod(\P)$, the $n$'th Quillen cohomology group of $\P$ with coefficients in $M$ is given by
	$$ \sH^{n}_Q(\P;M) \cong \pi_0\Map^{\h}_{\IbMod(\P)}(\ovl{\rL}_\P[-1],M[n]) \cong \pi_0\Map^{\h}_{\IbMod(\P)}(\ovl{\rL}_\P,M[n+1]).$$
\end{cor}

\begin{rem}\label{r:remainvalid} Suppose that the base category $\cS$ {is abundant, and in addition,} is good in the sense of Definition \ref{S9}. Then Proposition~\ref{cotantgentcplxOperadandBimodule} remains valid for any operad $\P$ that is $\Sigma$-cofibrant. Indeed, we first take $f : \Q \lrarsimeq \P$ to be a {cofibrant} resolution of $\P$ in $\Op(\mathcal{S})$ such that $f$ is a map in $\Op_C(\mathcal{S})$ (cf. Observations \ref{ob:simpleobs}). In particular, $\Q$ is good. The map $f$ gives rise to a commutative square of left Quillen equivalences
	\begin{equation}\label{eq:WLOGsigmacofib}
		\xymatrix{
			\T_\Q \BMod(\mathcal{Q}) \ar[r]^{\F_\Q^{b}}_{\simeq}\ar[d]_{\simeq} & \T_\Q \Op(\mathcal{S}) \ar[d]^{\simeq} \\
			\T_\P \BMod(\mathcal{P})\ar[r]_{\F_\P^{b}}^{\simeq} & \T_\P \Op(\mathcal{S}) \\
		}
	\end{equation}
	It is then not hard to see that the vertical functors preserve cotangent complexes and thus, the statement of Proposition~\ref{cotantgentcplxOperadandBimodule} holds for $\P$. Consequently, Theorem \ref{conclusionCotangentCplx} and moreover, all the statements in the next subsection hold for $\P$ as well.
\end{rem} 

\smallskip

\subsection{Comparison between Quillen cohomology and reduced Quillen cohomology}\label{sub:longexact}

\smallskip

{Let $\cS$ and $\P$ be as in the statement of Theorem \ref{conclusionCotangentCplx}}. The unit map $\eta : \I_C \lrar \P$ gives rise to the Quillen adjunctions
\begin{gather*}
	\eta_!^{ib} : \adjunction*{}{\mathcal{T}_{\mathcal{I}_C}\IbMod(\mathcal{I}_C)}{\T_\P \IbMod(\mathcal{P})}{} : \eta^{*}_{ib}, \;\; \text{and} \\ \eta_!^{op} : \adjunction*{}{\mathcal{T}_{\mathcal{I}_C}\Op(\mathcal{S})}{\mathcal{T}_\mathcal{P}\Op(\mathcal{S})}{} : \eta^{*}_{op}.
\end{gather*}
Moreover, there is a commutative diagram of Quillen adjunctions of the form 
\begin{center}
	\begin{tikzcd}[row sep=3.5em, column sep =3.5em]
		\mathcal{T}_{\mathcal{I}_C}\IbMod(\mathcal{I}_C) \arrow[r, shift left=1.5, "\mathcal{F}_{\mathcal{I}_C}^{ib}"] \arrow[d, shift right=1.5, "\eta_!^{ib}"'] & \mathcal{T}_{\mathcal{I}_C}\Op(\mathcal{S}) \arrow[l, shift left=1.5, "\perp"', "\mathcal{U}_{\mathcal{I}_C}^{ib}"] \arrow[d, shift left=1.5, "\eta_!^{op}"]  & \\
		\T_\P \IbMod(\mathcal{P}) \arrow[r, shift right=1.5, "\mathcal{F}^{ib}_\P"'] \arrow[u, shift right=1.5, "\dashv", "\eta^{*}_{ib}"'] & \mathcal{T}_\mathcal{P}\Op(\mathcal{S})  \arrow[l, shift right=1.5, "\downvdash", "\mathcal{U}^{ib}_\P"'] \arrow[u, shift left=1.5, "\vdash"', "\eta^{*}_{op}"] &
	\end{tikzcd}
\end{center}

The following is an analogue of [\cite{YonatanCotangent}, Corollary 3.2.9].
\begin{lem}\label{cofibersequence1} There is a homotopy cofiber sequence in $\T_\P\IbMod(\P)$ of the form
	\begin{equation}\label{eq:cofib2}
		\RR \U^{ib}_{\P,C}(\rL_\P^{\red}) \lrar \LL\eta_!^{ib}\,(\widetilde{\rL}_{\I_C}) \lrar \widetilde{\rL}_{\P}
	\end{equation}
	where $\U^{ib}_{\P,C}$ is the right Quillen equivalence $\T_\P \Op_C(\mathcal{S}) \lrarsimeq \T_\P\IbMod(\P)$ appearing in Theorem \ref{t:mainoptangent}.
\end{lem}
\begin{proof} Theorem \ref{conclusionCotangentCplx} proves the existence of weak equivalences $\widetilde{\theta}_{\I_C} : \widetilde{\rL}_{\I_C}[-1]\overset{\simeq}{\lrar}\mathbb{R}\mathcal{U}_{\mathcal{I}_C}^{ib}(\rL_{\mathcal{I}_C})$ and $ \widetilde{\theta}_\P:\widetilde{\rL}_{\P}[-1]\overset{\simeq}{\lrar}\mathbb{R}\mathcal{U}^{ib}_\P(\rL_{\mathcal{P}})$ in $\mathcal{T}_{\mathcal{I}_C}\IbMod(\mathcal{I}_C)$ and $\mathcal{T}_{\P}\IbMod(\P)$, respectively. Applying $\LL\eta_!^{op}$ to $\widetilde{\theta}_{\I_C}^{ad}:\LL\mathcal{F}_{\mathcal{I}_C}^{ib}\,(\widetilde{\rL}_{\I_C})\lrarsimeq \rL_{\mathcal{I}_C}[1]$ (i.e., the adjoint of $\widetilde{\theta}_{\I_C}$), and taking then the adjoint of the result, we obtain a weak equivalence in $\T_\P\IbMod(\P)$ of the form $\LL\eta_!^{ib}\,(\widetilde{\rL}_{\I_C})\overset{\simeq}{\lrar}\RR\mathcal{U}^{ib}_\P\,\LL\eta_!^{op} (\rL_{\mathcal{I}_C}[1])$.

	On other hand, by the definition of relative cotangent complex, there is a cofiber sequence in $\tpop$ of the form $$ \rL_{\P/\I_C} \lrar \LL\eta_!^{op}(\rL_{\I_C})[1] \lrar \rL_\P[1]  .$$
	By applying $\mathbb{R}\mathcal{U}^{ib}_\P$ to the latter and by the first paragraph, we get a cofiber sequence in $\T_\P\IbMod(\P)$:
	$$  \mathbb{R}\mathcal{U}^{ib}_\P(\rL_{\P/\I_C}) \lrar \LL\eta_!^{ib}\,(\widetilde{\rL}_{\I_C})\lrar \widetilde{\rL}_{\P} .$$
	Lemma~\ref{CotangentCplxOpC&OP} shows that  there is a weak equivalence $\RR \U^{ib}_{\P,C}(\rL_\P^{\red}) \lrarsimeq \mathbb{R}\mathcal{U}^{ib}_\P(\rL_{\P/\I_C})$ in $\T_\P\IbMod(\P)$. So we get the desired cofiber sequence. 
\end{proof}

We end this section with the following statement.

\begin{thm}\label{c:longexseq1}  Given an object $M\in \T_\P\IbMod(\P)$, there is a long exact sequence of abelian groups of the form
	$$ \cdots \longrightarrow \sH^{n-1}_Q (\mathcal{P};M) \lrar  \sH^{n}_{Q,r} (\mathcal{P};M) \lrar \sH^{n}_{Q,\red} (\mathcal{P};M) \lrar \sH^{n}_Q (\mathcal{P};M) \lrar \sH^{n+1}_{Q,r} (\mathcal{P};M)  \lrar \cdots $$
	where $\sH^{\bullet}_{Q,r} (\mathcal{P};-)$ refers to Quillen cohomology group of $\P$ when regarded as a right module over itself, while $\sH^{\bullet}_Q (\mathcal{P};-)$ refers to Quillen cohomology group of $\P$ and $\sH^{\bullet}_{Q,\red} (\mathcal{P};-)$ refers to reduced Quillen cohomology group of $\P$ (cf. Conventions \ref{conv:reducedQcohom}).
\end{thm}
\begin{proof} In the cofiber sequence of Lemma \ref{cofibersequence1}, by notation $\RR \U^{ib}_{\P,C}(\rL_\P^{\red})$ classifies the reduced Quillen cohomology of $\P$, while $\widetilde{\rL}_{\P}$ classifies the Quillen cohomology of $\P$ by Corollary \ref{c:quillencohom}. This cofiber sequence will hence give rise to the desired long exact sequence after having that $\LL\eta_!^{ib}\,(\widetilde{\rL}_{\I_C})$ classifies the Quillen cohomology of $\P$ when regarded as a right module over itself. To see this, we first consider the Quillen adjunction $\adjunction*{}{\T_\P\RMod(\P)}{\T_\P\IbMod(\P)}{}$ which is induced by the free-forgetful adjunction $\adjunction*{}{\RMod(\P)}{\IbMod(\P)}{}$. We denote by $\rL^{r}_\P \in \T_\P\RMod(\P)$ the cotangent complex of $\P$ when regarded as a right module over itself. It will suffice to prove that the derived image of $\rL^{r}_\P$ in $\T_\P\IbMod(\P)$ is weakly equivalent to $\LL\eta_!^{ib}\,(\widetilde{\rL}_{\I_C})$. For this last claim, observe first that $\eta_!^{ib}$ is the same as the functor $\T_{\I_C}\Coll_C(\cS) \lrar \T_\P\IbMod(\P)$ induced by the free functor $\Coll_C(\cS) \lrar \IbMod(\P)$. Moreover, under the identification $\mathcal{T}_{\mathcal{I}_C}\IbMod(\mathcal{I}_C) \simeq \T_{\I_C}\Coll_C(\cS)$, the object $\widetilde{\rL}_{\I_C}$ is nothing but the cotangent complex of $\mathcal{I}_C$ when regarded as an object of $\Coll_C(\cS)$, whose derived image through the left Quillen functor $\T_{\I_C}\Coll_C(\cS) \lrar \T_\P\RMod(\P)$ is exactly $\rL^{r}_\P$. The proof is therefore completed. 
\end{proof}

\begin{rem} The {gap} between reduced Quillen cohomology and Quillen cohomology of $\P$ is {really accessible} because the abelian groups $\sH^{\bullet}_{Q,r} (\mathcal{P};-)$ are computable. Indeed, when considering $\mathcal{P}$ as a right module over itself, it is free generated by $\I_C$. So we have a canonical isomorphism 
	$$ \sH^{n}_{Q,r} (\mathcal{P};M) \cong \sH^{n}_{Q,col} (\I_C;\ovl{M}) $$
	where the right hand side is the $n$'th Quillen cohomology group of $\I_C$ regarded as a $C$-collection, with coefficients in the derived image of $M$ in $\T_{\I_C}\Coll_C(\cS)$. When $\cS$ is further {stably abundant}, the cotangent complex of $\I_C$ is identified {with}  $\I_C$ itself under the right Quillen equivalence $ \T_{\I_C}\Coll_C(\cS) \lrarsimeq \Coll_C(\cS) $ (cf. Theorem \ref{t:mainoptangent'}). Therefore, in this situation we have that
	\begin{align*}
		 \sH^{n}_{Q,col} (\I_C;\ovl{M}) &\cong  \pi_0\Map^{\h}_{\Coll_C(\cS)}(\I_C,\RR\ker\Omega^{\infty}(\Sigma^{n}\ovl{M})) \\ &\cong \bigsqcap_{c\in C} \pi_0\Map^{\h}_{\cS}(1_\cS,\Sigma^{n}(\ovl{M}(c;c)\times^{\h}_{1_\cS}0))  .
	\end{align*}
\end{rem}

\smallskip

\section{Quillen cohomology of simplicial operads}\label{s:simplicialoperad}

\smallskip

The category $\Set_\Delta$ of simplicial sets, together with the standard (Kan-Quillen) model structure, {is abundant in the sense} of Conventions \ref{convention}, and moreover, is good in the sense of Definition \ref{S9} (cf. Example \ref{ex:basecategories} and Proposition \ref{p:splcsetS9}). So we may inherit the results of $\S$\ref{s:cotangentcplx} for $\Sigma$-cofibrant simplicial operads (see also Remark \ref{r:remainvalid}).

In the first subsection, we revisit the \textbf{straightening} and \textbf{unstraightening} constructions, according to [\cite{Luriehtt}, \S 2.2.1]. Given a simplicial (co)presheaf  $\F$ over a simplicial category $\C$, the unstraightening of $\F$ is in particular a simplicial set over $\sN\C$. We shall first describe the simplices of the unstraightening of $\F$, and then describe its \textbf{spaces of {right} morphisms}.

According to the work of Y. Harpaz, J. Nuiten and M. Prasma (\cite{YonatanCotangent}), the cotangent complex of an $\infty$-category (or a fibrant simplicial category) can be represented as a spectrum valued functor on its \textbf{twisted arrow $\infty$-category} (see Theorem \ref{t:Qcohominftycat}). In the second subsection, we introduce the construction of twisted arrow $\infty$-categories of simplicial operads. For a fibrant simplicial operad $\P$, the twisted arrow $\infty$-category $\Tw(\P)$ is given by the unstraightening of the simplicial copresheaf $\P : \textbf{Ib}^{\mathcal{P}} \lrar \Set_\Delta$, which encodes the data of $\P$ as an infinitesimal bimodule over itself (see $\S$\ref{s:operadicmodules}).

In the last subsection, for the main purpose we will show that the cotangent complex of a simplicial operad can be represented as a spectrum valued functor on its twisted arrow $\infty$-category. The cotangent complex of the little discs operads will be also particularly regarded. 

\subsection{Unstraightening of simplicial (co)presheaves}

\smallskip

We denote by $\mathfrak{C}[-] : \Set_\Delta \lrar \Cat(\Set_\Delta)$ the \textbf{rigidification functor}, which is left adjoint to the simplicial nerve functor $\sN$. Recall by definition that $\mathfrak{C}[\Delta^{n}]$ is the simplicial category whose set of objects is $[n] = \{0,\cdots,n\}$ and whose mapping spaces are given by $\Map_{\mathfrak{C}[\Delta^{n}]}(i,j) := \sN \rP_{i,j}$ the nerve of the poset $$\rP_{i,j} = \{A \, | \, \{i,j\} \subseteq A \subseteq [i,j] \} ,$$
with $[i,j] := \{i,i+1,\cdots,j\}$. (In particular, $\Map_{\mathfrak{C}[\Delta^{n}]}(i,j) = \emptyset$ when $i>j$). The composition in $\mathfrak{C}[\Delta^{n}]$ is induced by the maps  $\rP_{i,j} \times \rP_{j,k} \lrar \rP_{i,k}$ taking $(A,B)$ to $A \cup B$.

Depending if one wants to  $``$unstraighten'' simplicial presheaves or copresheaves, one would need to use the \textbf{contravariant} or \textbf{covariant unstraightening} functor, respectively.  Suppose we are given a simplicial set $S$ and a simplicial category $\C$. Let $\phi : \mathfrak{C}[S] \lrar \C$ be a simplicial functor. For each simplicial set $X$ over $S$, one takes two simplicial categories
$$\M_X^{\triangleleft} := \mathfrak{C}[X^{\triangleleft}]  \underset{\mathfrak{C}[X]}{\bigsqcup} \C \; \; \; \text{and} \; \; \; \M_X^{\triangleright} := \mathfrak{C}[X^{\triangleright}]  \underset{\mathfrak{C}[X]}{\bigsqcup} \C$$
where $X^{\triangleleft}$ and $X^{\triangleright}$ are respectively the left and right cones of $X$. We will always use the letter $``v$''  for the cone point, for both left and right cones.

The \textbf{covariant straightening functor} associated to $\phi$
$$ \St_{\phi}^{\triangleleft} : (\Set_\Delta)_{/S} \lrar \Fun(\C , \Set_\Delta)$$
is defined by sending each simplicial set $X$ over $S$ to the functor $\Map_{\M_X^{\triangleleft}}(v,-) : \C \lrar \Set_\Delta$. The functor $\St_{\phi}^{\triangleleft}$ admits a right adjoint denoted by $\Un_{\phi}^{\triangleleft}$ and called the \textbf{covariant unstraightening functor}. In fact, the adjunction $\St_{\phi}^{\triangleleft} \dashv \Un_{\phi}^{\triangleleft}$ forms a Quillen adjunction when one endows the category $\Fun(\C , \Set_\Delta)$ with the projective model structure and $(\Set_\Delta)_{/S}$ with the \textbf{covariant model structure}. Moreover, this adjunction is a Quillen equivalence as long as $\phi$ is a weak equivalence of simplicial categories. (See \cite{Luriehtt}, Chapter 3 for more details).

Dually, the \textbf{contravariant straightening functor} associated to $\phi$
$$ \St_{\phi}^{\triangleright} : (\Set_\Delta)_{/S} \lrar \Fun(\C^{\op} , \Set_\Delta) $$
is defined by sending each simplicial set $X$ over $S$ to the functor $\Map_{\M_X^{\triangleright}}(-,v) : \C^{\op} \lrar \Set_\Delta$. We denote by $\Un_{\phi}^{\triangleright}$ the right adjoint of $\St_{\phi}^{\triangleright}$ and refer to it as the \textbf{contravariant unstraightening functor}. The adjunction $\St_{\phi}^{\triangleright} \dashv \Un_{\phi}^{\triangleright}$ forms a Quillen adjunction when one endows the category $\Fun(\C^{\op} , \Set_\Delta)$ with the projective model structure and $(\Set_\Delta)_{/S}$ with the \textbf{contravariant model structure}, and moreover, {$\St_{\phi}^{\triangleright} \dashv \Un_{\phi}^{\triangleright}$} becomes a Quillen equivalence if $\phi$ is a weak equivalence.

\begin{notns}\label{no:unstoverS}
	
		(i) We are mainly interested in the case where $S = \sN \C$ and $\phi$ is given by the counit map $\varepsilon_\C : \mathfrak{C}[\sN \C] \lrar \C$. The corresponding covariant (resp. contravariant) straightening-unstraightening adjunction will be then denoted by $\St_{\C}^{\triangleleft} \dashv \Un_{\C}^{\triangleleft}$ (resp. $\St_{\C}^{\triangleright} \dashv \Un_{\C}^{\triangleright}$).  
		
		(ii) When $\C = \mathfrak{C}[S]$ and $\phi$ is the identity functor, the corresponding covariant (resp. contravariant) straightening-unstraightening adjunction will be denoted by $\St_{S}^{\triangleleft} \dashv \Un_{S}^{\triangleleft}$ (resp. $\St_{S}^{\triangleright} \dashv \Un_{S}^{\triangleright}$).
	
\end{notns}

\begin{rem}\label{r:Un_C} Note that $\Un_{\C}^{\triangleleft}$ is the same as the composed functor $$\Fun(\C , \Set_\Delta) \x{\varepsilon_\C^{*}}{\xrightarrow{\hspace*{0.7cm}}} \Fun(\mathfrak{C}[\sN \C] , \Set_\Delta) \x{\Un_{\sN\C}^{\triangleleft}}{\xrightarrow{\hspace*{0.7cm}}} (\Set_\Delta)_{/\sN\C} ,$$ (this follows from \cite{Luriehtt}, 2.2.1.1(2)). The same statement holds for the contravariant case.
\end{rem}

Let $\F : \C \lrar \Set_\Delta$ be a simplicial functor. We wish to understand the structure of $\Un_{\C}^{\triangleleft}(\F)$ as a simplicial set over $\sN\C$. For this, we will follow \cite{Rezk1}, yet in the opposite convention. For each $n\in\NN$, one establishes a simplicial functor 
$$ \mathfrak{D}^{\triangleleft}_{\Delta^{n}} : \mathfrak{C}[\Delta^{n}] \lrar \Set_\Delta $$
given by sending each $i\in [n]$ to $\mathfrak{D}^{\triangleleft}_{\Delta^{n}}(i):=\sN \rP^{\triangleleft}(i)$ the nerve of the poset $$\rP^{\triangleleft}(i) := \{A \, | \, \{i\}\subseteq A \subseteq [0,i] \, \}.$$ The structure maps of simplicial functor are defined by applying the union operation of subsets in an obvious way. Moreover, for each map $\delta : \Delta^{m}\lrar\Delta^{n}$, one defines a natural transformation
$ \mathfrak{D}^{\triangleleft}_{\delta} : \mathfrak{D}^{\triangleleft}_{\Delta^{m}} \lrar \mathfrak{D}^{\triangleleft}_{\Delta^{n}} \circ \mathfrak{C}[\delta] $
of the simplicial functors $\mathfrak{C}[\Delta^{m}] \lrar \Set_\Delta$ given at each $i\in [m]$ by the map $\mathfrak{D}^{\triangleleft}_{\delta}(i) : \mathfrak{D}^{\triangleleft}_{\Delta^{m}}(i)\lrar\mathfrak{D}^{\triangleleft}_{\Delta^{n}}(\delta i)$ which is induced by the map of posets $$ \rP^{\triangleleft}(i) \lrar \rP^{\triangleleft}(\delta i) \; , \; S \mapsto \delta(S) .$$

\begin{cons}\label{con:Un_CF} The data of an $n$-simplex of $\Un_{\C}^{\triangleleft}(\F)$ consist of

 (i) an $n$-simplex  $\varphi\in\sN(\C)$, i.e., a functor $\varphi : \mathfrak{C}[\Delta^{n}]\lrar \C$, and

 (ii) a natural transformation $t:\mathfrak{D}^{\triangleleft}_{\Delta^{n}}\lrar \F\circ \varphi$ between simplicial functors $\mathfrak{C}[\Delta^{n}] \lrar \Set_\Delta$.

	\noindent For each map $\delta : \Delta^{m}\lrar\Delta^{n}$, the corresponding structure map $\Un_{\C}^{\triangleleft}(\F)_n\lrar\Un_{\C}^{\triangleleft}(\F)_m$ is given by sending each pair $(\varphi,t)\in \Un_{\C}^{\triangleleft}(\F)_n$  to the pair
	$$ \left(\mathfrak{C}[\Delta^{m}] \x{\mathfrak{C}[\delta]}{\xrightarrow{\hspace*{0.7cm}}} \mathfrak{C}[\Delta^{n}] \x{\varphi}{\xrightarrow{\hspace*{0.7cm}}} \C \, \; , \, \; \mathfrak{D}^{\triangleleft}_{\Delta^{m}} \x{\mathfrak{D}^{\triangleleft}_{\delta}}{\xrightarrow{\hspace*{0.7cm}}} \mathfrak{D}^{\triangleleft}_{\Delta^{n}} \circ \mathfrak{C}[\delta] \x{t\circ \Id}{\xrightarrow{\hspace*{0.7cm}}} \F\circ \varphi \circ \mathfrak{C}[\delta] \right).$$
\begin{proof} Observe that for each map $\varphi_z: \Delta^{n} \lrar \sN\C$ there is a commutative square of right adjoint functors 
	$$ \xymatrix{
		\Fun(\C , \Set_\Delta) \ar^{ \; \; \; \; \; \; \Un_{\C}^{\triangleleft}}[r]\ar_{\varphi_z^{*}}[d] & (\Set_\Delta)_{/\sN\C} \ar^{\varphi_z^{*}}[d] \\
		\Fun(\mathfrak{C}[\Delta^{n}],\Set_\Delta) \ar_{ \; \; \; \; \; \; \; \; \; \Un_{\Delta^{n}}^{\triangleleft}}[r] & (\Set_\Delta)_{/\Delta^{n}} \\
	} $$
	(this follows by combining the two parts of [\cite{Luriehtt}, Proposition 2.2.1.1], along with Remark \ref{r:Un_C}). This implies the existence of a Cartesian square of the form 
	$$ \xymatrix{
		\Un^{\triangleleft}_{\Delta^{n}}(\F \circ \varphi_z) \ar[r]\ar[d] & \Delta^{n} \ar^{\varphi_z}[d] \\
		\Un_{\C}^{\triangleleft}(\F) \ar[r] & \sN\C. \\
	} $$
	So the data of an $n$-simplex $z \in \Un_{\C}^{\triangleleft}(\F)_n$ is equivalent to that of a map $\varphi_z : \Delta^{n} \lrar \sN\C$ and a section $\Delta^{n} \lrar \Un^{\triangleleft}_{\Delta^{n}}(\F \circ \varphi_z)$. By adjunction, the latter is identified {with} a natural transformation $\St^{\triangleleft}_{\Delta^{n}}(\Id_{\Delta^{n}}) \lrar \F \circ \varphi_z$. But by definition  $\St^{\triangleleft}_{\Delta^{n}}(\Id_{\Delta^{n}}) = \Map_{\mathfrak{C}[(\Delta^{n})^{\triangleleft}]}(v,-)$. Hence, it remains to establish for each $n$ an isomorphism  $$ \mathfrak{D}^{\triangleleft}_{\Delta^{n}}  \lrarsimeq \Map_{\mathfrak{C}[(\Delta^{n})^{\triangleleft}]}(v,-)$$ 
	between simplicial functors $\mathfrak{C}[\Delta^{n}] \lrar \Set_\Delta$ compatible with simplicial maps $\Delta^{m} \lrar \Delta^{n}$. For this, note that under the identification $(\Delta^{n})^{\triangleleft} \cong \Delta^{n+1}$, the functor $\Map_{\mathfrak{C}[(\Delta^{n})^{\triangleleft}]}(v,-)$ is isomorphic to the functor $[n] \ni i \mapsto \Map_{\mathfrak{C}[\Delta^{n+1}]}(0,i+1) = \sN \rP_{0,i+1}$. For each $i \in [n]$, we define an isomorphism of posets $$\rP^{\triangleleft}(i) \lrarsimeq \rP_{0,i+1} \; \; , \; \; A \mapsto \{0\} \sqcup (A+1)$$ with $A+1 := \{a+1 \, | \, a \in A \}$. This indeed gives us the desired isomorphism $\mathfrak{D}^{\triangleleft}_{\Delta^{n}}  \lrarsimeq \Map_{\mathfrak{C}[(\Delta^{n})^{\triangleleft}]}(v,-)$.
\end{proof}		
\end{cons}

\begin{rem}\label{r:vertexedge}  Unwinding definition, each vertex of $\Un_{\C}^{\triangleleft}(\F)$ is the choice of an object $x \in \Ob(\C)$ and a vertex $\mu \in \F(x)$. Let $x,y$ be two objects of $\C$. The data of an edge of $\Un_{\C}^{\triangleleft}(\F)$ from $\mu \in \F(x)$ to $\nu \in \F(y)$ consist of a vertex $\alpha \in \Map_\C(x,y)$ and an edge of $\F(y)$ of the form $t : \nu \lrar \alpha^{*}(\mu)$ where $\alpha^{*} : \F(x) \lrar \F(y)$ is the map induced by the simplicial functor structure of $\F : \C \lrar \Set_\Delta$.
\end{rem}

Let $\F' : \C^{\op} \lrar \Set_\Delta$ be a simplicial functor. Similarly as in the covariant case, one can get an explicit description of $\Un_{\C}^{\triangleright}(\F')$ the contravariant unstraightening of $\F'$. For each $n\in\NN$, one defines a simplicial functor 
$$ \mathfrak{D}^{\triangleright}_{\Delta^{n}} : \mathfrak{C}[\Delta^{n}]^{\op} \lrar \Set_\Delta $$
taking each $i\in [n]$ to $\mathfrak{D}^{\triangleright}_{\Delta^{n}}(i):=\sN \rP_n^{\triangleright}(i)$ the nerve of the poset $$\rP_n^{\triangleright}(i) := \{A \, | \, \{i\}\subseteq A \subseteq [i,n] \, \}.$$ For each map $\delta : \Delta^{m}\lrar\Delta^{n}$, there is also a natural transformation $\mathfrak{D}^{\triangleright}_{\delta} : \mathfrak{D}^{\triangleright}_{\Delta^{m}} \lrar \mathfrak{D}^{\triangleright}_{\Delta^{n}} \circ \mathfrak{C}[\delta]^{\op}$ of the simplicial functors $\mathfrak{C}[\Delta^{m}]^{\op} \lrar \Set_\Delta$. The structure maps are defined naturally as well as in the covariant case (see Construction \ref{con:Un_CF}).

\begin{cons}\label{con:Un_CF'}  The data of an $n$-simplex of $\Un_{\C}^{\triangleright}(\F')$ consist of

 (i) an $n$-simplex  $\varphi\in\sN(\C)$, i.e., a functor $\varphi : \mathfrak{C}[\Delta^{n}] \lrar \C$, and

 (ii) a natural transformation $t:\mathfrak{D}^{\triangleright}_{\Delta^{n}}\lrar \F'\circ \varphi^{\op}$ between simplicial functors $\mathfrak{C}[\Delta^{n}]^{\op} \lrar \Set_\Delta$.

	\noindent The simplicial structure maps of $\Un_{\C}^{\triangleright}(\F')$ are induced by the natural transformations $\mathfrak{D}^{\triangleright}_{\delta}$ mentioned above. (See also \cite{Rezk1}).
\end{cons}

When $\C$ is the terminal category $* \in \Cat(\Set_\Delta)$, we obtain the self-adjunctions on $\Set_\Delta$:
$$ \St_{*}^{\triangleleft} : \adjunction*{}{\Set_\Delta}{\Set_\Delta}{} : \Un_{*}^{\triangleleft} \; \; \; \text{and} \; \; \; \St_{*}^{\triangleright} : \adjunction*{}{\Set_\Delta}{\Set_\Delta}{} : \Un_{*}^{\triangleright} .$$
Note that $\St_{*}^{\triangleright}\dashv\Un_{*}^{\triangleright}$ agrees with the adjunction $|-|_{Q^{\bullet}} \dashv \Sing_{Q^{\bullet}}$ of [\cite{Luriehtt}, \S 2.2.2].

\begin{rem}\label{r:XKan} For a simplicial set $X$, there is a natural map $X \lrar \Un_{*}^{\triangleright}(X)$, which is in fact a weak equivalence as long as $X$ is a Kan complex (see [\cite{Luriehtt}, 2.2.2.7, 2.2.2.8]). On other hand, it can be shown that there exists a natural isomorphism $\Un_{*}^{\triangleleft}(X) \cong \Un_{*}^{\triangleright}(X)^{\op}$. So we get a natural map $X^{\op} \lrar \Un_{*}^{\triangleleft}(X)$, which is a weak equivalence when $X$ is Kan. In other words, while the (derived) contravariant unstraightening $\Un_{*}^{\triangleright}$ is weakly equivalent to the identity functor, the covariant one $\Un_{*}^{\triangleleft}$ is weakly equivalent to the opposite functor $X \mapsto X^{\op}$.
\end{rem}

We will need the following in the last subsection.

\begin{rem}\label{r:fiberkan} As in [\cite{Luriehtt}, Remark 2.2.2.11], for a given object $x \in \C$, there is a canonical isomorphism $$\Un_{\C}^{\triangleleft}(\F) \times_{\sN\C} \{x\} \cong \Un_{*}^{\triangleleft}(\F(x)) .$$ This follows by the compatibility of the unstraightening functor with taking base change along the map $\{x\} \lrar \sN\C$. Combining with Remark \ref{r:XKan}, we get a weak equivalence
	\begin{equation}\label{eq:r:XKan}
		\F(x)^{\op} \lrarsimeq \Un_{\C}^{\triangleleft}(\F) \times_{\sN\C} \{x\}
	\end{equation}
	whenever $\F(x)$ is a Kan complex.
\end{rem}

Let $x,y\in \Ob(\C)$ be two objects of $\C$. Suppose we are given two vertices $\mu \in \F(x)$ and $\nu \in \F(y)$ regarded as vertices of $\Un_{\C}^{\triangleleft}(\F)$ (see Remark \ref{r:vertexedge}). We wish to give a convenient model for $\Hom^{\rR}_{\Un_{\C}^{\triangleleft}(\F)}(\mu,\nu)$ the space of right morphisms from $\mu$ to $\nu$ in $\Un_{\C}^{\triangleleft}(\F)$. Recall by definition that an $n$-simplex of $\Hom^{\rR}_{\Un_{\C}^{\triangleleft}(\F)}(\mu,\nu)$ is an $(n+1)$-simplex $T : \Delta^{n+1} \lrar  \Un_{\C}^{\triangleleft}(\F)$ of $\Un_{\C}^{\triangleleft}(\F)$ such that $d_{n+1}T$ is degenerate on $\mu$ and $\restr{T}{\Delta^{\{n+1\}}} = \nu$. According to Construction \ref{con:Un_CF}, the data of $T$ consist of:

$\bullet$ An $(n+1)$-simplex $H : \Delta^{n+1} \lrar  \sN\C$ satisfying that $d_{n+1}H$ is degenerate on $x$ and $\restr{H}{\Delta^{\{n+1\}}} = y$. In other words, $H$ is nothing but an $n$-simplex of $\Hom^{\rR}_{\sN\C}(x, y)$.

$\bullet$ In addition, a natural transformation $t : \mathfrak{D}^{\triangleleft}_{\Delta^{n+1}} \lrar \F \circ H$ between simplicial functors $\mathfrak{C}[\Delta^{n+1}] \lrar \Set_\Delta$ satisfying that for every $i\in \{0,\cdots,n\}$ the map $$t(i) : \mathfrak{D}^{\triangleleft}_{\Delta^{n+1}}(i) \lrar (\F \circ H)(i) = \F(x)$$ collapses its domain to the vertex $\mu$ and that the initial vertex of $t(n + 1)$ agrees with $\nu$. (Note that $t(n + 1)$ performs an $(n + 1)$-cube in $\F(y)$).

Let us now denote by $\rho_\mu$ the composed map $$\Map_\C(x,y) \x{\F}{\lrar} \Map_{\Set_\Delta}(\F(x),\F(y)) \x{ev_\mu}{\lrar} \F(y)$$ with $ev_\mu$ being the evaluation at $\mu$. By [\cite{Luriehtt}, 2.2.2.13], there is an isomorphism $\Hom^{\rR}_{\sN\C}(x, y) \cong \Un_{*}^{\triangleright}\Map_\C(x,y)$. So we can form a canonical map 
\begin{equation}\label{eq:mapH}
	\Hom^{\rR}_{\sN\C}(x, y) \cong \Un_{*}^{\triangleright}\Map_\C(x,y) \x{\Un_{*}^{\triangleright}(\rho_\mu)}{\xrightarrow{\hspace*{1cm}}}  \Un_{*}^{\triangleright}\F(y) \; .
\end{equation}

\begin{lem}\label{l:mappingUnF} There is a canonical isomorphism 
	$$ \Hom^{\rR}_{\Un_{\C}^{\triangleleft}(\F)}(\mu,\nu) \cong (\Un_{*}^{\triangleright}\F(y))_{\nu/} \times_{\Un_{*}^{\triangleright}\F(y)} \Hom^{\rR}_{\sN\C}(x, y) =: \mathbb{P}_{\mu,\nu} .$$
\end{lem}

\begin{proof} The proof is straightforward using the above analyses.
\end{proof}

Suppose that both $\C$ and $\F : \C \lrar \Set_\Delta$ are fibrant. Then the map $\Un_{\C}^{\triangleleft}(\F) \lrar \sN\C$ is a left fibration and $\sN\C$ is an $\infty$-category. So $\Un_{\C}^{\triangleleft}(\F)$ is an $\infty$-category as well. Our main interest in this subsection is as follows.
\begin{prop}\label{p:mappingunCF}  Let $x,y\in \Ob(\C)$ be two objects of $\C$. Given two vertices $\mu \in \F(x)$ and $\nu \in \F(y)$ regarded as vertices of $\Un_{\C}^{\triangleleft}(\F)$, there is a homotopy equivalence 
	$$ \{\nu\} \times_{\F(y)}^{\h} \Map_\C(x,y) \lrarsimeq \Map_{\Un_{\C}^{\triangleleft}(\F)}(\mu,\nu) .$$
	
\end{prop}
\begin{proof} We make use of the pullback $\F(y)_{\nu/} \times_{\F(y)} \Map_\C(x,y)$ as a model for the left hand side. Since $\F(y)$ and $\Map_\C(x,y)$ are Kan complexes, that pullback is Kan as well. By Remark \ref{r:XKan} there is a weak equivalence
	$$  \F(y)_{\nu/} \times_{\F(y)} \Map_\C(x,y)  \lrar  \Un_{*}^{\triangleright}(\F(y)_{\nu/}) \times_{\Un_{*}^{\triangleright}\F(y)} \Un_{*}^{\triangleright}\Map_\C(x,y)  .$$
Furthermore, there is a canonical map 
$$ \Un_{*}^{\triangleright}(\F(y)_{\nu/}) \times_{\Un_{*}^{\triangleright}\F(y)} \Un_{*}^{\triangleright}\Map_\C(x,y) \lrar \mathbb{P}_{\mu,\nu} $$	
induced by the canonical embedding $\Un_{*}^{\triangleright}(\F(y)_{\nu/}) \lrar (\Un_{*}^{\triangleright}\F(y))_{\nu/}$. The above map is also a weak equivalence, since $\Un_{*}^{\triangleright}(\F(y)_{\nu/})$ and $(\Un_{*}^{\triangleright}\F(y))_{\nu/}$ are both contractible. We deduce then by using Lemma \ref{l:mappingUnF}.
\end{proof}

\smallskip

\subsection{Twisted arrow $\infty$-categories of simplicial operads}\label{s:twoperad}

\smallskip

Let $\mathcal{E}$ be an ordinary category. The \textit{(covariant) twisted arrow category} $\Tw(\mathcal{E})$ is by definition the category whose objects are the morphisms of $\mathcal{E}$ and such that maps from $f:x\rar y$ to $f':x'\rar y'$ are given by commutative diagrams of the form 
\begin{equation}\label{eq:twisted}
	\begin{tikzcd}[row sep=2em, column sep =2em]
		x \arrow[d, "f"'] & x' \arrow[l] \arrow[d, "f'"] & \\
		y \arrow[r] & y'
	\end{tikzcd}
\end{equation}
This classical notion was generalized into the $\infty$-categorical framework due to Lurie [\cite{Lurieha}, \S 5.2.1]. For an $\infty$-category $\D$, the twisted arrow $\infty$-category $\Tw(\D)$ is the $\infty$-category whose $n$-simplices are the $(2n+1)$-simplices of {$\D$}.	In particular, objects of $\Tw(\D)$ are the morphisms of $\D$. A map from $f:x \rar y$ to $f':x'\rar y'$ can be also depicted as a diagram of the type \eqref{eq:twisted} commutative up to a chosen homotopy. Furthermore, the twisted arrow $\infty$-category of a fibrant simplicial category $\C$ is simply given by $\Tw(\C) := \Tw(\sN\C)$. As we will see below, $\Tw(\C)$ can be represented as the unstraightening of a certain simplicial copresheaf. Due to this fact, we propose the construction of \textit{twisted arrow $\infty$-categories of (fibrant) simplicial operads}. 

\medskip

We will need the following notations and conventions:

$\bullet$ Let $S$ be a simplicial set. We denote by $(\Set_\Delta^{\cov})_{/S}$ the \textbf{covariant model category of simplicial sets over} $S$ whose cofibrations are monomorphisms and whose fibrant objects are left fibrations over $S$. (See [\cite{Luriehtt}, \S 2.1.4] for more details).

$\bullet$ Let $\C$ be a simplicial category. Recall that the category of $\C$-bimodules is equivalent to the category $\Fun(\C^{\op}\times\C , \Set_\Delta)$. We endow the latter with the projective model structure.

$\bullet$ For the remainder, we consider only the covariant unstraightening construction. So we will simply write $\Un_\C(-)$  for the covariant unstraightening of simplicial functors $\C \lrar \Set_\Delta$. 

\smallskip

When $\C$ is fibrant, according to [\cite{Lurieha}, Proposition 5.2.1.11] the unstraightening functor
\begin{equation}\label{eq:uncat}
	\Un_{\C^{\op}\times\C} : \Fun(\C^{\op}\times\C , \Set_\Delta) \lrarsimeq (\Set_\Delta^{\cov})_{/\sN\C^{\op}\times\sN\C},
\end{equation}
which is a right Quillen equivalence, identifies the functor
$$ \Map_\C : \C^{\op}\times\C \lrar \Set_\Delta \, , \; (x,y) \mapsto \Map_\C(x,y) $$
to $\Tw(\C)$. Alternatively, the unstraightening of $\C$ (viewed as a bimodule over itself) is exactly a model for $\Tw(\C)$. 

Let us now fix $\P$ to be a fibrant simplicial $C$-colored operad. Recall from $\S$\ref{s:infbimod} that there is a canonical isomorphism $\IbMod(\P) \cong \Fun(\textbf{Ib}^{\mathcal{P}},\Set_\Delta)$ between the categories of infinitesimal $\P$-bimodules and simplicial functors $\textbf{Ib}^{\mathcal{P}} \lrar \Set_\Delta$. The unstraightening construction gives a right Quillen equivalence 
\begin{equation}\label{eq:unop}
	\Un_{\textbf{Ib}^{\mathcal{P}}}: \Fun(\textbf{Ib}^{\mathcal{P}},\Set_\Delta) \lrarsimeq (\Set_\Delta^{\cov})_{/\sN(\textbf{Ib}^{\mathcal{P}})}.
\end{equation}
Consider the functor 
$$\P : \textbf{Ib}^{\mathcal{P}} \lrar \Set_\Delta \, , \; (c_1,\cdots,c_m;c) \mapsto \P(c_1,\cdots,c_m;c),$$ which encodes the data of $\P$ as an infinitesimal bimodule over itself. By assumption, we get a left fibration $\Un_{\textbf{Ib}^{\mathcal{P}}}(\P) \lrar \sN(\textbf{Ib}^{\mathcal{P}})$. So $\Un_{\textbf{Ib}^{\mathcal{P}}}(\P)$ is in particular an $\infty$-category. 

\begin{dfn} The \textbf{twisted arrow $\infty$-category of} $\P$ is given by $$\Tw(\P) := \Un_{\textbf{Ib}^{\mathcal{P}}}(\P)$$ the unstraightening of the simplicial functor $\P : \textbf{Ib}^{\mathcal{P}} \lrar \Set_\Delta$ or alternatively, {the} unstraightening of $\P$ regarded as an infinitesimal bimodule over itself.
\end{dfn}

\begin{prop}\label{p:Twisgood} {The construction $\Tw(-)$  determines a functor from fibrant simplicial operads to $\infty$-categories. Moreover, this functor sends Dwyer-Kan equivalences to $\infty$-categorical equivalences.}
\end{prop} 
\begin{proof} Let $f:\P\rar\Q$ be a map between fibrant simplicial operads. By the compatibility of the unstraightening functor with taking base change along $f:\P\rar\Q$, we obtain a Cartesian square of simplicial sets
	$$ \xymatrix{
		\Un_{\textbf{Ib}^{\mathcal{P}}}(f^*\Q)  \ar[r]\ar[d] & \Un_{\textbf{Ib}^{\mathcal{Q}}}(\Q) \ar[d] \\
		\sN(\textbf{Ib}^{\mathcal{P}}) \ar[r] & \sN(\textbf{Ib}^{\mathcal{Q}}) \\
	}$$
in which $f^*\Q$ is considered as an infinitesimal $\P$-bimodule (under $\P$) with the structure induced by $f$. Note that this square is also homotopy Cartesian with respect to the Joyal model structure, due to the fact that the right vertical map is a left fibration. The corresponding map $\Tw(f):\Tw(\P) \lrar\Tw(\Q)$ is then given by the composition
$$ \Un_{\textbf{Ib}^{\mathcal{P}}}(\P) \lrar \Un_{\textbf{Ib}^{\mathcal{P}}}(f^*\Q) \lrar \Un_{\textbf{Ib}^{\mathcal{Q}}}(\Q) .$$

Suppose that $f:\P\rar\Q$ is a {Dwyer-Kan} equivalence. Since $f$ is in particular a levelwise weak equivalence, it implies that the map $\P\lrar f^*\Q$ is a weak equivalence of infinitesimal $\P$-bimodules and thus, the map $\Un_{\textbf{Ib}^{\mathcal{P}}}(\P) \lrar \Un_{\textbf{Ib}^{\mathcal{P}}}(f^*\Q)$ is an equivalence of $\infty$-categories. On other hand, it is not hard to verify that the induced map $\textbf{Ib}^{f} : \textbf{Ib}^{\mathcal{P}}\lrar \textbf{Ib}^{\mathcal{Q}}$ is a {Dwyer-Kan} equivalence between (fibrant) simplicial categories. So the map $\sN(\textbf{Ib}^{\mathcal{P}}) \lrar \sN(\textbf{Ib}^{\mathcal{Q}})$ is an equivalence of $\infty$-categories. Combined with the first paragraph, it implies that the map $\Un_{\textbf{Ib}^{\mathcal{P}}}(f^*\Q)  \lrar \Un_{\textbf{Ib}^{\mathcal{Q}}}(\Q)$ is an equivalence of $\infty$-categories as well. We hence deduce that $\Tw(f):\Tw(\P) \lrar\Tw(\Q)$ is an equivalence.
\end{proof}

Due to the previous subsection, we can describe the simplicial structure of $\Tw(\P)$ as follows.
\begin{cons}\label{cons:TwP} The data of an $n$-simplex of $\Tw(\P)$ consist of

 (i) an $n$-simplex  $\varphi\in\sN(\textbf{Ib}^{\mathcal{P}})$, i.e., a functor $\varphi : \mathfrak{C}[\Delta^{n}]\lrar \textbf{Ib}^{\mathcal{P}}$, and

 (ii) a natural transformation $t:\mathfrak{D}^{\triangleleft}_{\Delta^{n}}\lrar \P \circ \varphi$ between simplicial functors $\mathfrak{C}[\Delta^{n}] \lrar \Set_\Delta$.

\noindent For each map $\delta : \Delta^{m}\lrar\Delta^{n}$, the simplicial structure map $\Tw(\P)_n\lrar\Tw(\P)_m$ is induced by the natural transformation  $\mathfrak{D}^{\triangleleft}_{\delta} : \mathfrak{D}^{\triangleleft}_{\Delta^{m}} \lrar \mathfrak{D}^{\triangleleft}_{\Delta^{n}} \circ \mathfrak{C}[\delta]$, as in Construction \ref{con:Un_CF}.
\end{cons}

\begin{rem}\label{r:objtwp}   Objects of $\Tw(\P)$ are precisely the \textbf{operations} of $\P$ (i.e., the vertices of the spaces of operations of $\P$), while morphisms of $\Tw(\P)$ are given by the \textit{factorizations}. Let $\mu \in \P(c_1, \cdots ,c_m;c)$ and $\nu\in\P(d_1, \cdots ,d_n;d)$ be two operations of $\P$. Denote by $\ovl{c}:=(c_1, \cdots ,c_m;c)$ and by $\ovl{d}:=(d_1, \cdots ,d_n;d)$. Explicitly, the data of a morphism $\mu\rar \nu$ in $\Tw(\P)$ consist of

	$\bullet$ a map $f:\langle n \rangle \lrar \langle m \rangle$ in $\Fin_*$,

	$\bullet$ a tuple of operations  $\alpha = (\alpha_0, \alpha_1,\cdots,\alpha_m) \in \Map^{f}_{\textbf{Ib}^{\P}}(\ovl{c} , \ovl{d})$, and

	$\bullet$ an edge $t : \nu \lrar \alpha^{*}(\mu)$ in $\P(\ovl{d})$ where $\alpha^{*} : \P(\ovl{c}) \lrar \P(\ovl{d})$ is the map determined by the simplicial functor structure of $\P : \textbf{Ib}^{\P} \lrar \Set_\Delta$. 
	
	 It is convenient to depict such a morphism as a diagram of the form
	\begin{equation}\label{eq:1simplex}
		\begin{tikzcd}[row sep=3.5em, column sep =3em]
			(c_1, \cdots ,c_m) \arrow[d, "\mu"'] & (d_1, \cdots ,d_n) \arrow[l, "(\alpha_1  \, \cdots  \, \alpha_m )"'] \arrow[d, "\nu"] & \\
			(c) \arrow[r, "\alpha_0"'] & (d)
		\end{tikzcd}
	\end{equation}
	which is commutative up to a chosen homotopy. 
\end{rem}

\begin{examples}\label{ex:twistcom&ass} When $\P$ is discrete then $\Tw(\P)$ is isomorphic to the nerve of an ordinary category. In this situation, we will identify $\Tw(\P)$ with the corresponding ordinary category and refer to it as the \textbf{twisted arrow category of $\P$}. For example, it is not hard to show that the twisted arrow category of the \textbf{commutative operad} $\Com$ is equivalent to $\Fin_*^{\op}$. Moreover, the twisted arrow category of the \textbf{associative operad} $\Ass$ is equivalent to the simplex category $\Delta$. We will describe the twisted arrow $\infty$-categories of the little discs operads later.
\end{examples}

\begin{rem}\label{r:mainTcom} By construction, there is a canonical left fibration $\Tw(\P) \lrar \sN(\text{Ib}^\P)$ given on objects by sending each operation $\mu\in\P(\ovl{c})$ to $\ovl{c}$. Let $M$ be an infinitesimal $\P$-bimodule. Then $M$ can be identified with a functor $M : \sN(\text{Ib}^\P) \lrar (\Set_\Delta)_\infty$. Composed with {the above map}, it induces a functor $M^\star : \Tw(\P) \lrar (\Set_\Delta)_\infty$. Now, since $\Tw(\P)$ is by definition the unstraightening of $\P$ considered as an infinitesimal bimodule over itself, we get a canonical weak equivalence of spaces:
	$$ 	\lim_{\Tw(\P)}M^\star \lrarsimeq \Map^{\h}_{\IbMod(\P)}(\P,M).$$
\end{rem}

{Construction \ref{cons:TwP} provides us with the simplicial structure of $\Tw(\P)$. We will now delve into the $\infty$-categorical structure by describing the mapping spaces in $\Tw(\P)$.}  

Again, let $\mu \in \P(\ovl{c})$ and $\nu\in\P(\ovl{d})$ be two operations of $\P$, regarded as objects of $\Tw(\P)$. We take $\rho_\mu$ to be the composite map 
$$ \Map_{\textbf{Ib}^{\P}}(\ovl{c} , \ovl{d}) \x{\P}{\lrar} \Map_{\Set_\Delta}(\mathcal{P}(\ovl{c}) , \mathcal{P}(\ovl{d})) \x{ev_\mu}{\lrar} \mathcal{P}(\ovl{d}) $$
where $ev_\mu$ is the evaluation at $\mu$. 

\begin{prop}\label{p:mappingtwp} There is a canonical homotopy equivalence 
	$$ \{\nu\} \times^{\h}_{\P(\ovl{d})}\Map_{\textbf{Ib}^{\P}}(\ovl{c} , \ovl{d}) \lrarsimeq \Map_{\Tw(\P)}(\mu , \nu) .$$
\end{prop}
\begin{proof} This immediately follows  from Proposition \ref{p:mappingunCF}.
\end{proof}

\begin{rem}\label{r:mappingcomponent} Note that there is a canonical map of $\infty$-categories $\Tw(\P) \lrar \sN(\Fin_*^{\op})$ factoring through $\sN(\textbf{Ib}^{\mathcal{P}})$. Consider the induced map $\Map_{\Tw(\P)}(\mu , \nu) \lrar \Hom_{\Fin_*^{\op}}(\langle m \rangle , \langle n \rangle)$. For each map $f : \langle n \rangle \lrar \langle m \rangle$ in $\Fin_*$, we denote by $\Map^{f}_{\Tw(\P)}(\mu , \nu)$ the component of $\Map_{\Tw(\P)}(\mu , \nu)$ over $f$, so that we can write 
	$$ \Map_{\Tw(\P)}(\mu , \nu) = \bigsqcup_{\langle n \rangle \x{f}{\lrar} \langle m \rangle} \Map^{f}_{\Tw(\P)}(\mu , \nu) .$$
	By the above proposition, there is a homotopy equivalence 
	$$ \{\nu\} \times^{\h}_{\P(\ovl{d})}\Map^{f}_{\textbf{Ib}^{\P}}(\ovl{c} , \ovl{d}) \lrarsimeq \Map^{f}_{\Tw(\P)}(\mu , \nu)  .$$
\end{rem}

{For further illustration,} we find a class of operads whose twisted arrow $\infty$-categories admit terminal objects.

\begin{dfn}\label{d:unihoconnected} A simplicial operad is said to be \textbf{homotopy unital} if all its spaces of nullary operations are weakly contractible. Furthermore, a simplicial operad is said to be \textbf{unitally homotopy connected} if it is homotopy unital with weakly contractible spaces of 1-ary operations.
\end{dfn}

\begin{prop}\label{l:terminal} Let $\P$ be a fibrant and homotopy unital simplicial operad. Let $d$ be a color of $\P$ and $\mu_d \in \P(d)$ a nullary operation of $\P$. Then $\mu_d$ is a terminal object of $\Tw(\P)$ if and only if for every color $c$, the space $\P(c;d)$ is contractible. Consequently, if $\P$ is fibrant and unitally homotopy connected then $\Tw(\P)$ admits terminal objects being precisely the nullary operations of $\P$. 
\end{prop}
\begin{proof} By definition, $\mu_d$ is a terminal object of $\Tw(\P)$ precisely if for every operation $\mu\in \P$ the mapping space $\Map_{\Tw(\P)}(\mu,\mu_d)$ is contractible. Given any operation $\mu\in \P(c_1,\cdots,c_m ; c)$, it suffices to show that there is a homotopy equivalence of spaces:
	$$ \P(c;d) \simeq \Map_{\Tw(\P)}(\mu,\mu_d) .$$
	By Proposition \ref{p:mappingtwp} there is a  homotopy equivalence 
	$$ \{\mu_d\}\times^{h}_{\P(d)}\Map_{\textbf{Ib}^{\P}}( (c_1,\cdots,c_m ; c) , (d) ) \lrarsimeq  \Map_{\Tw(\P)}(\mu,\mu_d) .$$
	Since $\P(d)$ is contractible, the homotopy pullback is equivalent to $\Map_{\textbf{Ib}^{\P}}( (c_1,\cdots,c_m ; c) , (d) )$. Furthermore, note that $$ \Map_{\textbf{Ib}^{\P}}( (c_1,\cdots,c_m ; c) , (d) ) = \P(c;d) \times \P(c_1) \times \cdots \times \P(c_m) ,$$ which is therefore weakly equivalent to $ \P(c;d)$. Thus we obtain the  desired homotopy equivalence.
\end{proof}

\smallskip

{Several interesting examples would emerge upon examination of the case of the little discs operads. Let us recall the definition of these. Let $\{U_i\}_{i\in I}$ be a finite collection of open subsets of $\RR^n$ and let $V\subseteq\RR^n$ be another open subset. We will say that a map $\underset{i\in I}{\bigsqcup}U_i \lrar V$ is a \textbf{rectilinear embedding} if it is an open embedding and such that each component map $U_i \lrar V$ is a rectilinear embedding in the usual sense, i.e. it is given by a formula of the form $x \mapsto \lambda x +c$ for $\lambda > 0$.}  We now denote by $$\sD^n:=\{x\in\RR^n \, | \, \lVert x \rVert <1 \}$$ the \textbf{open (unit) $n$-disc}. Then by definition, the \textbf{little $n$-discs operad} $\rE_n$ is a single-colored topological operad such that
$$ \rE_n(k) = \Rect(\underset{k}{\bigsqcup}\sD^n,\sD^n)$$ 
the space of rectilinear embeddings of $k$ disjoint open $n$-discs in another open $n$-disc, endowed with the topology as a subspace of the mapping space in $\Top$. The operadic composition is given by the composition of rectilinear embeddings. One may obtain the simplicial version of $\rE_n$ by applying the singular functor $\Sing : \Top \lrar \Set_\Delta$. 

\begin{conv} We use the same notation $\rE_n$ to refer to the simplicial version of the little $n$-discs operad. On other hand, by an $\rE_\infty$-\textbf{operad} we will mean a fibrant and $\Sigma$-cofibrant model for the simplicial commutative operad $\Com$.
\end{conv}

\begin{rem} We can show that the functor $\Tw(-)$ commutes with filtered homotopy colimits of simplicial operads. As a consequence, we obtain an $\infty$-categorical filtration of the nerve of $\Fin_*^{\op}$:
$$ \sN(\Delta) \simeq \Tw(\E_1) \lrar \Tw(\E_2) \lrar \cdots \lrar \Tw(\E_\infty) \simeq \sN(\Fin_*^{\op}).$$
The functor $\Tw(\E_1) \lrar \Tw(\E_\infty)$ {can be represented by the simplicial circle in the following sense. We will use the standard model for $\sS^1$ which has only two nondegenerate simplices. Moreover, we consider $\sS^1$ as a finite pointed simplicial set, so that it determines a functor $\sS^1 : \Delta^{\op} \lrar \Fin_*$. This latter {corresponds up to weak equivalence to} $\Tw(\E_1) \lrar \Tw(\E_\infty)$. More concretely, $\sS^1$ is given on objects by sending $[m]$ to $\l m \r$; and such that for each map $f : [m] \lrar [n]$ in $\Delta$, the map $\sS^1(f) : \l n \r \lrar \l m \r$ is given by, for each $1\leq j \leq m$, letting
	$$ \sS^1(f)^{-1}(j)  := \{ i \, | \, f(j-1) < i \leq f(j) \}.$$}  
\end{rem}

In the rest of this subsection, we shall give a description of the twisted arrow $\infty$-category of the operad $\E_n$. For this, we will need to use the following notations and conventions.

$\bullet$ We will {use the \textit{one point compactification}} of $\mathbb{R}^n$ as a model for the $n$-sphere. Thus we will write $\sS^n := \mathbb{R}^n \cup \{\infty\}$, and consider it as a pointed space with base point $``\infty$''. 

$\bullet$ We let $\sD^n_r \subseteq \mathbb{R}^n$ denote the $n$-dimensional open disc of radius $r$ (and with center at $0$). In particular, $\sD^n_1$ is simply written as  $\sD^n$, as mentioned above. 

$\bullet$ We denote by $\mathfrak{D}^n_r :=  \sS^n\setminus\overline{\sD^n_r}$, i.e. the complement of the closure of $\sD^n_r$ in $\sS^n = \mathbb{R}^n \cup \{\infty\}$. For simplicity, we will abbreviate $\mathfrak{D}^n := \mathfrak{D}^n_1$, and consider it as a pointed space with base point $``\infty$''. 

$\bullet$ By convention, a \textbf{rectilinear embedding} $\mathfrak{D}^n \longrightarrow \mathfrak{D}^n$ is a pointed embedding such that the induced map $\mathfrak{D}^n \setminus \{\infty\} \longrightarrow \mathfrak{D}^n \setminus \{\infty\}$ is an embedding of the form $x \mapsto rx$ for some fixed $r \geq 1$. Similarly, a map $\mathfrak{D}^n \longrightarrow \sS^n$ is a \textbf{rectilinear embedding} if it is a pointed embedding such that the induced map $\mathfrak{D}^n \setminus \{\infty\} \longrightarrow \sS^n\setminus \{\infty\} = \mathbb{R}^n$ is an embedding of the form $x \mapsto rx$ for some fixed $r > 0$. We will use the same notation $\delta_r$ to denote these two types of maps. In both cases, $\delta_r$ maps $\mathfrak{D}^n$ onto the subspace $\mathfrak{D}^n_r \subseteq \sS^n$. 

$\bullet$ Moreover, a map $\sD^n \longrightarrow \mathfrak{D}^n$ is a \textbf{rectilinear embedding} if it comes from a rectilinear embedding $\sD^n \longrightarrow \mathfrak{D}^n \setminus \{\infty\}$ between subsets of $\mathbb{R}^n$ in the usual sense. Similarly, a map $\sD^n \longrightarrow \sS^n$ is a \textbf{rectilinear embedding} if it comes from an actual rectilinear embedding $\sD^n \longrightarrow \sS^n\setminus \{\infty\} = \mathbb{R}^n$.

\begin{cons} We construct a topological category, denoted $\Disc_*^+$, as follows. An object of $\Disc_*^+$ is either (1) a finite coproduct of spaces of the form $\underset{k}{\bigsqcup}\sD^n \sqcup \, \mathfrak{D}^n$, with $k\geq0$, and regarded as a pointed space with the base point being given by $\infty \in \mathfrak{D}^n$, or (2) the (pointed) $n$-sphere $\sS^n = \mathbb{R}^n \cup \{\infty\}$. The mapping spaces of $\Disc_*^+$ are given as follows.
	
	$\bullet$ $\Map_{\Disc^+_*}(\underset{k}{\bigsqcup}\sD^n \sqcup \, \mathfrak{D}^n, \underset{l}{\bigsqcup}\sD^n \sqcup \, \mathfrak{D}^n)$ is the space of open embeddings $\underset{k}{\bigsqcup}\sD^n \sqcup \, \mathfrak{D}^n \longrightarrow \underset{l}{\bigsqcup}\sD^n \sqcup \, \mathfrak{D}^n$ that preserve base points and such that the restriction to each component is a rectilinear embedding.
	
	$\bullet$ $\Map_{\Disc^+_*}(\underset{k}{\bigsqcup}\sD^n \sqcup \, \mathfrak{D}^n, \sS^n)$ is defined in the same manner as described above.
	
	$\bullet$ $\Map_{\Disc^+_*}(\sS^n, \sS^n) = \{*\}$. 
	
	$\bullet$ $\Map_{\Disc^+_*}(\sS^n, \underset{k}{\bigsqcup}\sD^n \sqcup \, \mathfrak{D}^n) = \emptyset$. 

\noindent The composition in $\Disc_*^+$ is defined in an evident way. We will let $\Disc_* \subseteq \Disc_*^+$ denote the full subcategory spanned by the collection of objects $\{\underset{k}{\bigsqcup}\sD^n \sqcup \, \mathfrak{D}^n\}_{k\geq0}$.
\end{cons}

Next, we will use the same notations to denote the $\infty$-categories corresponding to $\Disc_*$ and $\Disc_*^+$. Moreover, we denote by
$$ \C := \Disc_*\times_{\Disc^+_*} (\Disc^+_*)_{/\sS^n} $$
the $\infty$-category formed by the pullback of $\Disc_*\lrar\Disc^+_*$ and the projection $(\Disc^+_*)_{/\sS^n} \lrar \Disc^+_*$. In particular, objects of $\C$ are the maps in $\Disc_*^+$ of the form $\underset{k}{\bigsqcup}\sD^n \sqcup \, \mathfrak{D}^n \lrar \sS^n$.

\begin{prop}\label{p:TwEn} There is an equivalence of $\infty$-categories $\psi : \C^{\op} \lrarsimeq \Tw(\rE_n)$.
\end{prop}

The proof will require two technical lemmas. Let us first elaborate more on the data of $\Disc_*^+$. 

\begin{rem}\label{r:elaborate} By construction, a map in $\Disc_*^+$ of the form $\underset{k}{\bigsqcup}\sD^n\sqcup \, \mathfrak{D}^n \lrar \mathfrak{D}^n$ is the choice of a map $\delta_r : \mathfrak{D}^n \lrar \mathfrak{D}^n$ with $r\geq1$, and together with a rectilinear embedding $(f_1,\cdots,f_k) : \underset{k}{\bigsqcup}\sD^n \lrar \sD^n_r\setminus \ovl{\sD^n}$. In visual terms, this simply shows a picture of $k$ disjoint $n$-dimensional (open) discs positioned between the boundaries of $\ovl{\sD^n}$ and $\ovl{\sD^n_r}$. On other hand, the data of a map $\underset{k}{\bigsqcup}\sD^n\sqcup \, \mathfrak{D}^n \lrar \sS^n$ consist of a map $\delta_r : \mathfrak{D}^n \lrar \sS^n$ with $r>0$, and along with a rectilinear embedding $(f_1,\cdots,f_k) : \underset{k}{\bigsqcup}\sD^n \lrar \sD^n_r$. Thus, such a map corresponds to a picture of $k$ disjoint $n$-dimensional discs, all contained within $\sD^n_r$. 
\end{rem}

Using our notations, we will write $(f_1,\cdots,f_k;\delta_r)$ for a map of the first type, and write $[f_1,\cdots,f_k;\delta_r]$  for a typical map $\underset{k}{\bigsqcup}\sD^n\sqcup \, \mathfrak{D}^n \lrar \sS^n$ as mentioned in the above remark.

\begin{lem}\label{l:mappingdisc+} For every $k\geq0$, there is a canonical homotopy equivalence $$\rho : \Map_{\Disc^+_*}(\underset{k}{\bigsqcup}\sD^n \sqcup \, \mathfrak{D}^n, \sS^n) \lrarsimeq \rE_n(k).$$
\begin{proof} Let us denote by $\Map^1_{\Disc^+_*}(\underset{k}{\bigsqcup}\sD^n \sqcup \, \mathfrak{D}^n, \sS^n) \subseteq \Map_{\Disc^+_*}(\underset{k}{\bigsqcup}\sD^n \sqcup \, \mathfrak{D}^n, \sS^n)$ the subspace consisting of those maps of the form $[f_1,\cdots,f_k;\delta_1]$. We have in fact a retraction
\begin{gather*}
\rho^1 : \Map_{\Disc^+_*}(\underset{k}{\bigsqcup}\sD^n \sqcup \, \mathfrak{D}^n, \sS^n) \lrar \Map^1_{\Disc^+_*}(\underset{k}{\bigsqcup}\sD^n \sqcup \, \mathfrak{D}^n, \sS^n)  , \\  [f_1,\cdots,f_k;\delta_r]  \mapsto [\frac{1}{r}f_1,\cdots,\frac{1}{r}f_k;\delta_1].
\end{gather*}	
Moreover, this exhibits the space on the right as a (strong) deformation retract of that on the left via a homotopy defined as follows: 
 $$  h_t : [f_1,\cdots,f_k;\delta_r] \mapsto [\lambda_{t,r}f_1,\cdots,\lambda_{t,r}f_k;\delta_{r\lambda_{t,r}}]$$
where $\lambda_{t,r} := \frac{1}{r}[t + (1-t)r]$ for every $0\leq t \leq 1$. The map $\rho$ is then defined to be the composed map
 $$ \Map_{\Disc^+_*}(\underset{k}{\bigsqcup}\sD^n \sqcup \, \mathfrak{D}^n, \sS^n) \overset{\rho^1}{\lrarsimequ} \Map^1_{\Disc^+_*}(\underset{k}{\bigsqcup}\sD^n \sqcup \, \mathfrak{D}^n, \sS^n) \overset{\cong}{\lrar} \rE_n(k)$$
in which the second map is the obvious homeomorphism given by $[f_1,\cdots,f_k;\delta_1] \mapsto (f_1,\cdots,f_k)$.
\end{proof}
\end{lem}

We will write $\varepsilon : \sD^n \lrar \sS^n$ to denote the actual inclusion. Let $k_0\geq1$ be a given integer. The inclusion $\mathfrak{D}^n \subseteq \sS^n$ allows us to identify $\Map_{\Disc^+_*}(\underset{k_0-1}{\bigsqcup}\sD^n \sqcup \, \mathfrak{D}^n, \mathfrak{D}^n)$ with the subspace 
\begin{equation}\label{eq:preiota}
	\Map^\bullet_{\Disc^+_*}(\underset{k_0}{\bigsqcup}\sD^n \sqcup \, \mathfrak{D}^n, \sS^n) \subseteq \Map_{\Disc^+_*}(\underset{k_0}{\bigsqcup}\sD^n \sqcup \, \mathfrak{D}^n, \sS^n)
\end{equation}
containing those maps of the form $[\varepsilon,f_1,\cdots,f_{k_0-1};\delta_r]$ with $r\geq1$. We now obtain a composed map
\begin{equation}\label{eq:iota}
	\Map_{\Disc^+_*}(\underset{k_0-1}{\bigsqcup}\sD^n \sqcup \, \mathfrak{D}^n, \mathfrak{D}^n) \cong \Map^\bullet_{\Disc^+_*}(\underset{k_0}{\bigsqcup}\sD^n \sqcup \, \mathfrak{D}^n, \sS^n) \lrar \Map_{\Disc^+_*}(\underset{k_0}{\bigsqcup}\sD^n \sqcup \, \mathfrak{D}^n, \sS^n)  \overset{\rho}{\lrarsimequ} \rE_n(k_0).
\end{equation}
Unwinding the definitions, we see that this composed map is given by the rule $$(f_1,\cdots,f_{k_0-1};\delta_r) \mapsto   (\frac{1}{r}\varepsilon,\frac{1}{r}f_1,\cdots,\frac{1}{r}f_{k_0-1}).$$ 

\begin{lem}\label{l:DiscandIbEn} There is a canonical weak equivalence of topological categories $$\iota : (\Disc_*)^{\op} \lrarsimeq \textbf{Ib}^{\rE_n}.$$
\end{lem}
\begin{proof} Since $\rE_n$ has a single color, the set of objects of $\textbf{Ib}^{\rE_n}$ is identified with $\mathbb{N}$. Suppose we are given a map $\alpha:\langle k \rangle \lrar \langle l \rangle$ in $\Fin_*$. Denote by $k_i$ the cardinality of $\alpha^{-1}(i)$ for $i=0,\cdots,l$. Recall by construction that the component of $\Map_{\textbf{Ib}^{\rE_n}}(l,k)$ over $\alpha$ is given by
	$$ \Map_{\textbf{Ib}^{\rE_n}}^\alpha(l,k) = \rE_n(k_0) \times \rE_n(k_1) \times\cdots\times \rE_n(k_l) .$$
	
	We let $\iota$ be defined on objects by sending $\underset{k}{\bigsqcup}\sD^n \sqcup \, \mathfrak{D}^n$ to $k$. Moreover, there is a canonical projection $\Disc_* \lrar \Fin_*$ such that the component of $\Map_{\Disc_*}(\underset{k}{\bigsqcup}\sD^n \sqcup \, \mathfrak{D}^n, \, \underset{l}{\bigsqcup}\sD^n \sqcup \, \mathfrak{D}^n)$ over $\alpha$ is given by
	$$ \Map_{\Disc_*}^\alpha(\underset{k}{\bigsqcup}\sD^n \sqcup \, \mathfrak{D}^n, \, \underset{l}{\bigsqcup}\sD^n \sqcup \, \mathfrak{D}^n) = \Map_{\Disc_*}(\underset{k_0-1}{\bigsqcup}\sD^n \sqcup \, \mathfrak{D}^n, \, \mathfrak{D}^n) \times \rE_n(k_1) \times\cdots\times \rE_n(k_l) . $$ 
	In this way, the structure map $\Map_{\Disc_*}^\alpha(\underset{k}{\bigsqcup}\sD^n \sqcup \, \mathfrak{D}^n, \, \underset{l}{\bigsqcup}\sD^n \sqcup \, \mathfrak{D}^n) \lrar \Map_{\textbf{Ib}^{\rE_n}}^\alpha(l,k)$ of $\iota$ is the map induced by the composed map \eqref{eq:iota}. Verifying the functoriality of $\iota$ is then straightforward. We shall complete the proof after showing the latter map is a weak equivalence. For this, it suffices to prove that the inclusion \eqref{eq:preiota} is one. We will denote by $$\Map^\odot_{\Disc^+_*}(\underset{k_0}{\bigsqcup}\sD^n \sqcup \, \mathfrak{D}^n, \sS^n)  \subseteq \Map^1_{\Disc^+_*}(\underset{k_0}{\bigsqcup}\sD^n \sqcup \, \mathfrak{D}^n, \sS^n)$$ the subspace comprising maps of the form $[\frac{1}{r}\varepsilon,f_1,\cdots,f_{k_0-1};\delta_1]$ with $r\geq1$. Then observe that there is a commutative square of the form
		$$ \xymatrix{
		\Map^\bullet_{\Disc^+_*}(\underset{k_0}{\bigsqcup}\sD^n \sqcup \, \mathfrak{D}^n, \sS^n)  \ar[r]\ar[d]^\cong & \Map_{\Disc^+_*}(\underset{k_0}{\bigsqcup}\sD^n \sqcup \, \mathfrak{D}^n, \sS^n) \ar[d]_{\simeq}^{\rho^1} \\
		\Map^\odot_{\Disc^+_*}(\underset{k_0}{\bigsqcup}\sD^n \sqcup \, \mathfrak{D}^n, \sS^n) \ar[r]_\simeq & \Map^1_{\Disc^+_*}(\underset{k_0}{\bigsqcup}\sD^n \sqcup \, \mathfrak{D}^n, \sS^n) \\
	}$$
in which the left vertical map is defined in the same manner as $\rho^1$, and can be readily verified to be a homeomorphism. Moreover, the bottom inclusion is weakly equivalent to the map between configuration spaces
$$ \Conf(\sD^n\setminus\{0\}, k_0-1 ) \lrar \Conf(\sD^n, k_0),$$
which is in fact a weak equivalence (cf., e.g. \cite{Cohen}). We deduce that the top inclusion is indeed a weak equivalence.
\end{proof}

\begin{rem}\label{r:harmonize} The map $\rho$ and the functor $\iota$ work harmoniously together in the following sense. Let $g\in \Map_{\Disc^+_*}(\underset{l}{\bigsqcup}\sD^n \sqcup \, \mathfrak{D}^n, \sS^n)$ be any map. Then we have a commutative square of the form
	$$ \xymatrix{
		\Map_{\Disc_*}(\underset{k}{\bigsqcup}\sD^n \sqcup \, \mathfrak{D}^n, \, \underset{l}{\bigsqcup}\sD^n \sqcup \, \mathfrak{D}^n)  \ar[r]^{ \;\;\;\;\;\;\;\;\;\;\;\;\;\;\;\;\;\;\;\;\; \iota}\ar[d] & \Map_{\textbf{Ib}^{\rE_n}}(l,k) \ar[d] \\
		\Map_{\Disc^+_*}(\underset{k}{\bigsqcup}\sD^n \sqcup \, \mathfrak{D}^n, \sS^n) \ar[r]_{\;\;\;\;\;\;\;\;\;\;\;\;\;\;\;\;\;\;\;\;\rho} & \rE_n(k) \\
	}$$
in which the left and right vertical maps are respectively induced by $g : \underset{l}{\bigsqcup}\sD^n \sqcup \, \mathfrak{D}^n \lrar \sS^n$ and the operation $\rho(g) \in \rE_n(l)$. This will support the proof below.
\end{rem}

\begin{proof}[\underline{Proof of Proposition \ref{p:TwEn}}] We will use the same notations to denote the topological categories associated to $\mathfrak{C}[\Delta^{m}]$ and $\mathfrak{D}^{\triangleleft}_{\Delta^{m}}$ for $m\geq0$. An $m$-simplex of $\C^{\op}$ is a functor $\widetilde{\varphi} : \mathfrak{C}[\Delta^{m+1}]^{\op} \lrar \Disc^+_*$ such that $\widetilde{\varphi}(0)=\sS^n$ and $\widetilde{\varphi}(j)\in\Disc_*$ for $j=1,\cdots,m+1$. We put $\widetilde{\varphi}(j)=\underset{k_j}{\bigsqcup}\sD^n \sqcup \, \mathfrak{D}^n$. The restriction of $\widetilde{\varphi}$ to $\mathfrak{C}[\Delta^{\{1,\cdots,m+1\}}]^{\op}$ determines a functor $\mathfrak{C}[\Delta^{m}] \lrar (\Disc_*)^{\op}$ taking $j$ to $\underset{k_{j+1}}{\bigsqcup}\sD^n \sqcup \, \mathfrak{D}^n$ for $j=0,\cdots,m$. Composed with $\iota : (\Disc_*)^{\op} \lrar \textbf{Ib}^{\rE_n}$ (see Lemma \ref{l:DiscandIbEn}), it yields a functor 
	$$\varphi : \mathfrak{C}[\Delta^{m}] \lrar \textbf{Ib}^{\rE_n} .$$
	\noindent We now construct a natural transformation $t:\mathfrak{D}^{\triangleleft}_{\Delta^{m}}\lrar \rE_n \circ \, \varphi$ between functors $\mathfrak{C}[\Delta^{m}] \lrar \Top$. For each $j\in\{0,\cdots,m\}$, we need to define a natural map $t(j):\mathfrak{D}^{\triangleleft}_{\Delta^{m}}(j)\lrar \rE_n(k_{j+1})$. This is given by the composed map
	$$ \mathfrak{D}^{\triangleleft}_{\Delta^{m}}(j) \overset{\cong}{\lrar} \Map_{\mathfrak{C}[\Delta^{m+1}]^{\op}}(j+1,0) \lrar \Map_{\Disc^+_*}(\underset{k_{j+1}}{\bigsqcup}\sD^n \sqcup \, \mathfrak{D}^n, \sS^n) \x{\rho}{\lrar} \rE_n(k_{j+1})$$ 
	where the second map is the structure map of $\widetilde{\varphi}$. It can be verified that the assignment $\widetilde{\varphi} \mapsto (\varphi, t)$ determines a functor $\psi : \C^{\op} \lrar \Tw(\rE_n)$. In particular, $\psi$ is given on objects by sending a map $f : \underset{k}{\bigsqcup}\sD^n\sqcup \, \mathfrak{D}^n \lrar \sS^n$ to an operation $\rho(f)\in \rE_n(k)$ as in Lemma \ref{l:mappingdisc+}.
	
	We shall now prove that $\psi$ is an equivalence. Since $\psi$ is surjective on objects, it remains to show that the induced maps between mapping spaces are all weak equivalences. Let $f : \underset{k}{\bigsqcup}\sD^n\sqcup \, \mathfrak{D}^n \lrar \sS^n$ and $g : \underset{l}{\bigsqcup}\sD^n\sqcup \, \mathfrak{D}^n \lrar \sS^n$ be any two objects of $\C$. The map $\Map_{\C}(f,g) \lrar \Map_{\Tw(\rE_n)}(\rho(g), \rho(f))$ is weakly equivalent to the map between homotopy fibers
	$$ \{f\} \times^{\h}_{\Map_{\Disc^+_*}(\underset{k}{\bigsqcup}\sD^n \sqcup \, \mathfrak{D}^n, \sS^n)} \Map_{\Disc_*}(\underset{k}{\bigsqcup}\sD^n \sqcup \, \mathfrak{D}^n, \, \underset{l}{\bigsqcup}\sD^n \sqcup \, \mathfrak{D}^n) \lrar \{\rho(f)\} \times^{\h}_{\rE_n(k)} \Map_{\textbf{Ib}^{\rE_n}}(l,k),$$ 
	which is a weak equivalence due to Lemmas \ref{l:mappingdisc+} and \ref{l:DiscandIbEn}. (See also Remark \ref{r:harmonize}, Proposition \ref{p:mappingtwp}   and [\cite{Luriehtt}, Proposition 5.5.5.12]).
\end{proof}

\smallskip

\subsection{Cotangent complex of simplicial operads}\label{sub:cotantsplcop}

\smallskip

Let $\P$ be a $C$-colored simplicial operad which is fibrant and $\Sigma$-cofibrant. We will show that the cotangent complex $\rL_\P \in \T_\P\Op(\Set_\Delta)$ can be represented as a spectrum valued functor on $\Tw(\P)$.

As introduced previously, $\Tw(\P)$ is the image of $\P$ through the unstraightening functor $$\Un_{\textbf{Ib}^{\mathcal{P}}} : \IbMod(\P)  \lrarsimeq  (\Set_\Delta^{\cov})_{/\sN(\textbf{Ib}^{\mathcal{P}})} .$$ We will abbreviate $\Un_{\textbf{Ib}^{\mathcal{P}}}$ by $\U n$. Recall that there is a canonical left fibration $\Tw(\P) \lrar \sN(\textbf{Ib}^{\mathcal{P}})$. Observe now that $\U n$ induces a right Quillen equivalence $$\U n_{\P//\P} : \IbMod(\P)_{\P//\P} \lrarsimeq  (\Set_\Delta^{\cov})_{\Tw(\P)//\Tw(\P)}$$ where $(\Set_\Delta^{\cov})_{\Tw(\P)//\Tw(\P)}$ denotes the pointed model category associated to $(\Set_\Delta^{\cov})_{/\Tw(\P)}$. So we have a right Quillen equivalence of stabilizations
$$ \U n^{\Sp}_{\P//\P} : \T_\P\IbMod(\P) \lrarsimeq \Sp((\Set_\Delta^{\cov})_{\Tw(\P)//\Tw(\P)}) .$$
Moreover, the straightening functor $\St_{\Tw(\P)} : (\Set_\Delta^{\cov})_{/\Tw(\P)} \lrarsimeq \Fun(\mathfrak{C}[\Tw(\P)] , \Set_\Delta)$ induces a left Quillen equivalence 
$$ \St^{\Sp}_{\Tw(\P)} : \Sp((\Set_\Delta^{\cov})_{\Tw(\P)//\Tw(\P)}) \lrarsimeq \Fun(\mathfrak{C}[\Tw(\P)] , \Spectra) $$
where $\Spectra$ refers to the stable model category of spectra. We thus obtain a sequence of right or left Quillen equivalences 
	$$\T_\P\Op(\Set_\Delta) \x{\U^{ib}_\P}{\lrarsimequ} \T_\P\IbMod(\P) \underset{\simeq}{\x{\U n^{\Sp}_{\P//\P}}{\xrightarrow{\hspace*{1cm}}}} \Sp((\Set_\Delta^{\cov})_{\Tw(\P)//\Tw(\P)})  \underset{\simeq}{\x{\St^{\Sp}_{\Tw(\P)}}{\xrightarrow{\hspace*{1cm}}}} \Fun(\mathfrak{C}[\Tw(\P)] , \Spectra) .$$
Let us describe the derived image of $\rL_\P\in \T_\P\Op(\Set_\Delta)$ through this composed functor. We have seen that $\RR\U^{ib}_\P(\rL_\P) [1] \simeq \widetilde{\rL}_\P$ (cf. Theorem \ref{conclusionCotangentCplx}). More explicitly, $\widetilde{\rL}_\P\in\T_\P\IbMod(\P)$ is the prespectrum with $(\widetilde{\rL}_\P)_{k,k}=\P\circ\sS_C^{k}$ {where $\sS^{k}_C$ is the $C$-collection given by $\sS_C^{k}(c;c)=\sS^{k}$ for every color $c$ and taking the value $\emptyset$ on the other levels.} Furthermore, $\RR\U n^{\Sp}_{\P//\P}$ sends $\widetilde{\rL}_\P$ to the prespectrum $$\U n(\P\circ \sS_C^{\bullet}) \in \Sp((\Set_\Delta^{\cov})_{\Tw(\P)//\Tw(\P)})$$ with $\U n(\P\circ \sS_C^{\bullet})_{k,k} = \U n(\P\circ \sS_C^{k})$. We now denote by 
	$$\F_\P := \LL\St^{\Sp}_{\Tw(\P)}(\U n(\P\circ \sS_C^{\bullet})) \in \Fun(\mathfrak{C}[\Tw(\P)] , \Spectra) .$$
So $\F_\P$ is the derived image of $\rL_\P [1]$ under the composed functor $$\T_\P\Op(\Set_\Delta) \lrar \Fun(\mathfrak{C}[\Tw(\P)] , \Spectra).$$ To get a description of the functor $\F_\P$, observe first that there is an equivalence of $\infty$-categories
$$ \Fun(\mathfrak{C}[\Tw(\P)] , \Spectra)_\infty \lrarsimeq \Fun(\Tw(\P) , \Sp) $$ 
where $\Sp$ is the \textbf{$\infty$-category of spectra}. Thus we will regard $\F_\P$ as an $\infty$-functor $\Tw(\P) \lrar \Sp$. Let $\mu\in \P(c_1,\cdots,c_m;c)$ be an operation of $\P$, regarded as an object of $\Tw(\P)$. By construction, $\F_\P(\mu) \in \Sp$ is the prespectrum with $\F_\P(\mu)_{k,k}$ being given by the value at $\mu$ of the functor
$$\St_{\Tw(\P)}(\U n(\P\circ \sS_C^{k})) : \mathfrak{C}[\Tw(\P)] \lrar \Set_\Delta.$$
By using [\cite{Luriehtt}, Proposition 2.2.3.15] (with the opposite convention), we obtain that
$$ \F_\P(\mu)^{\op}_{k,k} \simeq \U n(\P\circ \sS_C^{k}) \times_{\U n(\P)} \{\mu\} .$$
It follows by construction that $(\P\circ \sS_C^{k}) (c_1,\cdots,c_m;c) = \P(c_1,\cdots,c_m;c) \times (\sS^{k})^{\times m}$.  
Using Remark \ref{r:fiberkan}, we can then show that $\F_\P(\mu)_{k,k} \simeq (\sS^{k})^{\times m}$ and hence, $\F_\P(\mu) \simeq \mathbb{S}^{\oplus m}$ the $m$-fold coproduct of the \textbf{sphere spectrum}. 

\begin{rem}\label{r:FEinfty} We consider the operad $\rE_\infty$ and describe the functor $\F_{\E_\infty}$. Since $\Tw(\E_\infty) \simeq \sN(\Fin_*^{\op})$, we may regard $\F_{\E_\infty}$ as a functor $\Fin_*^{\op} \lrar \Spectra$ given on objects by sending $\l m \r$ to $\mathbb{S}^{\oplus m}$. Let $f : \l n \r \lrar \l m \r$ be map in $\Fin_*$. Unwinding the definitions, the structure map 
	$$ \F_{\E_\infty}(f) : \mathbb{S}^{\oplus \{1,\cdots,m\}} \lrar \mathbb{S}^{\oplus \{1,\cdots,n\}}  $$ 
is given by, for each $i\in \{1,\cdots,m\}$, copying the $i$'th summand to the summands of position $j\in f^{-1}(i)$ when this fiber is nonempty or collapsing that summand to the zero spectrum otherwise. 
\end{rem}

{Here is another significant remark.}

\begin{rem} {Without loss of generality, we can assume that $\P$ comes equipped with a map $\varPhi : \P \lrar \E_\infty$. We then claim that the (derived) functor $\varPhi^* : \T_{\E_\infty}\Op(\Set_\Delta) \lrar \T_\P\Op(\Set_\Delta)$ sends $\rL_{\E_\infty}$ to $\rL_\P$. Indeed, by using Theorem \ref{conclusionCotangentCplx}, we just need to verify that the induced right adjoint $\varPhi^* : \T_{\E_\infty}\IbMod(\E_\infty) \lrar \T_\P\IbMod(\P)$ sends the prespectrum $\widetilde{\rL}_{\E_\infty}$ to $\widetilde{\rL}_\P$ (see also [\cite{YonatanBundle}, Corollary 2.4.8]). For this, we need to show that the functor $$\varPhi^* : \IbMod(\E_\infty)_{\E_\infty//\E_\infty} \lrar \IbMod(\P)_{\P//\P}$$ sends $\E_\infty\circ\sS_*^{k}$ to $\P\circ\sS_C^{k}$ for every $k\geq0$. (Here we write $\sS_*^{k}$, reflecting that the operad $\E_\infty$ has a single color).  This follows from the isomorphisms of the form $$\P(c_1,\cdots,c_m;c) \times (\sS^{k})^{\times m} \; \cong \; \P(c_1,\cdots,c_m;c) \times_{\E_\infty(m)} (\E_\infty(m)\times (\sS^{k})^{\times m}).$$ 
Accordingly, we deduce that the functor $\F_\P$ can be represented as the composed functor   
		\begin{equation}\label{eq:F_P}
			\Tw(\P) \lrar \sN(\Fin_*^{\op}) \x{\F_{\E_\infty}}{\lrar} \Sp.
	\end{equation}}
\end{rem}

\smallskip

We summarize the above steps in the following statement.
\begin{thm}\label{c:main} Let $\P$ be a fibrant and $\Sigma$-cofibrant simplicial operad. There is an equivalence of $\infty$-categories
	$$ \T_\P\Op(\Set_\Delta)_\infty \lrarsimeq \Fun(\Tw(\P) , \Sp) .$$
	Moreover, under this equivalence, the cotangent complex $\rL_\P$ is identified {with} $\F_\P[-1]$ the desuspension of the functor $ \F_\P : \Tw(\P)  \lrar \Sp$ \eqref{eq:F_P}, which is given on objects by sending each operation $\mu\in\P$ of arity $m$  to $\F_\P(\mu) = \mathbb{S}^{\oplus m}$. 
\end{thm}

\begin{cor}\label{cor:main} For a given functor $\F : \Tw(\P) \lrar \Sp$, the $n$'th Quillen cohomology group of $\P$ with coefficients in $\F$ is computed by the formula
	$$ \sH^{n}_Q (\P ; \F) = \pi_{0}\Map_{\Fun(\Tw(\P) , \Sp)} (\F_\P , \F[n+1]) .$$
\end{cor}

We are particularly concerned about the cotangent complex of the little $n$-discs operad $\rE_n$. Let $\mu_0 \in \E_n(0)$ denote the unique nullary operation of $\E_n$ and let $\eta : \{*\}  \lrar \Tw(\E_n)$ be the functor classified by $\mu_0$. Moreover, we write $$\eta_! : \Fun(\{*\},\Sp) \lrar \Fun(\Tw(\E_n),\Sp)$$ to denote the left adjoint functor induced by $\eta$. Let $\textbf{S} := (\Set_\Delta)_\infty$ denote the $\infty$-category of spaces. We also write $\eta_! : \Fun(\{*\},\textbf{S}) \lrar \Fun(\Tw(\E_n),\textbf{S})$ for the functor induced by $\eta$.

We wish to describe the functor $\eta_!(\mathbb{S}) : \Tw(\E_n) \lrar \Sp$ where we regard $\mathbb{S}$ as a functor $\{*\} \lrar \Sp$ classified by the sphere spectrum $\mathbb{S}\in\Sp$.  Observe that $\eta_!(\mathbb{S})$ is equivalent to the image of $\{*\} \x{\Delta^0}{\lrar}\textbf{S}$ through the composed functor
$$ \Fun(\{*\},\textbf{S}) \x{\eta_!}{\lrar} \Fun(\Tw(\E_n),\textbf{S})  \lrar \Fun(\Tw(\E_n),\Sp)$$ 
in which the second functor is given by the post-composition with $\Sigma^{\infty}_+ : \textbf{S} \lrar \Sp$.  We first describe the functor $\eta_!(\Delta^0) : \Tw(\E_n) \lrar \textbf{S}$. Let $\mu$ be an operation of $\E_n$ regarded as an object of $\Tw(\E_n)$. Observe that $\eta_!(\Delta^0)$ sends $\mu$ to $\Map_{\Tw(\E_n)}(\mu_0,\mu) \in \textbf{S}$. Suppose further that $\mu \in \E_n(k)$. Using Proposition \ref{p:mappingtwp}, we may show that $\Map_{\Tw(\E_n)}(\mu_0,\mu)$ is  equivalent to the fiber over $\mu$ of the map
$$ (-) \circ_{1}\mu_0 : \E_n(k+1) \lrar \E_n(k) .$$
A classical result asserts that this fiber is equivalent to the $k$-fold wedge sum of the $(n-1)$-sphere {(cf., e.g. \cite{Cohen})}. So we get that 
$$\Map_{\Tw(\E_n)}(\mu_0,\mu) \simeq \underset{k}{\bigvee} \sS^{n-1}  .$$
Moreover, we have $\Sigma^{\infty}_+(\underset{k}{\bigvee} \sS^{n-1}) \simeq \left(\underset{k}{\bigoplus} \, \mathbb{S}[n-1] \right )\oplus\mathbb{S}$. We thus obtain that $\eta_!(\mathbb{S}) : \Tw(\E_n) \lrar \Sp$ sends each operation of arity $k$ to $\left(\underset{k}{\bigoplus} \, \mathbb{S}[n-1] \right )\oplus\mathbb{S}$. 

{Next, we write $\id \in \E_n(1)$ for the identity operation, and write $\gamma : \{*\}  \lrar \Tw(\E_n)$ for the functor classified by $\id$. Moreover, we let $$\gamma_* : \Fun(\{*\},\Sp) \lrar \Fun(\Tw(\E_n),\Sp)$$
denote the right adjoint to the restriction functor along the functor $\gamma$. We also regard the functor  $\gamma_*(\mathbb{S}) : \Tw(\E_n) \lrar \Sp$. Unwinding the definitions, this is given on objects by sending each operation $\mu \in \E_n(k)$ to
$$\gamma_*(\mathbb{S})(\mu) \, \simeq \, [\Map_{\Tw(\E_n)}(\mu,\id), \mathbb{S}] \, \simeq \, [\E_n(2) \sqcup \; \underline{k} \, , \mathbb{S}] \, \simeq \, [\sS^{n-1} \sqcup \; \underline{k} \, , \mathbb{S}] \, \simeq \, \left(\mathbb{S}\oplus\mathbb{S}[-n+1]\right) \oplus \mathbb{S}^{\oplus k}$$
where $[-,-]$ signifies the powering of simplicial sets over spectra and $\underline{k}$ represents a discrete space of cardinality $k$. (For the computation of $\Map_{\Tw(\E_n)}(\mu,\id)$, we make use of Proposition \ref{p:mappingtwp} again).}

\begin{rem}\label{r:twocofib} {We may regard $\F_{\E_n}$ as a sub-functor of  $\gamma_*(\mathbb{S})$ via the embedding $\F_{\E_n} \lrar \gamma_*(\mathbb{S})$ classified by the identity $\mathbb{S} = \F_{\E_n}(\id) \lrar \mathbb{S}$. More precisely, there is a canonical cofiber sequence in $\Fun(\Tw(\E_n),\Sp)$ of the form
\begin{equation}\label{eq:twocofib1}
	\F_{\E_n}  \lrar \gamma_*(\mathbb{S}) \lrar \ovl{\mathbb{S}} \oplus \ovl{\mathbb{S}}[-n+1]
\end{equation}
where $\ovl{\mathbb{S}} : \Tw(\E_n) \lrar \Sp$ refers to the constant functor on $\mathbb{S} \in \Sp$, and the second map is given at each operation $\mu\in\E_n$ by the projection onto the component $[\sS^{n-1} , \mathbb{S}] \simeq \mathbb{S}\oplus\mathbb{S}[-n+1]$ of $\gamma_*(\mathbb{S})(\mu)$. Additionally, we have another cofiber sequence in $\Fun(\Tw(\E_n),\Sp)$ of the form
\begin{equation}\label{eq:twocofib2}
\eta_!(\mathbb{S}[-n+1]) \lrar \gamma_*(\mathbb{S}) \lrar \ovl{\mathbb{S}}
\end{equation}
where the first map is the embedding classified by the inclusion $$\mathbb{S}[-n+1] \lrar \mathbb{S} \oplus \mathbb{S}[-n+1] \simeq \gamma_*(\mathbb{S})(\mu_0)$$ (here we note that $\eta_!(\mathbb{S}[-n+1])$ sends each operation of arity $k$ to $\mathbb{S}[-n+1] \oplus \mathbb{S}^{\oplus k}$), and while the second map is given at each object by the projection onto the summand $``\mathbb{S}$'' contained in  $[\sS^{n-1} , \mathbb{S}]$.}      
\end{rem}

\begin{thm}\label{p:cotantEn} There is a canonical cofiber sequence in $\Fun(\Tw(\E_n),\Sp)$ of the form
	\begin{equation}\label{eq:cotantEn}
\eta_!(\mathbb{S}) \lrar \ovl{\mathbb{S}} \lrar \F_{\E_n}[n]
	\end{equation}
where the first map is classified by the identity on $\mathbb{S}$.
\begin{proof} {The cofiber sequence \eqref{eq:twocofib2} gives rise to a coCartesian square in $\Fun(\Tw(\E_n),\Sp)$ of the form
		$$ \xymatrix{
		\eta_!(\mathbb{S}[-n+1])  \ar[r]\ar[d] & \gamma_*(\mathbb{S}) \ar[d] \\
		\ovl{\mathbb{S}}[-n+1] \ar[r] & \ovl{\mathbb{S}} \oplus \ovl{\mathbb{S}}[-n+1] \\
	}$$
in which the left vertical map is induced by the identity on $\mathbb{S}[-n+1]$. We can then deduce by combining this with the cofiber sequence \eqref{eq:twocofib1}.}
\end{proof}
\end{thm}

\begin{cor}\label{c:cotantEn} Let $\F : \Tw(\E_n) \lrar \Sp$ be a given functor. There is a long exact sequence of abelian groups of the form
	$$ \cdots \lrar \sH^{-k-n-1}_Q(\E_n ; \F) \lrar \pi_k\lim\F \lrar \pi_k\F(\mu_0) \lrar \sH^{-k-n}_Q(\E_n ; \F) \lrar \pi_{k-1}\lim\F \lrar \cdots .$$
In particular, when $\F(\mu_0)=0$ then for every $k\in\ZZ$, we have a canonical isomorphism
$$ \sH^{-k-n}_Q(\E_n ; \F) \lrarsimeq \pi_{k-1}\lim\F .$$
\end{cor}

{These are explored more extensively in \cite{Hoang1}, where we further investigate the cotangent and \textbf{Hochschild complexes} of algebras over an operad. Notably, the Quillen principle for $\E_n$-spaces is established there.}

For more illustration, we shall now relate Quillen cohomology of {the operad} $\E_\infty$ to {the} \textbf{stable cohomotopy of right $\Gamma$-modules} (cf. \cite{Pirashvili}). Let $\textbf{k}$ be a commutative ring. By definition, a \textbf{right $\Gamma$-module} is a functor $\Fin_*^{\op} \lrar \Mod_\textbf{k}$. One particularly regards the right $\Gamma$-module $t : \Fin_*^{\op} \lrar \Mod_\textbf{k}$ given by taking $\l m \r$ to $t\l m \r=[\l m \r,\textbf{k}]$ the $\textbf{k}$-module of based maps from $\l m \r$ to $\textbf{k}$ (where $\textbf{k}$ has base point $0_\textbf{k}$). 

\begin{rem}\label{r:tandfcinfty}
	We can show that $t\l m \r \cong \textbf{k}^{\oplus \, m}$. Moreover, if $f : \l m \r \lrar \l n \r$ is a map in $\Fin_*$ then the structure map $ t(f) :   \textbf{k}^{\oplus \, n} \lrar \textbf{k}^{\oplus \, m} $ is given by, for each $i\in \{1,\cdots,n\}$, copying the $i$'th summand to the summands of position $j\in f^{-1}(i)$ when this fiber is nonempty or collapsing that summand to the zero module otherwise. 
\end{rem}

We now fix $T : \Fin_*^{\op} \lrar \Mod_\textbf{k}$ to be a right $\Gamma$-module. According to \cite{Pirashvili}, the \textbf{$k$'th stable cohomotopy group} of $T$ is computed by the formula 
$$ \pi^{k}T \cong \Ext_{\Gamma}^{k}(t,T) \cong \pi_0\Map^{\h}_{\Fun(\Fin_*^{\op},\C(\textbf{k}))}(t,T[k]) $$
where $\C(\textbf{k})$ denotes the category of dg $\textbf{k}$-modules. Denote by $\ovl{T}$ the composed functor
$$ \Fin_*^{\op} \x{T}{\lrar}  \Mod_\textbf{k} \lrar \Sp(\sMod_\textbf{k}) $$
in which the second functor is defined by taking $M\in\Mod_\textbf{k}$ to the suspension spectrum on $M$. Furthermore, consider the adjunction 
$$ \Sp(\mathbb{F}) : \adjunction*{}{\Spectra}{\Sp(\sMod_\textbf{k})}{} : \Sp(\mathbb{U})$$
induced by the free-forgetful adjunction $\mathbb{F} : \adjunction*{}{(\Set_\Delta)_*}{\sMod_\textbf{k} }{} : \mathbb{U}$ between pointed simplicial sets and simplicial $\textbf{k}$-modules where a simplicial $\textbf{k}$-module is considered as a pointed simplicial set with base point $0$. We denote by $\widetilde{T}$ the composed functor 
$$\Fin_*^{\op} \x{\ovl{T}}{\lrar} \Sp(\sMod_\textbf{k}) \x{\Sp(\mathbb{U})}{\xrightarrow{\hspace*{0.8cm}}} \Spectra .$$

\begin{prop}\label{p:stablecohomotopy} The $(k-1)$'th Quillen cohomology group of the operad $\E_\infty$ with coefficients in $\widetilde{T}$ is isomorphic to the $k$'th stable cohomotopy group of $T$:
	$$ \sH^{k-1}_Q (\E_\infty ; \widetilde{T}) \cong \pi^{k}T .$$
	
\end{prop}
\begin{proof} By Corollary \ref{cor:main}, we just need to prove the existence of a weak equivalence 
	$$ \Map^{\h}_{\Fun(\Fin_*^{\op} , \, \Spectra)} (\F_{\E_\infty} , \widetilde{T}[k]) \simeq \Map^{\h}_{\Fun(\Fin_*^{\op},\C(\textbf{k}))}(t,T[k]) .$$
	Note first that the right hand side can be replaced equivalently by $\Map^{\h}_{\Fun(\Fin_*^{\op},\C_{\geqslant0}(\textbf{k}))}(t,T[k])$ where $\C_{\geqslant0}(\textbf{k})$ denotes the category of connective dg $\textbf{k}$-modules. On the other hand, by the Quillen adjunction $\Sp(\mathbb{F})\dashv \Sp(\mathbb{U})$, we have a weak equivalence
	$$ \Map^{\h}_{\Fun(\Fin_*^{\op} , \, \Spectra)} (\F_{\E_\infty} , \widetilde{T}[k]) \simeq \Map^{\h}_{\Fun(\Fin_*^{\op} , \, \Sp(\sMod_\textbf{k}))}(\Sp(\mathbb{F})\circ \F_{\E_\infty} , \ovl{T}[k] ) .$$
	Here we note that $\Sp(\mathbb{F})\circ\F_{\E_\infty}$ has already the right type, by Observation \ref{ob:stablecohomotopy}. It therefore suffices to prove the existence of a weak equivalence of the form
	\begin{equation}\label{eq:stablecohomotopy}
		\Map^{\h}_{\Fun(\Fin_*^{\op},\C_{\geqslant0}(\textbf{k}))}(t,T[k]) \simeq  \Map^{\h}_{\Fun(\Fin_*^{\op} , \, \Sp(\sMod_\textbf{k}))}(\Sp(\mathbb{F})\circ \F_{\E_\infty} , \ovl{T}[k] ).
	\end{equation}
	Consider the composed adjunction 
	$$ \adjunction*{}{\C_{\geqslant0}(\textbf{k})}{\Sp(\C_{\geqslant0}(\textbf{k})) }{} \adjunction*{}{}{\Sp(\sMod_\textbf{k})}{} $$ 
	in which the second adjunction is induced by the Dold-Kan correspondence. Observe now that the functor $\Sp(\mathbb{F})\circ \F_{\E_\infty}$ is weakly equivalent to the composed functor $$\Fin_*^{\op} \x{t}{\lrar} \C_{\geqslant0}(\textbf{k}) \lrar \Sp(\sMod_\textbf{k}).$$ On the other hand, it {is} not hard to show that the derived image of $\ovl{T}[k]$ in $\Fun(\Fin_*^{\op},\C_{\geqslant0}(\textbf{k}))$ is weakly equivalent to $T[k]$. Thus, by adjunction, we deduce the existence of the weak equivalence \eqref{eq:stablecohomotopy}.
\end{proof}

\newpage

\bibliographystyle{amsplain}

\begin{thebibliography}{99} 
	
		\bibitem{Turchin} G. Arone and V. Turchin, {\it On the rational homology of high-dimensional analogues of spaces of long knots}, Geom. Topol. 18(3) (2014) 1261-1322.
		
		\bibitem{Barwick} C. Barwick, {\it On left and right model categories and left and right Bousfield localizations}, Homol. Homotopy Appl. 12(2) (2010) 245-320.	
		
			\bibitem{Baber} M. Batanin and C. Berger, {\it Homotopy theory for algebras over polynomial monads}, Theory Appl. Categ. 32(6) (2017) 148-253. 
		
			\bibitem{White} M. Batanin and D. White, {\it Left Bousfield localization without left properness}, J. Pure
			Appl. Algebra (2024), DOI: 10.1016/j.jpaa.2023.107570.
			
			
		    \bibitem{BauWi} H. J. Baues and G. Wirsching, {\it Cohomology of small categories}, J. Pure
		    Appl. Algebra. 38 (1985) 187-211.	
		
		
	
		
		\bibitem{BergerMoerdijk} C. Berger and I. Moerdijk, {\it On the derived category of an algebra over an operad},  Georgian Math. J. 16(1) (2009) 13-28.
		
		\bibitem{Ieke} C. Berger and I. Moerdijk,  \textit{On the homotopy theory of enriched categories}, Q. J. Math. 64(3) (2013) 805-846.
		
				\bibitem{Caviglia} G. Caviglia, {\it A model structure for enriched coloured operads}, available at author's homepage \url{https://www.math.ru.nl/~gcaviglia/} (2015).
		
		
		\bibitem{Cisinski} D.C. Cisinski and I. Moerdijk, {\it Dendroidal sets and simplicial operads}, J. Topol. 6(3) (2013) 705-756.
			
			
					\bibitem{Cohen} F. R. Cohen, {\it Introduction to configuration spaces and their applications}, Lect. Notes Ser. Inst. Math.
			Sci. Natl. Univ. Singap. 19, World Sci. Publ., Hackensack, NJ (2010).
		
			
			
				\bibitem{Julien} J. Ducoulombier, B. Fresse and V. Turchin, {\it Projective and Reedy model category structures for (infinitesimal) bimodules over and operad}, Appl. Categ. Struct. 30 (2022) 825-920.
				
				
				\bibitem{DuVic} J. Ducoulombier and V. Turchin, {\it Delooping the functor calculus tower}, Proc. Lond. Math. Soc. 124(6)  (2022) 772-853.
				
				
					\bibitem{Iandundas} B. I. Dundas, {\it Localization of $V$-categories}, Theory Appl. Categ. 8 (2001) 284-312.
			
				\bibitem{Hess} W. G. Dwyer and K. Hess, {\it Long knots and maps between operads}, Geom. Topol. 16(2) (2012) 919-955.
				
		 	\bibitem{DK} W.G. Dwyer and D.M. Kan, {\it Function complexes in homotopical algebra}, Topology 19(4) (1980) 427-440.
			
				\bibitem{DKSmith} W. G. Dwyer, D. M. Kan, and J. H. Smith, {\it An obstruction theory for simplicial categories}, Indag. Math. 89(2) (1986) 153-161.
			
			
		\bibitem{Francis} J. Francis, {\it The tangent complex and Hochschild cohomology of $\sE_n$-rings}, Compos. Math. 149(3) (2013) 430-480.
			
			
				\bibitem{Fresse1} B. Fresse, {\it Modules over operads and functors}, Lecture Notes in Mathematics. 1967, Springer-Verlag, Berlin (2009).
			
			\bibitem{Fresse2} B. Fresse, {\it Props in model categories and homotopy invariance of structures}, Georgian Math. J. 17(1) (2010) 79-160.
			
			
			{\bibitem{Goerss} P. G. Goerss and J. F. Jardine, {\it Simplicial homotopy theory}, Progress in Mathematics. 174, Birkhäuser Verlag, Basel (1999).}
			

			
			
				\bibitem{Guti}  J. J. Gutiérrez and R. M. Vogt, {\it A model structure for coloured operads in symmetric spectra}, Math. Z. 270 (2012) 223-239.
				
			
						\bibitem{YonatanCotangent} Y. Harpaz, J. Nuiten and M. Prasma, {\it The abstract cotangent complex and Quillen cohomology of enriched categories}, J. Topol. 11(3) (2018) 752-798.
			
			
					\bibitem{Yonatan} Y. Harpaz, J. Nuiten and M. Prasma, {\it  Tangent categories of algebras over operads}, Isr. J. Math. 234(2) (2019) 691-742.
			
			
				\bibitem{YonatanBundle} Y. Harpaz, J. Nuiten and M. Prasma, {\it The tangent bundle of a model category}, Theory Appl.
				Categ. 34 (2019) 1039-1072.
			

			
				\bibitem{Hinich} V. Hinich, {\it Dwyer-Kan localization revisited}, Homol. Homotopy Appl. 18(1) (2016) 27-48.
				
				
					\bibitem{Hirschhorn2} P. S. Hirschhorn, {\it Overcategories and undercategories of model categories}, preprint, available at \url{http://www-math.mit.edu/ psh/undercat.pdf} (2005).
					
					
				{\bibitem{Hoang} T. Hoang, {\it Operadic Dold-Kan correspondence and mapping spaces between enriched operads}, preprint,	arXiv:2312.07906 (2023).}
				
				{\bibitem{Hoang1} T. Hoang, {\it Hochschild and cotangent complexes of operadic algebras},  to appear in Trans. Amer. Math. Soc.}
			
			
				\bibitem{Hovey} M. Hovey, {\it Model categories}, Mathematical Surveys and Monographs. 63, Amer. Math. Soc. (1999).
				
				

			
			
				\bibitem{JY} M.W. Johnson and D. Yau, {\it On homotopy invariance for algebras over colored PROPs}, J. Homotopy Relat. Struct. 4(1) (2009) 275-315.
			
			
					\bibitem{Ko} M. Kontsevich, {\it Operads and motives in deformation quantization}, Lett. Math.
			Phys. 48 (1999) 35-72.
			
			
				\bibitem{Lawson} T. Lawson, {\it Localization of enriched categories and cubical sets}, Theory Appl. Categ. 32(35) (2017) 1213-1221.
			
			
			    \bibitem{Lod}  J. L. Loday, {\it Opérations sur l'homologie cyclique des algèbres commutatives}, Invent. Math. 96(1) (1989) 205-230.
			
				\bibitem{Loday} J.-L. Loday and B. Vallette, {\it Algebraic Operads}, Grundlehren der mathematischen Wissenschaften. 346, Springer-Verlag (2012).
			
			

			
			\bibitem{Luriehtt} J. Lurie, {\it Higher topos theory}, Annals of Mathematics Studies. 170, Princeton University Press (2009).
			
			\bibitem{Luriefmp} J. Lurie, {\it Derived Algebraic Geometry X: Formal moduli problems}, available at \url{https://www.math.ias.edu/~lurie/} (2011).
			
			
				\bibitem{Lurieha} J. Lurie, {\it Higher algebra}, preprint, available at \url{https://www.math.ias.edu/~lurie/} (2011).
			
			
				\bibitem{ClureSmith} J. E. McClure and J. H. Smith, {\it A solution of Deligne’s Hochschild cohomology conjecture}, In Recent progress in homotopy theory (Baltimore, MD, 2000), Contemp. Math. 293 (2002) 153-193.
				
			
			
				\bibitem{Vallette} S. Merkulov and B. Vallette, {\it Deformation theory of representations of prop(erad)s. II}, J. Reine Angew. Math. 636 (2009) 1-52.
				
					\bibitem{Muro} F. Muro, {\it Dwyer-Kan homotopy theory of enriched categories}, J. Topol. 8(2)  (2015) 377-413.
				
				\bibitem{Pavlov} D. Pavlov and J. Scholbach, {\it Admissibility and rectification of colored symmetric operads}, J. Topol. 11(3) (2018)  559-601.
				
				\bibitem{Pavlov1} D. Pavlov and J. Scholbach, {\it Homotopy theory of symmetric powers}, Homol. Homotopy Appl. 20(1) (2018) 359-397.
				
					\bibitem{Pirashvili} T. Pirashvili, {\it Dold-Kan type theorems for $\Gamma$-groups}, Math. Ann. 318(2) (2000) 277-298.
					
					
					    \bibitem{Pirash}  T. Pirashvili, {\it Hodge decomposition for higher order Hochschild homology}, Ann. Sci. Ecole Norm. Sup. 33(2) (2000) 151-179.
				
				
					\bibitem{Pridham} J. Pridham, {\it Unifying derived deformation theories}, 	Adv. Math. 224(3)  (2010) 772-826.
			
			
				\bibitem{Quillen} D. Quillen, {\it Homotopical algebra}, Lecture Notes in Mathematics. 43, Springer-Verlag (1967).
			
			\bibitem{Quillen2} D. Quillen, {\it On the (co-)homology of commutative rings}, Applications of Categorical Algebra,
			Proc. Sympos. Pure Math. 17, New York 1968, Amer. Math. Soc., Providence, R.I. (1970) 65-87.
			
			
				\bibitem{Rezk} C. Rezk, {\it Spaces of algebra structures and cohomology of operads}, PhD Thesis, Massachusetts Institute of Technology (1996).
			
			\bibitem{Rezk1} C. Rezk, {\it Stuff about quasicategories}, available at \url{https://faculty.math.illinois.edu/~rezk/quasicats.pdf} (2019).
			
			
				\bibitem{Robinson} A. Robinson, {\it $E_\infty$-obstruction theory}, Homol. Homotopy Appl. 20(1) (2018) 155-184.
				
				
					{\bibitem{Schwede} S. Schwede and B. Shipley, {\it Equivalences of monoidal model categories}, Algebr. Geom. Topol. 3(1) (2003) 287-334.}
			
				\bibitem{Spitzweck} M. Spitzweck, {\it Operads, algebras and modules in general model categories}, preprint, arXiv:math/0101102 (2001).
			
			
			
			\bibitem{Tabuada1} G. Tabuada, {\it Postnikov towers, k-invariants and obstruction theory for dg categories}, 	J. Algebra. 321(12) (2009) 3850-3877.
			
			
				\bibitem{Toen} B. Toën, {\it The homotopy theory of dg-categories and derived Morita theory}, Invent. Math. 167(3) (2007) 615-667.
	

	
		\bibitem{Yau} D. Yau, {\it Colored Operads}, Graduate Studies in Math. 170, Amer. Math. Soc. (2016).

	
	
		

	

	
	

	
	

	
	
	

	
	
	

	

	

	

	

	

	

	

	

	

	

	
	
	

	
	

	
	

	
	

	

		
		

	
	


    
    

    
    
    

	
	
	
	

	

	

	

	

	

	

	
	

	

	

	

	
	

	
	

	
	
	

	

	

	

	
	
	

	
	

			
			

	
	

	

	
	

	

	
	
\end{thebibliography}

\end{document}